\documentclass{amsart}

\usepackage{amssymb,amsmath,esint}
\usepackage[all]{xy}
\usepackage{graphicx}
\usepackage{enumerate}

\usepackage{colortbl}
\usepackage{cite}
\usepackage{xcolor}

\usepackage{amsrefs}

\usepackage{verbatim}

\usepackage{geometry}
\usepackage{comment}
\usepackage{graphicx,caption,subcaption}
\usepackage{tikz}
\usepackage{multirow} 

\colorlet{linkequation}{blue}
\usepackage[colorlinks,plainpages=true,pdfpagelabels,hypertexnames=true,colorlinks=true,pdfstartview=FitV,linkcolor=blue,citecolor=red,urlcolor=black]{hyperref}
\usepackage{hyperref}

\PassOptionsToPackage{unicode}{hyperref}
\PassOptionsToPackage{naturalnames}{hyperref}

\definecolor{dgreen}{rgb}{0,0.5,0}
\definecolor{violet}{rgb}{0.5,0,0.5}
\definecolor{dred}{rgb}{0.7,0,0}
\definecolor{ddred}{rgb}{0.5,0,0}
\definecolor{dblue}{rgb}{0,0,0.5}
\definecolor{ddblue}{rgb}{0,0,0.3}
\definecolor{llgray}{rgb}{0.9,0.9,0.9}
\definecolor{lgray}{rgb}{0.7,0.7,0.7}

\newtheorem{defn}{Definition}[section]
\newtheorem{lemma}[defn]{Lemma}
\newtheorem{proposition}[defn]{Proposition}
\newtheorem{theorem}[defn]{Theorem}

\newtheorem{remark}[defn]{Remark}
\numberwithin{equation}{section}

\newcommand{\bq}{\begin{equation}}
\newcommand{\eq}{\end{equation}}

\newcommand{\R}{{ \mathbb{R}  }}

\newcommand{\bbr}{{ \mathbb{R}  }}

\newcommand{\calC}{{ \mathcal C  }}

\newcommand{\calP}{{ \mathcal P  }}

\newcommand{\bke}[1]{\left( #1 \right)}

\newcommand{\norm}[1]{\left\Vert #1 \right\Vert}
\newcommand{\abs}[1]{\left| #1 \right|}

\DeclareMathOperator*{\esssup}{ess \ sup}

\DeclareMathOperator{\loc}{loc}

\AtBeginDocument{%
   \def\MR#1{}
}
\nocite{*}

\begin{document}

\title[Existence of PME with drifts in a bounded domain]{Existence of weak solutions for Porous medium equation \\
   with a divergence type of drift term in a bounded domain}

\author[S. Hwang]{Sukjung Hwang}
\address{S. Hwang: Department of Mathematics Education, Chungbuk National University, Cheongju 28644, Republic of Korea}
\email{sukjungh@chnu.ac.kr}

\author[K. Kang]{Kyungkeun Kang}
\address{K. Kang: Department of Mathematics, Yonsei University,  Seoul 03722, Republic of Korea}
\email{kkang@yonsei.ac.kr}

\author[H.K. Kim]{Hwa Kil Kim}
\address{H.K. Kim: Department of Mathematics Education, Hannam University, Daejeon 34430, Republic of Korea}
\email{hwakil@hnu.kr}

\thanks{
S. Hwang's work is partially supported by funding for the academic research program of Chungbuk National University in 2022 and NRF-2022R1F1A1073199. K. Kang's work is partially supported by NRF-2019R1A2C1084685. H. Kim's work is partially supported by NRF-2021R1F1A1048231.
}
\date{}

\makeatletter
\@namedef{subjclassname@2020}{%
  \textup{2020} Mathematics Subject Classification}
\makeatother




\begin{abstract}
We study porous medium equations with a divergence form of drift terms in a bounded domain with no-flux lateral boundary conditions. We establish $L^q$-weak solutions for $ 1\leq q < \infty$ in Wasserstein space under appropriate conditions on the drift, which is an extension of authors' previous works done in the whole space into the case of bounded domains.
Applying existence results to a certain Keller-Segel equation of consumption type, construction of $L^q$-weak solutions is also made, in case that the  equation of a biological organism is of porous medium type. 
\end{abstract}

\maketitle

{
  \hypersetup{linkcolor=black}
  \tableofcontents
}



\section{Introduction}

In this paper, we consider the porous medium equations with the divergence form of drifts in the form of 
\begin{equation}\label{E:Main}
\partial_t \rho =   \nabla \cdot( \nabla \rho^m  - V\rho) \qquad\ \text{ in } \ \Omega_{T}:= \Omega \times (0, T), 
\end{equation}
where $\Omega \Subset \bbr^d$, $d\geq 2$, is a bounded domain with smooth boundaries and $0 <T < \infty$. Here $V: \Omega_{T} \to \mathbb{R}^d$ is a given vector field, which will be more specific later. In this paper, we study the initial value problem of \eqref{E:Main} with the boundary conditions
\begin{equation}\label{E:Main-bc-ic}
 \left( \nabla \rho^m - V\rho \right)\cdot \textbf{n} = 0  \ \text{ on } \ \partial \Omega \times (0, T) \quad\ \text{ and } \quad\
   \rho(\cdot,0)=\rho_0  \ \text{ on } \ \Omega,
\end{equation}

where the vector $\textbf{n}$ is normal to the boundary $\partial \Omega$.

For given nonnegative initial data $\rho_0 \in L^{q}(\Omega)$ for $q\geq 1$, our main purpose is to establish the existence of nonnegative weak solutions of \eqref{E:Main}-\eqref{E:Main-bc-ic} when the drift belongs to proper function spaces, so-called {\it sub-scaling classes}, which also allows mass conservation property (see Definition \ref{D:Serrin}). 

This work is an extension to the case of bounded domains of previous results done by authors for the case of the whole space, i.e., $\Omega=\R^d$ (we refer \cite{HKK} for more details and references therein for related works). 

Although not a few minor modifications are required in comparison to the results in \cite{HKK} for whole space, it is worth making clear the noticeable differences for the case of bounded domains.
Firstly, moment estimates are required to control entropy and obtain estimates of speed for the whole space, and it is obviously straightforward in bounded domains. As a result, no need for  moment estimates allows us to construct $L^1$-weak solution newly in Theorem~\ref{T:weakSol}~(i), Theorem~\ref{T:weakSol_DivFree_1}, and  Theorem~\ref{T:weakSol_DivFree_1}.

Another good point for a bounded domain is the matter of the embedding of the drift vector field $V$, in particular for spatial variables. To be more precise, if $V$ is integrable in $L^{q_1}_xL^{q_2}_t(\Omega_T)$, then it is automatic that $L^{p_1}_xL^{p_2}_t(\Omega_T)$ for all $1\le p_i\le q_i, \, i=1,2$, and meanwhile it is in general not valid for $1\le p_1\le q_1$ for the whole space. In virtue of such advantage, the existence of weak solutions is ensured in a bit larger class of $V$ compared to the case of $\R^d$.  

On the other hand, since we look for solutions of mass conservation, we need to place a restriction on  $V$ relevant to the boundary condition \eqref{E:Main-bc-ic}. Roughly speaking, $V$ is an approximation of smooth vector fields with homogeneous normal flux (see more details in Definition~\ref{D:Serrin} and compare to \cite[Definition~1.1]{HKK}). Another shortcoming in bounded domains is the lack of uniqueness results since the way of proving it does not work out (compare to \cite[Section~2.2]{HKK}), although we expect that it is the case as in the whole space.

In the meantime, an improvement has been made regarding the compactness arguments so that we can cover the range $1 \leq q \leq m+1$ (see Proposition~\ref{P:AL1} for more precise statements), which is valid for the case of whole space as well. This new observation allows covering all $q\ge 1$ together with the result for the range $q \geq \max\{1, m-1\}$ proved in Proposition~\ref{P:AL2} and, as a result, the combination of both propositions enables us to enlarge the subscaling space of $V$. We reiterate that this additional compactness argument is applicable to the existence results in $\mathbb{R}^d$, and we give an improved version of statements in $\mathbb{R}^d$ in Appendix~\ref{Appendix:compact}.  

As an application, we establish  $L^q$-weak solution for a certain Keller-Segel equation of consumption type given in \eqref{PME-KS-10}-\eqref{PME-KS-bc} (see Section~\ref{SS:application}).
More specifically, with the aid of our existence results for \eqref{E:Main}-\eqref{E:Main-bc-ic},  $L^q$-weak solution with $1\leq q<\infty$ can be constructed in dimensions three and higher (see Theorem \ref{KS-thm}).

The notion of Wasserstein space and distance is reviewed in Section~\ref{SS:Wasserstein}, and the structures of $V$ are categorized as one of three different types: (i) $V \in \mathcal{S}_{m,q}^{(q_1, q_2)}$, or  (ii) $\nabla \cdot V \geq 0$, or (iii) $V \in \tilde{\mathcal{S}}_{m,q}^{(\tilde{q}_1, \tilde{q}_2)}$. The scaling invariant classes of $V$ are described in Definition~\ref{D:Serrin}, Remark~\ref{R:S-tildeS}, and \cite[Figs.~1, 2]{HKK}.

For a clear guide of our results, we deliver two tables: Table~\ref{Table1} for $L^q$-weak solutions satisfying an energy estimate, and Table~2 
for absolutely continuous $L^q$-weak solutions. Each type of solution is sorted by the structure of $V$ and the range of $m$ and $q$.

\begin{table}[hbt!]
\begin{center}
\caption{\footnotesize Guide of existence results of $L^q$-weak solutions in a bounded domain}
\smallskip
{\scriptsize
\begin{tabular}{| c  || c | c | c |}\hline
\rule[-8pt]{0pt}{22pt}
 \textbf{Structure of $V$} & {\textbf{Range of $m$ }} & {\textbf{Range of $q$}}  & \textbf{References}   \\ \hline \hline

\rule[-8pt]{0pt}{22pt}
\multirow{5}{*}{$V\in \mathcal{S}_{m,q}^{(q_1, q_2)}$}
& \multirow{3}{*}{$1<m\leq 2$}
& $q=1$
& Theorem~\ref{T:weakSol} (i), $\overline{\textbf{DE}}$ in Fig.~\ref{F:S:1m2}
 \\
\cline{3-4}

\rule[-8pt]{0pt}{22pt}

& 
& $q>1$
& Theorem~\ref{T:weakSol} (ii), Figs.~\ref{F:S:1m2}, \ref{F:S:1m2:more}
\\
\cline{2-4}

\rule[-8pt]{0pt}{22pt}
& $m>2$
& $q\geq m-1$
& Theorem~\ref{T:weakSol} (ii), Figs.~\ref{F:S:m2}, \ref{F:S:m2:more}
\\
\hline

\rule[-8pt]{0pt}{22pt}
\multirow{3}{*}{$\nabla \cdot V \geq 0$}
&\multirow{3}{*}{$m>1$}
& $q=1$
& Theorem~\ref{T:weakSol_DivFree_1}, Fig.~\ref{F:divfree:1}
\\
\cline{3-4}

\rule[-8pt]{0pt}{22pt}
& 
& $q>1$
& Theorem~\ref{T:weakSol_DivFree_q}, Fig.~\ref{F:divfree:q}
\\
\hline

\rule[-8pt]{0pt}{22pt}
\multirow{3}{*}{$ V \in \tilde{\mathcal{S}}_{m,q}^{(\tilde{q}_1, \tilde{q}_2)}$}
& \multirow{3}{*}{$m>1$}
& $q=1$
& Theorem~\ref{T:weakSol_tilde} (i), \cite[$\overline{\textbf{DE}}$ in Fig.~10 ]{HKK}
 \\
 \cline{3-4}
\rule[-8pt]{0pt}{22pt}
& 
& $q>1$
&Theorem~\ref{T:weakSol_tilde} (ii), \cite[Figs.~10, 19]{HKK}
\\
\hline

\end{tabular}
\label{Table1}
}
\end{center}
\end{table}

Some comments on Table~\ref{Table1} are prepared below.
\begin{itemize}
\item[(i)] For $L^1$-weak solution, we prescribe the initial data to satisfy $\rho_0 \in L^1 (\Omega)$ and $\int_{\Omega}\rho_0 \log \rho_0 \,dx < \infty$, because merely $\rho_0 \in L^1(\Omega)$ is not sufficient to obtain a priori estimate, Proposition~\ref{Lq-energy}~(i). We note that the $p$-moment estimate is crucially used to control the integrability of negative part of $\rho \log \rho$ when $\Omega=\R^d$, but it is automatic in case of bounded domains, since $\rho \log \rho$ is bounded below.

\item[(ii)] For the initial data $\rho_0 \in L^{q}(\Omega)$, $1<q<\infty$, we construct $L^{q}$-weak solutions in Theorem~\ref{T:weakSol}~(ii), Theorem~\ref{T:weakSol_DivFree_q}, and Theorem~\ref{T:weakSol_tilde}~(ii) which corresponds to \cite[Theorem~2.7, Theorem~2.11, and Theorem~2.15]{HKK}.

\item[(iii)] The notable difference is observed in Theorem~\ref{T:weakSol_DivFree_q} in case $\nabla \cdot V \geq 0$, compared to the whole space \cite[Section~2.1.3]{HKK}. Indeed, the range of $V$ is relaxed due to the improved compactness arguments and embeddings in both space and time axes. 

\end{itemize}

\begin{table}[hbt!]
\begin{center}
\caption{\footnotesize Guide of existence results of absolutely continuous $L^q$-weak solutions in a bounded domain}
\smallskip

{\scriptsize
\begin{tabular}{| c  || c | c | c |}\hline
\rule[-8pt]{0pt}{22pt}
 \textbf{Structure of $V$} &  {\textbf{Range of $m$ }} & {\textbf{Range of $q$}}  & \textbf{References}  \\ \hline \hline

\rule[-8pt]{0pt}{22pt}
\multirow{5}{*}{$V\in \mathcal{S}_{m,q}^{(q_1, q_2)}$}
&  \multirow{3}{*}{$1 < m \leq 2$}
&  $q=1$
& Theorem~\ref{T:ACweakSol} (i), $\overline{\textbf{DE}}$ in Fig.~\ref{F:S:1m2}
 \\
\cline{3-4}

\rule[-8pt]{0pt}{22pt}

& 
& $q>1$
& Theorem~\ref{T:ACweakSol} (ii), Figs.~\ref{F:S:1m2}, \ref{F:S:1m2:more}
\\
\cline{2-4}

\rule[-8pt]{0pt}{22pt}
& $m>2$
& $q\geq m-1$
& Theorem~\ref{T:ACweakSol} (ii), Figs.~\ref{F:S:m2}, \ref{F:S:m2:more}
\\
\hline

\rule[-8pt]{0pt}{22pt}
\multirow{3}{*}{$\nabla \cdot V \geq 0$}
& \multirow{3}{*}{$m>1$}
& $q=1$ 
& Theorem~\ref{T:ACweakSol_DivFree} (i), \cite[$\overline{\textbf{DE}}$ in Fig.~9 ]{HKK} 
\\
\cline{3-4}
\rule[-8pt]{0pt}{22pt}
&
& $q>1$
& Theorem~\ref{T:ACweakSol_DivFree} (ii), \cite[Figs.~9, 15, 16]{HKK}
 \\
\hline

\rule[-8pt]{0pt}{22pt}
\multirow{3}{*}{$ V \in \tilde{\mathcal{S}}_{m,q}^{(\tilde{q}_1, \tilde{q}_2)}$}
& \multirow{3}{*}{$m>1$}
& $q=1$ 
& Theorem~\ref{T:ACweakSol_tilde} (i), \cite[$\overline{\textbf{DE}}$ in Fig.~10 ]{HKK}
 \\
 \cline{3-4}
\rule[-8pt]{0pt}{22pt}
&
& $q>1$
&Theorem~\ref{T:ACweakSol_tilde} (ii), \cite[Figs. 10, 19]{HKK}
\\
\hline

\rule[-8pt]{0pt}{22pt}
& \multicolumn{2}{c|}{Embedding results}
& Theorem~\ref{T:ACweakSol}~(iii), Theorem~\ref{T:ACweakSol_DivFree}~(iii), Theorem~\ref{T:ACweakSol_tilde}~(iii) \\
\hline

\end{tabular}
\label{Table2}
}
\end{center}
\end{table}

Some comments on Table~\ref{Table2} are also prepared below.

\begin{itemize}
\item[(i)]
The existence results of absolutely continuous $L^q$-weak solutions for \eqref{E:Main}-\eqref{E:Main-bc-ic} agree mostly with the existence results of the same weak solutions in $\mathbb{R}^d\times (0, T)$ in \cite{HKK}. 
The references for $q=1$, Theorem~\ref{T:ACweakSol}~(i),
Theorem~\ref{T:ACweakSol_DivFree}~(i), and Theorem~\ref{T:ACweakSol_tilde}~(i) corresponds to Theorem~2.3, Theorem~2.4, and Theorem~2.5 in \cite{HKK}, respectively. For absolutely continuous $L^q$-weak solutions, Theorem~\ref{T:ACweakSol}~(ii) matches to \cite[Theorem~2.9]{HKK}, Theorem~\ref{T:ACweakSol_DivFree}~(ii) is to \cite[Theorem~2.13]{HKK}, and Theorem~\ref{T:ACweakSol_tilde}~(ii) is to \cite[Theorem~2.16]{HKK}. 

\item[(ii)] 
We remark, however, that there is a difference caused by finite size of a bounded domain compared to $\mathbb{R}^d$.
Indeed, the results of embeddings for spatial variables 
in Theorem~\ref{T:ACweakSol}~(iii), Theorem~\ref{T:ACweakSol_DivFree}~(iii), and Theorem~\ref{T:ACweakSol_tilde}~(iii), are newly extended compare to Theorem~2.3~(ii) and Theorem~2.9~(ii), Theorem~2.13~(ii), and Theorem~2.16~(ii) in \cite{HKK}. 
For example, a comparison of Fig.~\ref{F:S:m2} and \cite[Fig. 21]{HKK} shows that the embedding in space variables affects the range of $(q_1, q_2)$. 

\end{itemize}

Our paper is organized as follows:
In Section~\ref{S: Main}, we state all main results, and several remarks and figures are provided to help readers understand.
Some preliminaries are prepared in Section~\ref{S:Preliminaries}. Section~\ref{S:a priori} is devoted to giving a priori estimates for regular solutions (see Definition~\ref{D:regular-sol}).
In Section~\ref{splitting method}, the existence of regular solutions of PME is established by the splitting method.
Section~\ref{Exist-weak} is prepared for proofs of existence results stated in Section~\ref{S: Main}.
We present supplementary figures in Appendix~\ref{Appendix:fig} and comments on compactness arguments in Appendix~\ref{Appendix:compact}.



\section{Main results}\label{S: Main}

In this section, we describe our main results and make relevant remarks for each of them. First, the existence of weak solutions for \eqref{E:Main}-\eqref{E:Main-bc-ic} are categorized according to hypotheses of given initial data and drifts. As an application, we study coupled parabolic type Keller-Segel equations to improve previously known results by taking advantage of developed main results.

Before stating the main theorems, for convenience, we introduce notations.
\begin{itemize}
\item Let us denote by $\mathcal{P}(\Omega)$ the set of all Borel
probability measures on $\Omega$.
We refer Section~\ref{SS:Wasserstein} for definitions and related properties of the Wasserstein distance denoted by $W_p$, the Wasserstein space, and AC (absolutely continuous) curves.

\item  Let $A\subset \mathbb{R}$ be a measurable set. We denote by $\delta_A$ the indicator function of $A$. 

  \item For any $m>1$ and $q\geq 1$, let us define the following constant
\begin{equation}\label{lambda_q}
 \lambda_q := \min \left\{2, 1+\frac{d(q-1)+q}{d(m-1)+q} \right\},
\end{equation}
that can be rewritten as ${\lambda_q} = \left\{1+\frac{d(q-1)+q}{d(m-1)+q} \right\}\cdot \delta_{\{1 \leq q \leq m\}} + 2 \cdot \delta_{\{q > m\}}$.
The constant ${\lambda_q} \in (1, 2]$ in \eqref{lambda_q}  naturally comes from evaluating the speed estimates which are essential to play in Wasserstein spaces (refer Section~\ref{SS:Speed}). A straightforward observation gives that $ 1 + \frac{d(q-1)+q}{d(m-1)+q}=2$ if $q=m$.
Also, for any $q\geq 1$, note that ${\lambda_q} =2$ if $m=1$.

\item The letters $c$ and $C$ are used for generic constants. Also, the letter $\theta$ is a generic constant  which varies with arguments of interpolation.
Throughout the paper, we omit the dependence on $m$, $q$, $d$, $T$, and $\|\rho_0\|_{L^{1}(\Omega)}$, because we regard $m>1$, $q \geq 1$, $d\geq 2$, $T>0$ are given constants and $\|\rho_0\|_{L^{1}(\Omega)} = 1$
in $\mathcal{P}(\mathbb{R}^d)$. 

\item In figures, the notation $\mathcal{R}(\textbf{Ab\ldots Yz})$ is used to indicate a polygon with vertices $\textbf{A}, \textbf{b}, \ldots, \textbf{Y}, \textbf{z}$. Also the notation $\overline{\textbf{Ab\ldots Yz}}$ is used for the piecewise line segments connecting $\textbf{A}$, $\textbf{b}$, $\ldots$, $\textbf{Y}$, $\textbf{z}$.
\end{itemize}

Here we introduce the notion of weak solutions of \eqref{E:Main}-\eqref{E:Main-bc-ic}.
\begin{defn}\label{D:weak-sol}
Let $q\in [1, \infty)$ and $V$ be a measurable vector field.
We say that a nonnegative measurable function $\rho$ is
a \textbf{$L^q$-weak solution} of \eqref{E:Main}-\eqref{E:Main-bc-ic} with $\rho_0 \in L^{q}(\Omega)$ if the followings are satisfied:
\begin{itemize}
\item[(i)] It holds that
\[
\rho \in L^{\infty}\left(0, T; L^q (\Omega)\right)\cap L^{m} (\Omega_T),\quad \nabla \rho^m \in L^{1}(\Omega_T), \quad \nabla \rho^{\frac{m+q-1}{2}} \in L^{2}(\Omega_T), \quad   \text{and} \quad \rho V \in L^{1}(\Omega_T).
\]
\item[(ii)] For any function $\varphi \in {\mathcal{C}}^\infty_c (\overline{\Omega} \times [0,T))$, it holds that
\begin{equation}\label{KK-May7-40}
\iint_{\Omega_T} \left\{ \rho \varphi_t - \nabla \rho^m \cdot \nabla \varphi + \rho V \cdot \nabla \varphi \right\} \,dx dt = -\int_{\Omega} \rho_{0} (\cdot) \varphi(\cdot, 0) \,dx.
\end{equation}
\end{itemize}
\end{defn}

We remind the property of \emph{mass conservation} for nonnegative solutions of \eqref{E:Main}-\eqref{E:Main-bc-ic}, i.e.
$\|\rho(\cdot, t)\|_{L^{1} (\Omega)} = \|\rho_0\|_{L^{1}(\Omega)}$ for a.e. $t\in (0, T)$ (see e.g. \cite[Theorem~11.2]{Vaz07}). Without loss of generality, assuming that $\rho_0 \in \mathcal{P} (\Omega)$, that is, $\|\rho_0\|_{L^1 (\Omega)} = 1$, the mass conservation property in time implies that $\|\rho(\cdot, t)\|_{L^{1} (\Omega)} = 1$ a.e.  $t\in (0, T)$. If $\|\rho_0\|_{L^1(\Omega)} = c > 0$, then we replace $\rho$ by $\tilde{\rho} = \rho/c$ that does satisfy the equation with similar structure (see \cite[Remark~2.2]{HKK}).

Motivated by preserving $L^q$-norm of a weak solution of \eqref{E:Main}, we derive the corresponding scaling-invariant classes for $V$ and $\nabla V$ in $L^{q_1, q_2}_{x,t}$ spaces (refer \cite[Definition~1.1]{HKK}). In the following definition, we introduce, the (sub)scaling-invariant classes for $V$ and $\nabla V$ adapted to the conservation of mass and boundary conditions \eqref{E:Main-bc-ic}.

\begin{defn}\label{D:Serrin}
Let $m, q \geq 1$ and $q_{m,d} := \frac{d(m-1)}{q}$. For constants $q_1, q_2, \tilde{q}_1, \tilde{q}_2 > 0$, we define following spaces.
\begin{itemize}
\item[(i)]Define
\[
\mathcal{L}_{x,t}^{q_1, q_2}:= \overline{\{V : V \in C_c^\infty(\overline{\Omega} \times[0, T)) ~~\mbox{and}~~ V\cdot \textbf{n} = 0 ~ \mbox{on} ~ \partial \Omega\}}^{L_{x,t}^{q_1, q_2}},
\]
where $\textbf{n}$ is the outward unit normal vector to the boundary of $\Omega$.
That is, the vector field $ V \in \mathcal{L}_{x,t}^{q_1, q_2}$ means that there exists a sequence $V_n  \in C_c^\infty(\overline{\Omega} \times[0, T)) $ such that
$ V\cdot \textbf{n} = 0$ on $\partial \Omega$ and
\[
\|V_n-V \|_{L_{x,t}^{q_1, q_2}} \rightarrow 0 \quad \mbox{as} ~~ n \rightarrow \infty.
\]

\item[(ii)] The \textit{scaling invariant classes} of $V$ and $\nabla  V$ are defined as
\begin{equation}\label{Serrin}
 \mathcal{S}_{m,q}^{(q_1, q_2)}:= \left\{ V \in \mathcal{L}^{q_1, q_2}_{x,t} : \ \|V\|_{\mathcal{L}^{q_1, q_2}_{x,t}} < \infty \ \text{ where } \ \frac{d}{q_1} + \frac{2+q_{m,d}}{q_2} = 1 + q_{m,d} \right\},
\end{equation}
and, for $\tilde{q}^{\ast}_1 = \frac{d \tilde{q}_1}{d-\tilde{q}_1}$ where $\tilde{q}_1 \in (1, d)$, 
\begin{equation}\label{Serrin-grad}
  \tilde{\mathcal{S}}_{m, q}^{(\tilde{q}_1, \tilde{q}_2)}:= 
  \left\{ V \in \mathcal{L}_{x,t}^{\tilde{q}^{\ast}_1, \tilde{q}_2} : 
  \ \|\nabla V\|_{\mathcal{L}^{\tilde{q}_1, \tilde{q}_2}_{x,t}} < \infty  \ \text{ where }  \ \frac{d}{\tilde{q}_1} + \frac{2+q_{m,d}}{\tilde{q}_2} = 2 + q_{m,d}\right\}. 
\end{equation}
Moreover, let us name $\|V\|_{\mathcal{S}_{m,q}^{(q_1, q_2)}}$ and $\|V\|_{\tilde{\mathcal{S}}_{m,q}^{(\tilde{q}_1, \tilde{q}_2)}}$, as the  \textit{scaling invariant norms} corresponding to each spaces.

\item[(iii)] The \textit{sub-scaling classes} are defined as
\begin{equation}\label{subSerrin}
 \mathfrak{S}_{m,q}^{(q_1, q_2)}:= \left\{ V \in  \mathcal{L}^{q_1, q_2}_{x,t} : \ \|V\|_{\mathcal{L}^{q_1, q_2}_{x,t}} < \infty \ \text{ where } \ \frac{d}{q_1} + \frac{2+q_{m,d}}{q_2} \leq 1 + q_{m,d} \right\},
\end{equation}
and, for $\tilde{q}^{\ast}_1 = \frac{d \tilde{q}_1}{d-\tilde{q}_1}$ where $\tilde{q}_1 \in (1, d)$, 
\begin{equation}\label{subSerrin-grad}
  \tilde{\mathfrak{S}}_{m, q}^{(\tilde{q}_1, \tilde{q}_2)}:= \left\{ V \in \mathcal{L}_{x,t}^{\tilde{q}^{\ast}_1, \tilde{q}_2} : 
  \ \|\nabla V\|_{\mathcal{L}^{\tilde{q}_1, \tilde{q}_2}_{x,t}} < \infty  \text{ where }  \frac{d}{\tilde{q}_1} + \frac{2+q_{m,d}}{\tilde{q}_2} \leq 2 + q_{m,d} \right\}.
\end{equation}
Let us name $\|V\|_{\mathfrak{S}_{m,q}^{(q_1, q_2)}}$ and $\|V\|_{\tilde{\mathfrak{S}}_{m,q}^{(\tilde{q}_1, \tilde{q}_2)}}$, as the  \textit{sub-scaling norms} corresponding to each spaces.

\end{itemize}
\end{defn}

\begin{remark}\label{R:S-tildeS}
\begin{enumerate}
\item[(i)] We remark that the existence results are given under certain assumptions on $V$ in sub-scaling classes. The critical case matters the most and the strict sub-scaling case (the case of strict inequality in \eqref{subSerrin} or \eqref{subSerrin-grad}) is simpler, and therefore, proofs throughout the paper concern only the case of scaling invariant classes.

\item[(ii)] The range of a pair $(q_1, q_2)$ satisfying \eqref{Serrin} is plotted in \cite[Fig. 1]{HKK}. As $m \to 1$ or $q\to \infty$, the constant $q_{m,d}$ becomes $0$, and the conditions in \eqref{Serrin} correspond to 
\begin{equation}\label{linear-Serrin}
V \in L^{q_1, q_2}_{x,t} \quad \text{ where } \quad \frac{d}{q_1} + \frac{2}{q_2} \leq 1,
\end{equation}
which is critical conditions concerning the continuity of \eqref{E:Main} (see, for example, \cite{CHKK17, KZ18, HZ21}). Therefore, the scaling invariant conditions \eqref{Serrin} are more comprehensive for showing correlations of the nonlinear factor $m$ constructing $L^q$-weak solutions, which also includes \eqref{linear-Serrin} as a special case when $m=1$ or $q=\infty$.  
    
\item[(iii)] The equation \eqref{E:Main} is defined on a bounded domain in both spatial and temporal directions; hence, we are able to compute the inequality between two scaling invariant norms. For $1 \leq q \leq q^{L}$, $q_1^L \geq q_1$, and $q_2^L \geq q_2$, it holds that $\mathcal{S}_{m,q}^{(q_1, q_2)} \supseteq \mathcal{S}_{m,q^L}^{(q_1^L, q_2^L)}$ by H\"{o}lder inequalities.


\item[(iv)] Fig. 2 in \cite{HKK} shows the graphs of $(\tilde{q}_1, \tilde{q}_2)$ satisfying 
$\tilde{\mathcal{S}}_{m,q}^{(\tilde{q}_1, \tilde{q}_2)}$ in \eqref{Serrin-grad}. In fact, by using embedding property for $\tilde{q_1} \in (1, d)$, one can suppress the norm of $V$ by the norm of $\nabla V$;  that is,
    $\|V\|_{\mathcal{S}_{m,q}^{(\tilde{q}^{\ast}_{1},  \tilde{q}_2)}} \lesssim  \|V\|_{\tilde{\mathcal{S}}_{m,q}^{(\tilde{q}_1, \tilde{q}_2)}} $
   provided
    \begin{equation}\label{tilde-q2}
    \begin{cases}
    \frac{2+q_{m,d}}{1+{q_{m,d}}} < \tilde{q}_2 \leq \infty, & \text{ if } 1 < m < 1+ \frac{q (d-2)}{d}  \vspace{1 mm}\\
    \frac{2+q_{m,d}}{1+{q_{m,d}}} < \tilde{q}_2 < \frac{2+q_{m,d}}{2-d+q_{m,d}}, & \text{ otherwise}.
    \end{cases}
    \end{equation}
In \cite[Fig. 2]{HKK}, the intersection of lines and shaded region is where \eqref{tilde-q2} holds.
\end{enumerate}
\end{remark}

\subsection{Existence results}\label{SS:Existence}

In this section, we present all existence results and related remarks composed of three subsections depending on the structure of the drift $V$. Results in each subsection are categorized by the type of weak solutions, the initial data, and the range of $m$ and $q$. 

\subsubsection{Existence for case: $V$ in sub-scaling classes}\label{SS:energy-sol}

Here, we assume that $V$ belongs to the sub-scaling class $\mathfrak{S}_{m,q}^{(q_1,q_2)}$ that includes the scaling invariant class $\mathcal{S}_{m,q}^{(q_1,q_2)}$. We begin with  $L^q$-weak solutions for $1 \leq q < \infty$ satisfying energy inequality, and later we introduce absolutely continuous $L^q$-weak solutions. 

\begin{theorem}\label{T:weakSol}
Let $m > 1$ and $q\geq 1$.
\begin{itemize}
\item[(i)] Let $1 <m \leq 2$. Assume that $\rho_0 \in  \mathcal{P}(\Omega)$ and $\int_{\Omega} \rho_0 \log \rho_0 \,dx < \infty$ and
\begin{equation}\label{T:weakSol:V_1}
V \in \mathfrak{S}_{m,1}^{(q_1,q_2)} \ \text{ for } \
    \begin{cases}
         2 \leq q_2 \leq  \frac{m}{m-1},  & \text{ if } d > 2 \vspace{1 mm}\\
         2 \leq q_2 <  \frac{m}{m-1},  & \text{ if } d = 2.
    \end{cases}
\end{equation}
Then, there exists a nonnegative $L^1$-weak solution of \eqref{E:Main}-\eqref{E:Main-bc-ic} in Definition~\ref{D:weak-sol}
that holds
 \begin{equation}\label{T:weakSol:E_1}
 \esssup_{0 \leq t \leq T} \int_{\Omega} \rho \abs{\log \rho  }(\cdot, t) \,dx
 + \iint_{\Omega_T} \abs{\nabla \rho^{\frac{m}{2}}}^2 \,dx\,dt
\leq C,
\end{equation}
with $C = C ( \|V\|_{\mathfrak{S}_{m,1}^{(q_1, q_2)}}, \int_{\Omega} \rho_0 \log \rho_0 \,dx)$.

\item[(ii)] Let $q>1$ and $q \geq m-1$.
Assume that $\rho_0 \in  \mathcal{P}(\Omega) \cap L^{q}(\bbr^d)$ and
\begin{equation}\label{T:weakSol:V_q}
V \in \mathfrak{S}_{m,q}^{(q_1,q_2)} \ \text{ for } \
    \begin{cases}
         2 \leq q_2 \leq  \frac{q+m-1}{m-1},  & \text{ if } d > 2 \vspace{1 mm}\\
         2 \leq q_2 <  \frac{q+m-1}{m-1},  & \text{ if } d = 2.
    \end{cases}
\end{equation}
Then, there exists a nonnegative $L^q$-weak solution of \eqref{E:Main}-\eqref{E:Main-bc-ic} in Definition~\ref{D:weak-sol}
that holds
 \begin{equation}\label{T:weakSol:E_q}
 \esssup_{0 \leq t \leq T} \int_{\Omega} \rho^q (\cdot, t) \,dx
 + \iint_{\Omega_T} \abs{\nabla \rho^{\frac{q+m-1}{2}}}^2 \,dx\,dt
\leq C,
\end{equation}
with $C = C ( \|V\|_{\mathfrak{S}_{m,q}^{(q_1, q_2)}}, \|\rho_{0}\|_{L^{q} (\Omega)})$.
\end{itemize}
\end{theorem}

We make a few remarks regarding Theorem~\ref{T:weakSol}.
\begin{remark}\label{R:T:energy}
\begin{enumerate}
\item[(i)] Compared to the existence results of \eqref{E:Main} on $\mathbb{R}^d \times [0, T]$ (see \cite[Section~2.1.1]{HKK}), the existence of $L^1$-weak solution in Theorem~\ref{T:weakSol} (i) is new because $p$-th moment estimates are trivial on a bounded domain. The valid range for $(q_1, q_2)$ of \eqref{T:weakSol:V_1} is the line $\overline{\textbf{ED}}$ on Fig.~\ref{F:S:1m2}.

\item[(ii)] Because of energy estimates, there are competing inequalities $q>1$ and $q\geq m-1$ which are sorted into two cases either $q>1$ \& $1<m\leq 2$ or $q\geq m-1$ \& $m>2$. 
The range of $(q_1,q_2)$ satisfying \eqref{T:weakSol:V_q} is illustrated as  $\mathcal{R}(\textbf{abDE})$ in Figs.~\ref{F:S:1m2}, \ref{F:S:1m2:more} or $\mathcal{R}(\textbf{abD})$ in Figs.~\ref{F:S:m2}, \ref{F:S:m2:more} (Figs.~\ref{F:S:1m2:more}, \ref{F:S:m2:more} are in Appendix~\ref{Appendix:fig}).
\end{enumerate}
\end{remark}

Now, we establish the absolutely continuous $L^q$-weak solutions of \eqref{E:Main}-\eqref{E:Main-bc-ic}. Compared to results on the unbounded domain (see \cite[Theorem~2.9]{HKK}), the difference appears when we apply the embedding in the spatial variable. 

\begin{theorem}\label{T:ACweakSol}
Let $m>1$ and $q\geq 1$. Also let $\lambda_q$ be given in \eqref{lambda_q}.
Suppose that $\rho_0 \in
\mathcal{P}(\Omega) \cap L^{q} (\Omega)$.
\begin{itemize}
\item[(i)] Let $1 < m \leq 2$. Suppose that $\rho_0 \in \mathcal{P}(\Omega)$ and $\int_{\Omega} \rho_0 \log \rho_0 \,dx < \infty$ and $V \in \mathfrak{S}_{m,1}^{(q_1,q_2)} $ satisfies \eqref{T:weakSol:V_1}.
Then, there exists a nonnegative $L^1$-weak solution of \eqref{E:Main}-\eqref{E:Main-bc-ic} in Definition~\ref{D:weak-sol} such that  $\rho \in AC(0,T; \mathcal{P} (\Omega))$ with $\rho(\cdot, 0)=\rho_0$.
Furthermore, $\rho$ satisfies
\begin{equation}\label{T:ACweakSol:E_1}
\esssup_{0\leq t \leq T} \int_{\mathbb{R}^d\times \{t\}} \rho \abs{\log \rho} (\cdot, t)  \,dx
+ \iint_{\Omega_T} \{ \abs{\nabla \rho^{\frac m2}}^2 + \left(\abs{\frac{\nabla
\rho^m}{\rho}}^{\lambda_1}+|V|^{\lambda_1}\right ) \rho  \}\,dx\,dt
 \leq  C,
\end{equation}
and
\begin{equation}\label{T:ACweakSol:W_1} 
 W_{\lambda_1}(\rho(t),\rho(s))\leq C (t-s)^{\frac{\lambda_1 -1}{\lambda_1}},\qquad
\forall ~~0\leq s\leq t\leq T,
\end{equation}
where the constant $C= C ( \|V\|_{\mathfrak{S}_{m,1}^{(q_1,q_2)}},\,
\int_{\Omega} \rho_0 \log \rho_0 \,dx )$.

  \item [(ii)]  Let $q>1$ and $q\geq m-1$. Assume that
\begin{equation}\label{T:ACweakSol:V_q}
V \in \mathfrak{S}_{m,q}^{(q_1,q_2)} \ \text{ for } \
    \begin{cases}
        q_1 \leq \frac{2m}{m-1}, \  2 \leq q_2 \leq  \frac{q+m-1}{m-1} \,\delta_{\{  1 \leq q \leq m\}} + \frac{2m-1}{m-1}\,\delta_{ \{q > m \} },  & \text{ if } d > 2, \vspace{1 mm}\\
        q_1 \leq \frac{2m}{m-1}, \ 2 \leq q_2 <  \frac{q+m-1}{m-1}\,\delta_{\{1 \leq q \leq m\}} + \frac{2m-1}{m-1} \,\delta_{ \{q > m \} },  & \text{ if } d = 2.
    \end{cases}
\end{equation}
Then, there exists a nonnegative $L^q$-weak solution of \eqref{E:Main}-\eqref{E:Main-bc-ic} in Definition~\ref{D:weak-sol} such that  $\rho \in AC(0,T; \mathcal{P} (\Omega))$ with $\rho(\cdot, 0)=\rho_0$.
Furthermore, $\rho$ satisfies
\begin{equation}\label{T:ACweakSol:E_q}
 \esssup_{0 \leq t \leq T} \int_{\mathbb{R}^d}  \rho^q  \,dx
 + \iint_{\Omega_T} \{ \left| \nabla \rho^{\frac{q+m-1}{2}}\right |^2 + \left(\abs{\frac{\nabla \rho^m}{\rho}}^{\lambda_q}+|V|^{\lambda_q}\right ) \rho \} \,dx\,dt \leq C,
\end{equation}
and
\begin{equation}\label{T:ACweakSol:W_q} 
 W_{\lambda_q} (\rho(t),\rho(s))\leq C (t-s)^{\frac{\lambda_q -1}{\lambda_q}},\qquad
\forall ~~0\leq s\leq t\leq T,
\end{equation}
where $C = C (\|V\|_{\mathfrak{S}_{m,q}^{(q_1, q_2)}}, \, \|\rho_0\|_{L^{q} (\Omega)} )$.

\item[(iii)] (Embedding) Let $q>1$ and $q\geq m-1$. Further assume that $V$ satisfies \eqref{T:weakSol:V_q}. For a pair of constants $(q_1, q_2)$ satisfying \eqref{T:weakSol:V_q}, there exists a pair of constants $(q_1^\ast, q_2^\ast)$ where $V$ belongs $\mathfrak{S}_{m, q^\ast}^{(q_1^{\ast}, q_2^\ast)}$ satisfying \eqref{T:ACweakSol:V_q} for some $q^\ast \in [1, q]$. Furthermore, the same conclusions hold as in (ii) except that $\lambda_q$ is replaced by $\lambda_{q^\ast}$.
\end{itemize}
\end{theorem}

Here are remarks about Theorem~\ref{T:ACweakSol}.

\begin{remark}\label{R:T:ACweakSol}
\begin{itemize}
\item[(i)] To establish an absolutely continuous solution, both energy estimates (\eqref{V-L1-energy} or \eqref{V-Lq-energy} for $1\leq q \leq m$)  and speed estimates (\eqref{V:speed} for $q>m$) are required. Compared to \cite[Theorem~2.9]{HKK}, the results above are same except part (iii) which is more general because now it is possible to apply embedding in both space and time variables. 

\item[(ii)] In Figs.~\ref{F:S:1m2}, \ref{F:S:1m2:more}, dark shaded region $\mathcal{R}(\textbf{CDEFG})$ illustrates the pairs $(q_1, q_2)$ satisfying \eqref{T:ACweakSol:V_q} in case $1 < m \leq 2$ and $q\geq 1$. In case $m>2$ and $q\geq m-1$, dark shaded region $\mathcal{R}(\textbf{CDEG})$ in Figs.~\ref{F:S:m2}, \ref{F:S:m2:more} shows the range of $(q_1, q_2)$ satisfying \eqref{T:ACweakSol:V_q}.

\begin{figure}
\centering

\begin{tikzpicture}[domain=0:16]



\fill[fill= lgray]
(0,2) -- (1.2, 0) -- (2.35, 1.1) -- (1.45, 2);


\fill[fill= gray]
(0.4, 1.3)--(0.8, 0.63) -- (1.85,0.63) -- (2.35, 1.1) -- (1.45, 2)--(0.4, 2);

\draw[->] (0,0) node[left] {\scriptsize $0$}
-- (5,0) node[right] {\scriptsize $\frac{1}{q_1}$};
\draw[->] (0,0) -- (0,4.5) node[left] { \scriptsize $\frac{1}{q_2}$};

\draw (0,4) node{\scriptsize $+$} node[left]{\scriptsize $1$} ;
\draw (4,0) node{\scriptsize $+$} ;

\draw[very thin] (0, 2) -- (1.45, 2) ;
\draw(1.45, 2) node{\scriptsize $\bullet$} node[right] {\scriptsize \textbf{E}};

\draw[very thin] (0, 2) -- (1.45, 2);

\draw[very thin]  (1.2, 0) -- (2.35, 1.1);

\draw (2.35, 1.1) node{\scriptsize $\bullet$} node[right] {\scriptsize \textbf{D}};

\draw[thick] (0.8, 0.63) circle(0.05) node[left] {\scriptsize \textbf{B}};
\draw[very thin]
 (0.8, 0.63)
-- (1.85, 0.63) node{\scriptsize $\bullet$} node[right]{\scriptsize \textbf{C}} ;

\draw[thick] (0.4,1.3) circle(0.05) node[below] {\scriptsize \textbf{A}};
\draw[very thin]
(0.4, 1.3)--(0.4, 2) node{\scriptsize $\bullet$} node[above] {\scriptsize \textbf{F}};

\draw[thick] (0,2) circle(0.05) node[left] {\scriptsize \textbf{a}};
\draw[thick] (1.2,0) circle(0.05) node[below] {\scriptsize \textbf{b}};
\draw[dashed] (0,2) -- (1.2, 0);
\draw (0,3.4) node {\scriptsize $\times$} node[left] {\scriptsize $\frac{1}{p_1}$}
-- (3.5, 0) node {\scriptsize $\times$} node[below] {\scriptsize $\frac{1+d(m-1)}{d}$};
\draw (3.9, 0.3) node{\scriptsize $\mathcal{S}_{m,1}^{(q_1, q_2)}$};

\draw (0,2.4) node {\scriptsize $\times$} node[left] {\scriptsize $\frac{m+d(m-1)}{2m+d(m-1)}$}
    -- (2.5, 0) node {\scriptsize $\times$} node[below] {\scriptsize \textbf{c}};
\draw (2.7, 0.3) node{\scriptsize $\mathcal{S}_{m,m}^{(q_1, q_2)}$};

\draw[thin, dotted] (0.4, 0) node{\scriptsize $*$} node[below] {\scriptsize $\frac{m-1}{2m}$} --(0.4, 2);

\draw[thin, dotted] (0, 1.3) node{\scriptsize $*$}  -- (0.4, 1.3);
\draw[thin, dotted] (0, 0.63) node{\scriptsize $*$} node[left]{\scriptsize $\frac{m-1}{2m-1}$}  -- (0.8, 0.63);
\draw[thin, dotted] (0.8, 0) node{\scriptsize $*$}  -- (0.8, 0.63);

\draw[thin, dotted] (1.85, 0) node{\scriptsize $*$}  -- (1.85, 0.63);

\draw[thin, dotted] (2.35, 0) node{\scriptsize $*$}  -- (2.35, 1.1);
\draw[thin, dotted] (0, 1.1) node{\scriptsize $*$}  -- (2.35, 1.1);

\draw[thin, dotted] (1.45, 0) node{\scriptsize $*$}  -- (1.45, 2);

\draw (6, 4.3) node[right] {\scriptsize $\overline{\textbf{ab}} = \mathcal{S}_{m, \infty}^{(q_1, q_2)}$, $\textbf{a} = (0, \frac 12)$, $\textbf{b} = (\frac 1d, 0)$ };

\draw (6, 3.5) node[right] {\scriptsize $\mathcal{R} (\textbf{ABCDEF})$: scaling invariant class of \eqref{T:ACweakSol:V_q}. };
\draw (6, 3) node[right] {\scriptsize $\mathcal{R} (\textbf{bCB})$: scaling invariant class of \eqref{T:weakSol:V_q}. };

\draw (6, 2.3) node[right] {\scriptsize $\textbf{A} = (\frac{m-1}{2m}, \frac{d+m(2-d)}{4m})$, $\textbf{B} = (\frac{1}{d(2m-1)}, \frac{m-1}{2m-1})$ };
\draw (6, 1.7) node[right] {\scriptsize $\textbf{C} = (\frac{1+d(m-1)}{d(2m-1)}, \frac{m-1}{2m-1})$, $\textbf{c}= (\frac{m+d(m-1)}{md},0)$ };
\draw (6, 1.2) node[right] {\scriptsize $\textbf{D} = (\frac{(2-m)+d(m-1)}{md}, \frac{m-1}{m})$};
\draw (6, 0.7) node[right] {\scriptsize $\textbf{E} = (\frac{m-1}{2}, \frac 12)$, $\textbf{F} = (\frac{m-1}{2m}, \frac 12)$};

\end{tikzpicture}
\caption{\footnotesize  Theorem~\ref{T:ACweakSol} for $1<m\leq 2$, $q>1$, $2 < d \leq \max\{2, \frac{2m}{(2m-1)(m-1)}\}$.}
\label{F:S:1m2}
\end{figure}

\item[(iii)] (Embedding Figure) In Fig.~\ref{F:S:m2}, we illustrate the strategy to obtain an absolutely continuous $L^q$-weak solution as long as $V$ holds \eqref{T:weakSol:V_q} by using embedding arguments. Starting from  $(q_1, q_2)$ satisfying \eqref{T:ACweakSol:V_q}, we are able to find $q^\ast \in [1, q]$ and a pair $(q_1^\ast, q_2^\ast)$ where $V \in \mathcal{S}_{m,q^\ast}^{(q_1^\ast, q_2^\ast)}$ satisfying \eqref{T:ACweakSol:V_q} and $\|V\|_{\mathcal{S}_{m,q^\ast}^{(q_1^\ast, q_2^\ast)}} \leq c\|V\|_{\mathcal{S}_{m,q}^{(q_1, q_2)}}$ by Remark~\ref{R:S-tildeS}~(iii) (cf. \cite[Theorem~2.9, Fig. 21]{HKK}).

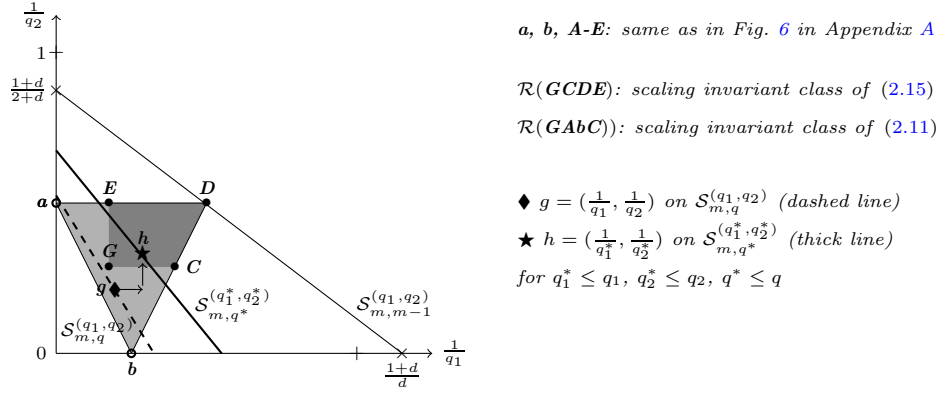
\begin{figure}
\centering 

\begin{tikzpicture}[domain=0:16]



\fill[fill= lgray]
(0, 2) -- (1,0) -- (2,2);


\fill[fill= gray]
(0.7, 2) -- (0.7,1.15) -- (1.58, 1.15) --(2,2);

\draw[->] (0,0) node[left] {\scriptsize $0$} -- (5,0) node[right] {\scriptsize $\frac{1}{q_1}$};
\draw[->] (0,0) -- (0,4.5) node[left] { \scriptsize $\frac{1}{q_2}$};

\draw (0,4) node{\scriptsize $+$} node[left] {\scriptsize $1$} ;
\draw (4,0) node{\scriptsize $+$};

\draw
(0, 2) node[left] {\scriptsize \textbf{a}} -- (2, 2) ;

\draw[very thin] (0,2)  node[left] {\scriptsize $\textbf{a}$}
 -- (1, 0)   node[below] {\scriptsize $\textbf{b}$};
\draw[thick] (0,2) circle(0.05);
\draw[thick] (1,0) circle(0.05);

\draw(2, 2) node{\scriptsize $\bullet$} node[above] {\scriptsize \textbf{D}};

\draw(0.7,2) node{\scriptsize $\bullet$} node[above]{\scriptsize \textbf{E}};

\draw (1.58, 1.15) node{\scriptsize $\bullet$} node[right]{\scriptsize \textbf{C}};

\draw(0.7, 1.15) node{\scriptsize $\bullet$} node[above]{\scriptsize \textbf{G}};

\draw[very thin] (1, 0)  -- (2, 2);

\draw (0.78,0.85) node{\scriptsize $\blacklozenge$} node[left]{\scriptsize {$\textbf{g}$}};
\draw[thick, dashed] (0, 2.1)--(1.3, 0);
\draw (0.5, 0.3) node{\scriptsize { $\mathcal{S}_{m, q}^{(q_1,q_2)}$}} ;


\draw[ ->] (0.78,0.85) -- (1.15, 0.85)  ;
\draw[ ->] (1.15, 0.9) -- (1.15, 1.2) ;

\draw (1.15, 1.33) node{\scriptsize $\bigstar$} node[above]{\scriptsize $\textbf{h}$ };
\draw[thick] (0, 2.7) -- (2.2, 0) ;

\draw (2.3, 0.65) node{\scriptsize { $\mathcal{S}_{m, q^\ast}^{(q_1^\ast,q_2^\ast)}$}} ;

\draw[thin] (0,3.5) node {\scriptsize $\times$} node[left] {\scriptsize $\frac{1+d}{2+d}$}
-- (4.6, 0) node {\scriptsize $\times$} node[below] {\scriptsize $\frac{1+d}{d}$} ;
\draw (4.5, 0.65) node {\scriptsize $\mathcal{S}_{m,m-1}^{(q_1, q_2)}$} ;


\draw (6, 4.3) node[right] {\scriptsize $\textbf{a, b, A-E}$: same as in Fig.~\ref{F:S:m2:more} in Appendix~\ref{Appendix:fig} };
\draw (6, 3.5) node[right] {\scriptsize $\mathcal{R}(\textbf{GCDE})$: scaling invariant class of \eqref{T:ACweakSol:V_q} };
\draw (6, 3) node[right] {\scriptsize $\mathcal{R}(\textbf{GAbC}))$: scaling invariant class of \eqref{T:weakSol:V_q} };
\draw (6,2) node[right] {\scriptsize $\blacklozenge$  $ g = (\frac{1}{q_1}, \frac{1}{q_2})$ on $\mathcal{S}_{m, q}^{(q_1, q_2)}$ (dashed line)} ;
\draw (6,1.5) node[right] {\scriptsize $\bigstar$  $ h = (\frac{1}{q_1^\ast}, \frac{1}{q^{\ast}_2})$ on $\mathcal{S}_{m, q^\ast}^{(q_1^\ast, q_2^\ast)}$ (thick line)} ;
\draw (6,1) node[right] {\scriptsize for $q^{\ast}_1 \leq q_1$, $q^{\ast}_2 \leq q_2$, $q^\ast \leq q$} ;

\end{tikzpicture}
\caption{\footnotesize Theorem~\ref{T:ACweakSol} for $m>2$, $q\geq m-1$, $ 2 < d \leq \frac{2m}{m-1}$.}
\label{F:S:m2}
\end{figure}

\item[(iv)] To sum up the existence results in this section, we deliver Table~\ref{Table3}. 

\begin{table}[hbt!]
\begin{center}
\caption{\footnotesize Guide of Theorem~\ref{T:weakSol} and Theorem~\ref{T:ACweakSol} in Section~\ref{SS:energy-sol}.}
\smallskip
{\scriptsize
\begin{tabular}{| c  || c | c | c |}\hline
\rule[-8pt]{0pt}{22pt}
 \textbf{Structure of $V$} & \textbf{Range of $m$} & \textbf{Range of $q$} & \textbf{References}   \\ \hline \hline

\rule[-8pt]{0pt}{22pt}
\multirow{9}{*}{$V\in \mathfrak{S}_{m,q}^{(q_1, q_2)}$}
& {$1<m\leq 2$}
& \multirow{3}{*}{$q=1$}
& Theorem~\ref{T:weakSol} (i), Theorem~\ref{T:ACweakSol} (i)
 \\
\cline{2-2} \cline{4-4}

\rule[-8pt]{0pt}{22pt}

& {$ m>2$}
& 
& open
\\
\cline{2-4}

\rule[-8pt]{0pt}{22pt}
& $1 < m \leq 2$
& $1 < q < \infty$
& \multirow{3}{*}{Theorem~\ref{T:weakSol} (ii), Theorem~\ref{T:ACweakSol} (ii)}
\\
\cline{2-3}

\rule[-8pt]{0pt}{22pt}

&\multirow{3}{*}{$m>2$}
& $m-1 \leq q < \infty$
& 
\\
\cline{3-4}

\rule[-8pt]{0pt}{22pt}
& 
& $1 < q < m-1$
& open
\\
\hline

\end{tabular}
\label{Table3}
}
\end{center}
\end{table}

\end{itemize}
\end{remark}


\subsubsection{Existence for case: $\nabla\cdot V \geq 0$.}\label{SS:divergence-free}

Here we assume that the divergence of $V$ is non-negative, including the divergence-free case. Compared to the earlier subsection, theorems in this subsection work for a less restricted range for $V$. Let us begin with the $L^1$-weak solution.

\begin{theorem}\label{T:weakSol_DivFree_1}
Let $m>1$ and  $\nabla \cdot V \geq 0$. Furthermore, assume that $\rho_0 \in  \mathcal{P} (\Omega)$, $\int_{\Omega} \rho_0 \log \rho_0 \,dx< \infty$, and
\begin{equation}\label{T:weakSol_DivFree:V_1}
V\in \mathfrak{S}_{m,1}^{(q_1,q_2)} \quad \text{for} \quad
\begin{array}{l l}
\begin{cases}
 q_2^{1,l} \leq q_2 \leq \infty,
 & \text{if } 1 < m \leq m^\ast, \vspace{1 mm} \\
q_2^{1,l}\leq q_2 \leq q_{2}^{1,u} ,
& \text{if }  m > m^\ast, \vspace{1 mm} \\
\end{cases}
\text{ and } d>2 ,\vspace{2 mm} \\
\begin{cases}
 q_2^{1,l} \leq q_2 \leq \infty,
 & \text{if }  1<m<m^\ast, \vspace{1 mm} \\
  q_2^{1,l}\leq q_2 < \infty,
 & \text{if }  m = m^\ast, \vspace{1 mm} \\
q_2^{1,l}\leq q_2 < q_{2}^{1,u},
& \text{if }  m > m^\ast, \vspace{1 mm}
\end{cases}
 \text{ and } d=2,
\end{array}
\end{equation}
where
\begin{equation}\label{m_ast}
q_{2}^{1, l}:= \frac{2+d(m-1)}{1+d(m-1)}, \quad  q_2^{1,u}:= \left(\frac{md+1}{md+2} - \frac{1}{m}\right)^{-1},\quad m^\ast := \frac{\sqrt{d^2 + 6d +1} + (d-1)}{2d}.
\end{equation}
Then, there exists
  a nonnegative $L^1$-weak solution of \eqref{E:Main}-\eqref{E:Main-bc-ic} in Definition~\ref{D:weak-sol} that holds \eqref{T:weakSol:E_1}
where $C = C(\int_{\Omega} \rho_0 \log \rho_0 \,dx)$.
\end{theorem}

\begin{remark}
\begin{itemize}
\item[(i)] The constant $m^\ast$ is determined by the intersection of the valid regions where both energy estimates and compactness arguments hold. In Fig.~\ref{F:divfree:1}, $m^\ast$ is determined by the point $\textbf{C}$, the intersection of the right-end point of $\mathcal{S}_{m^\ast, 1}^{(q_1, q_2)}$ and the line $l$ from \eqref{V_compact_m+1}. 
The $m^\ast$ in \eqref{m_ast} decreases to $1$ as $d\to \infty$. 
\item[(ii)] Fig.~\ref{F:divfree:1} is for \eqref{T:weakSol_DivFree:V_1} when $q=1$, $d>2$ and $m>m^\ast$. When $d=2$, $\overline{\textbf{CD}}$ is excluded from $\mathcal{R}(\textbf{ABCDE})$. 

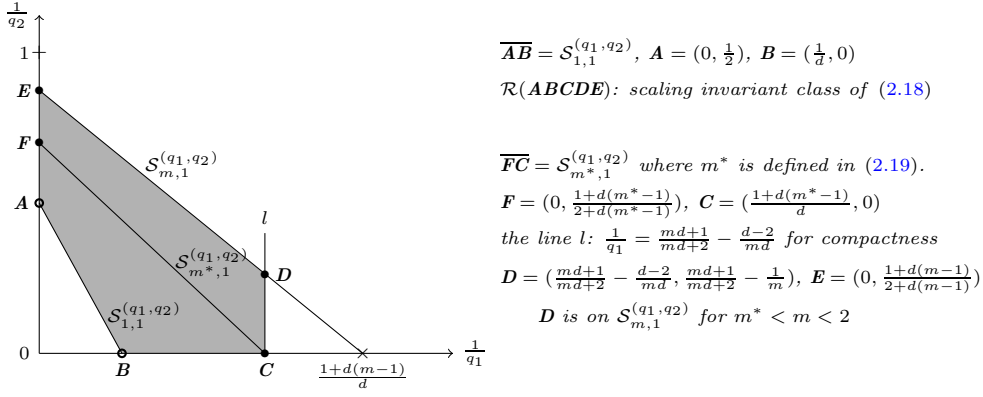
\begin{figure}
\centering
\begin{tikzpicture}[domain=0:16]


\fill[fill= lgray]
(0, 2) -- (1.1,0) -- (3, 0) -- (3, 1.05) -- (0, 3.5) -- (0, 2);

\draw[->] (0,0) node[left] {\scriptsize $0$}
-- (5.5,0) node[right] {\scriptsize $\frac{1}{q_1}$};
\draw[->] (0,0) -- (0,4.5) node[left] { \scriptsize $\frac{1}{q_2}$};

\draw (0,4) node{\scriptsize $+$} node[left]{\scriptsize $1$} ;

\draw[thick] (0, 2) circle(0.05) node[left] {\scriptsize \textbf{A}};
\draw (1.4, 0.5) node{\scriptsize $\mathcal{S}_{1,1}^{(q_1, q_2)}$};

\draw (0,2) -- (1.1, 0)  node[below] {\scriptsize $\textbf{B}$};
\draw[thick] (1.1, 0) circle(0.05) ;


\draw (0,2.8) node {\scriptsize $\bullet$}  node[left]{\scriptsize \textbf{F}}
    -- (3, 0) ;
\draw[thick] (3, 0) node {\scriptsize $\bullet$} node[below] {\scriptsize \textbf{C}};
\draw (2.3, 1.2) node{\scriptsize $\mathcal{S}_{m^\ast,1}^{(q_1, q_2)}$};

\draw (0,3.5) node {\scriptsize $\bullet$} node[left] {\scriptsize \textbf{E} }
    -- (4.3, 0) node{\scriptsize $\times$} node[below] {\scriptsize $\frac{1+d(m-1)}{d}$} ;
\draw (1.9, 2.5) node{\scriptsize $\mathcal{S}_{m,1}^{(q_1, q_2)}$};
\draw[thick] (3,1.05) node {\scriptsize $\bullet$} node[right]{\scriptsize \textbf{D}};
\draw (3, 0) -- (3, 1.6) node[above]{\scriptsize $l$};



\draw (6, 4) node[right] {\scriptsize $\overline{\textbf{AB}} = \mathcal{S}_{1, 1}^{(q_1, q_2)}$, $\textbf{A} = (0, \frac 12)$, $\textbf{B} = (\frac 1d, 0)$ };

\draw (6, 3.5) node[right] {\scriptsize $\mathcal{R} (\textbf{ABCDE})$: scaling invariant class of \eqref{T:weakSol_DivFree:V_1}};

\draw (6, 2.5) node[right] {\scriptsize $\overline{\textbf{FC}} = \mathcal{S}_{m^\ast, 1}^{(q_1, q_2)}$ where $m^\ast$ is defined in \eqref{m_ast}. };
\draw (6, 2) node[right] {\scriptsize $\textbf{F} = (0, \frac{1+d(m^\ast -1)}{2+d(m^\ast -1)})$, $\textbf{C} = (\frac{1+d(m^\ast -1)}{d}, 0)$ };
\draw (6, 1.5) node[right]{\scriptsize the line $l$: $\frac{1}{q_1} = \frac{md+1}{md+2}-\frac{d-2}{md}$ for compactness};
\draw (6, 1) node[right] {\scriptsize $\textbf{D} = (\frac{md+1}{md+2} - \frac{d-2}{md}, \frac{md+1}{md+2} - \frac{1}{m})$, $\textbf{E}= (0, \frac{1+d(m-1)}{2 + d(m-1)})$};
\draw (6.5, 0.5) node[right]{\scriptsize $\textbf{D}$ is on $\mathcal{S}_{m, 1}^{(q_1, q_2)}$ for $m^\ast <m < 2$};
\end{tikzpicture}
\caption{\footnotesize Theorem~\ref{T:weakSol_DivFree_1} for $ m> m^\ast$, $q=1$, $d>2$.}
\label{F:divfree:1}
\end{figure}
\end{itemize}
\end{remark}

Now we introduce $L^q$-weak solutions satisfying an energy estimate in two cases: either $m$ close to $1$ or not. 

\begin{theorem}\label{T:weakSol_DivFree_q}
Let $m>1$, $q>1$, and $q_{m,d}= \frac{d(m-1)}{q}$. Furthermore, assume that $\nabla \cdot V \geq 0$ and $\rho_0 \in  \mathcal{P} (\Omega) \cap L^{q} (\Omega)$. 
\begin{itemize}
\item[(i)] Let $1 < m \leq m^{\ast}$ for $m^\ast$ in \eqref{m_ast}. Suppose that, for $q_2^{q,l}:= \frac{2+q_{m,d}}{1+q_{m,d}}$,
 \begin{equation}\label{T:weakSol_DivFree:V_q1}
V\in \mathfrak{S}_{m,q}^{(q_1,q_2)} \quad \text{for} \quad
\begin{cases}
q_2^{q,l} \leq q_2 \leq \infty, &\text{if } d>2, \vspace{1 mm} \\
 q_2^{q,l} \leq q_2 \leq \infty,
 & \text{if }  1<m<m^\ast \text{ and } d=2, \vspace{1 mm} \\
  q_2^{q,l}\leq q_2 < \infty,
 & \text{if }  m = m^\ast \text{ and } d=2 .
\end{cases}
\end{equation}
Then, there exists
  a nonnegative $L^q$-weak solution of \eqref{E:Main}-\eqref{E:Main-bc-ic} in Definition~\ref{D:weak-sol} that holds \eqref{T:weakSol:E_q}
where $C = C( \|\rho_0\|_{L^{q}(\Omega)})$.

\item[(ii)] Let $m>m^\ast$. Suppose that
\begin{equation}\label{T:weakSol_DivFree:V_q2}
V\in \mathfrak{S}_{m,q}^{(q_1,q_2)} \quad \text{for} \quad
\begin{array}{l l}
\begin{cases}
 q_2^{q,l} \leq q_2 \leq q_{2}^{q,u},
 & \text{if }  1<q\leq q^\ast  , \vspace{1 mm} \\
q_2^{q,l}\leq q_2 \leq \infty,
& \text{if }   q > q^\ast, \vspace{1 mm} \\
\end{cases}
\text{ and } d>2,\vspace{2 mm} \\
\begin{cases}
 q_2^{q,l} \leq q_2 < q_{2}^{q,u},
 & \text{if }   1<q<q^\ast, \vspace{1 mm} \\
  q_2^{q,l}\leq q_2 < \infty,
 & \text{if }   q=q^\ast, \vspace{1 mm} \\
q_2^{q,l}\leq q_2 \leq \infty,
& \text{if }  q > q^\ast, \vspace{1 mm}
\end{cases}
 \text{ and } d=2,
\end{array}
\end{equation}
where
\begin{equation}\label{q_ast}
q_2^{q,l}:= \frac{2+q_{m,d}}{1+q_{m,d}}, \ q_{2}^{q,u}:= \left(\frac{1+q_{m,d}}{2+q_{m,d}} - \frac{d(md+1)}{md+2} +\frac{d-2}{m}\right)^{-1}, \  q^\ast :=\left[ \frac{1}{m-1} \left(\frac{md+1}{md+2} - \frac{m+d-2}{md}\right)\right]^{-1}.
\end{equation}
Then, there exists
  a nonnegative $L^q$-weak solution of \eqref{E:Main}-\eqref{E:Main-bc-ic} in Definition~\ref{D:weak-sol} that holds \eqref{T:weakSol:E_q}
where $C = C( \|\rho_0\|_{L^{q}(\Omega)})$.
\end{itemize}
\end{theorem}

We make a few remarks about restrictions on $V$ in the above theorem.
\begin{remark}\label{R:T:weakSol_DivFree}
\begin{itemize}

\item[(i)] We plot the range of $(q_1, q_2)$ satisfying \eqref{T:weakSol_DivFree:V_q2} in Fig.~\ref{F:divfree:q}. 
The line $l$ is the restriction from \eqref{V_compact_m+1} and \eqref{V_compact_m-1} with $q=1$ for carrying compactness arguments. 

\item[(ii)] The constant $q^\ast$ is determined by computing the intersection of the right-end point of $S_{m,q^\ast}^{(q_1, q_2)}$ and the line $l$. The $q^\ast$ tends to $m$ as $d$ to $\infty$, and $q^\ast \to \infty$ as $m \to \infty$. 

\begin{figure}

\centering
\begin{tikzpicture}[domain=0:16]


\fill[fill= lgray]
(0, 2) -- (1.1,0) -- (3, 0) -- (3, 1.05) -- (0, 3.5) -- (0, 2);

\draw[->] (0,0) node[left] {\scriptsize $0$}
-- (5.5,0) node[right] {\scriptsize $\frac{1}{q_1}$};
\draw[->] (0,0) -- (0,4.5) node[left] { \scriptsize $\frac{1}{q_2}$};

\draw (0,4) node{\scriptsize $+$} node[left]{\scriptsize $1$} ;

\draw[thick] (0, 2) circle(0.05) node[left] {\scriptsize \textbf{A}};
\draw (1.4, 0.4) node{\scriptsize $\mathcal{S}_{m,\infty}^{(q_1, q_2)}$};

\draw (0,2) -- (1.1, 0)  node[below] {\scriptsize $\textbf{B}$};
\draw[thick] (1.1, 0) circle(0.05) ;


\draw[thick] (0,2.7) node {\scriptsize $\bullet$}  node[left]{\scriptsize \textbf{F}}
    -- (3, 0) ;
\draw[thick] (3, 0) node {\scriptsize $\bullet$} node[below] {\scriptsize \textbf{C}};
\draw (2.2, 1.2) node{\scriptsize $\mathcal{S}_{m,q^\ast}^{(q_1, q_2)}$};

\draw (0,3.5) node {\scriptsize $\bullet$} node[left] {\scriptsize \textbf{E} }
    -- (4.3, 0) node{\scriptsize $\times$} node[below] {\scriptsize $\frac{1+d(m-1)}{d}$} ;
\draw (2, 2.5) node{\scriptsize $\mathcal{S}_{m,1}^{(q_1, q_2)}$};
\draw[thick] (3, 1.05) node {\scriptsize $\bullet$} node[right]{\scriptsize \textbf{D}};


\draw[thick] (3, 0) -- (3, 1.6) node[above]{\scriptsize $l$};

\draw (6, 4) node[right] {\scriptsize $\overline{\textbf{AB}} = \mathcal{S}_{m, \infty}^{(q_1, q_2)}$, $\textbf{A} = (0, \frac 12)$, $\textbf{B} = (\frac 1d, 0)$ };

\draw (6, 3.5) node[right] {\scriptsize $\mathcal{R} (\textbf{ABCDE})$: scaling invariant class of \eqref{T:weakSol_DivFree:V_1}};

\draw (6, 2.5) node[right] {\scriptsize $\overline{\textbf{FC}} = \mathcal{S}_{m, q^\ast}^{(q_1, q_2)}$, $\textbf{F} = (0, \frac{1+q^{\ast}_{m,d}}{2+q^{\ast}_{m,d}})$, $\textbf{C} = (\frac{1+q^{\ast}_{m,d}}{d}, 0)$ };
\draw (6.5, 2) node[right] {\scriptsize $q^\ast$ is defined in \eqref{q_ast} and $q^\ast_{m,d}=\frac{d(m-1)}{q^\ast}$};
\draw (6,1.5) node[right]{\scriptsize the line $l$: $\frac{1}{q_1} = \frac{md+1}{md+2}-\frac{d-2}{md}$ for compactness};
\draw (6, 1) node[right] {\scriptsize $\textbf{D} = (\frac{md+1}{md+2} - \frac{d-2}{md}, \frac{md+1}{md+2} - \frac{1}{m})$, $\textbf{E}= (0, \frac{1+d(m-1)}{2 + d(m-1)})$};
\end{tikzpicture}
\caption{\footnotesize Theorem~\ref{T:weakSol_DivFree_q} for $ m> m^\ast$, $q>1$, $d>2$.}
\label{F:divfree:q}
\end{figure}
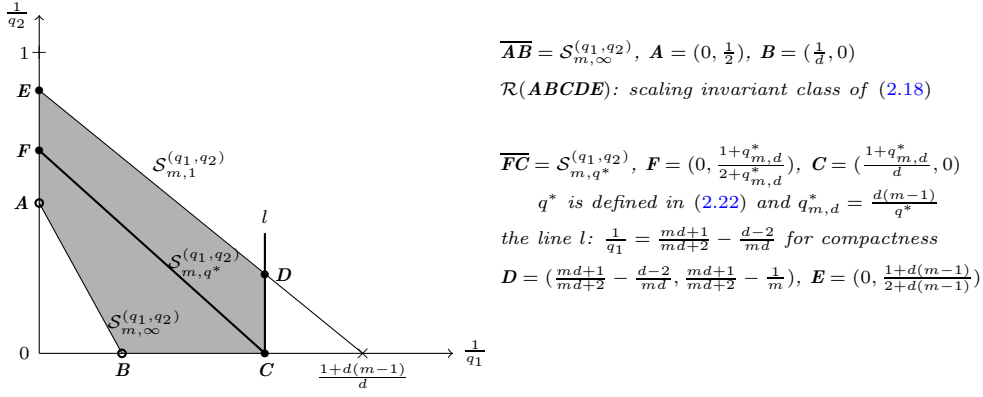

\item[(iii)] Compared to parallel results in \cite{HKK}, conditions on $V$ are less restricted because of improved compactness arguments (both Proposition~\ref{P:AL1} and Proposition~\ref{P:AL2}). Compactness arguments are applicable for unbounded domain results (see Appendix~\ref{Appendix:compact}). 
\end{itemize}
\end{remark}

Now we construct absolutely continuous $L^q$-weak solutions for $q\geq 1$ with the aid of speed estimate in Lemma~\ref{P:W-p}. The following theorem is parallel to \cite[Theorem~2.13, Figs.~8, 9]{HKK}, corresponding results on $\mathbb{R}^d \times (0, T)$. 

\begin{theorem}\label{T:ACweakSol_DivFree}
Let $m>1$ and $q\geq 1$. Furthermore, assume that $\nabla \cdot V \geq 0$.

 \begin{itemize}
 \item[(i)] Suppose that $\rho_0 \in  \mathcal{P}(\Omega)$ and $\int_{\Omega} \rho_0 \log \rho_0 \,dx < \infty$, and
 \begin{equation} \label{T:ACweakSol_DivFree:V_1}
V\in \mathfrak{S}_{m,1}^{(q_1, q_2)}  \ \text{ for } \
\begin{cases}
\lambda_1 \leq q_2 \leq \frac{\lambda_1 m}{m-1}, & \text{ if }  d>2, \vspace{1 mm}\\
\lambda_1 \leq q_2 < \frac{\lambda_1 m}{m-1}, & \text{ if }  d=2.
\end{cases}
\end{equation}
Then, there exists a nonnegative $L^1$-weak solution of \eqref{E:Main}-\eqref{E:Main-bc-ic} in Definition~\ref{D:weak-sol} such that
 $\rho \in AC(0,T; \mathcal{P}(\Omega))$ with $\rho(\cdot, 0)=\rho_0$.
Furthermore, $\rho$
satisfies \eqref{T:weakSol:E_1}, \eqref{T:ACweakSol:E_1}, and \eqref{T:ACweakSol:W_1} with
$C=C(\left\| V\right\|_{\mathfrak{S}_{m,1}^{(q_1, q_2)}},\, \int_{\Omega} \rho_0 \log \rho_0 \,dx )$.

 \item[(ii)] Let $q > 1$. Suppose that $\rho_0 \in  \mathcal{P}(\Omega) \cap L^{q} (\Omega)$ and
\begin{equation}\label{T:ACweakSol_DivFree:V_q}
V\in \mathfrak{S}_{m,q}^{(q_1, q_2)}  \ \text{ for } \
\begin{cases}
\lambda_q \leq q_2 \leq \frac{\lambda_q (q+m-1)}{q+m-2}, & \text{ if }  1 < q\leq m, \text{ and } d>2 \vspace{1 mm}\\
q_1 \leq \frac{2m}{m-1},\  2 \leq q_2 \leq \frac{2m-1}{m-1}, & \text{ if }  q> m,\text{ and } d>2 \vspace{1 mm}\\
\lambda_q \leq q_2 < \frac{\lambda_q (q+m-1)}{q+m-2}, & \text{ if }  1 < q\leq m, \text{ and } d=2 \vspace{1 mm}\\
q_1 \leq \frac{2m}{m-1},\  2 \leq q_2 < \frac{2m-1}{m-1}, & \text{
if }   q> m, \text{ and } d=2.
\end{cases}
\end{equation}
Then, there exists a nonnegative $L^q$-weak solution of \eqref{E:Main}-\eqref{E:Main-bc-ic} in Definition~\ref{D:weak-sol} such that
 $\rho \in AC(0,T; \mathcal{P} (\Omega))$ with $\rho(\cdot, 0)=\rho_0$.
Furthermore, $\rho$
satisfies \eqref{T:weakSol:E_q}, \eqref{T:ACweakSol:E_q}, and \eqref{T:ACweakSol:W_q}  with
$C=C( \left\| V\right\|_{\mathfrak{S}_{m,q}^{(q_1, q_2)}},\, \|\rho_0\|_{L^{q}(\Omega)} )$.

\item[(iii)] (Embedding) Let $q > 1$. Suppose that $\rho_0 \in  \mathcal{P}(\Omega) \cap L^{q} (\Omega)$ and
\begin{equation}\label{T:ACweakSol_DivFree:V_q:Embedding}
V\in \mathfrak{S}_{m,q}^{(q_1, q_2)} \ \text{ for } \
 \begin{cases}
\frac{2+q_{m,d}}{1+q_{m,d}} < q_2 \leq \frac{\lambda_1 m}{m-1}, & \text{ if } 1 < q \leq \frac{md}{d-1}, \text{ and } d>2, \vspace{1 mm}\\
\frac{2+q_{m,d}}{1+q_{m,d}} < q_2 < \frac{\lambda_1 m}{m-1}, & \text{ if } 1 < q \leq \frac{md}{d-1},\text{ and } d=2 \vspace{1 mm}\\
\frac{2+q_{m,d}}{1+q_{m,d}} < q_2 \leq \infty, & \text{ if } q >
\frac{md}{d-1}, \text{ and } d\geq 2.
\end{cases}
\end{equation}
For given $(q_1, q_2)$ in \eqref{T:ACweakSol_DivFree:V_q:Embedding}, there
exists constants $q_{2}^{\ast} = q_{2}^{\ast}(q_1) \in [2, q_2]$
such that $ V \in \mathfrak{S}_{m, q^\ast}^{(q_1, q_2^{\ast})}$
satisfying \eqref{T:ACweakSol_DivFree:V_q}.
 Furthermore, the same conclusion
holds as in (ii) except that $\lambda_q$ is replaced by $\lambda_{q^\ast}$.
\end{itemize}
\end{theorem}

\subsubsection{Existence for case: $\nabla V$ in sub-scaling classes }\label{SS:tilde Serrin}

Now we assume that $\nabla V$ is satisfying sub-scaling classes.
The first result is concerned with $L^q$-weak solutions with an energy inequality.

\begin{theorem}\label{T:weakSol_tilde}
Let $m>1$ and $q\geq 1$.

\begin{itemize}
\item[(i)] Suppose that $\rho_0 \in  \mathcal{P}(\Omega)$ and $\int_{\Omega} \rho_0 \log \rho_0 \,dx < \infty$, and
 \begin{equation}\label{T:weakSol_tilde:V_1}
V\in \tilde{\mathfrak{S}}_{m,1}^{(\tilde{q}_1, \tilde{q}_2)}  \ \text{ for } \
\begin{cases}
 \frac{2+d(m-1)}{1+d(m-1)} < \tilde{q}_2 \leq \frac{m}{m-1},    &\text{ if } d \geq 2, \vspace{2 mm}\\
  \frac{2+d(m-1)}{1+d(m-1)} < \tilde{q}_2 < \frac{m}{m-1},   & \text{ if } d= 2.
 \end{cases}
\end{equation}
Then, there exists a nonnegative $L^1$-weak solution of \eqref{E:Main}-\eqref{E:Main-bc-ic} in Definition~\ref{D:weak-sol} satisfying \eqref{T:weakSol:E_1} with 
$C = C(\left\| V\right\|_{\tilde{\mathfrak{S}}_{m,q}^{(\tilde{q}_1, \tilde{q}_2)}}, \int_{\Omega} \rho_0 \log \rho_0 \,dx )$.

\item[(ii)] Let $q >1$. Assume that $\rho_0 \in  \mathcal{P} (\Omega) \cap L^{q} (\Omega)$ .
Suppose that
 \begin{equation}\label{T:weakSol_tilde:V_q}
V\in \tilde{\mathfrak{S}}_{m,q}^{(\tilde{q}_1, \tilde{q}_2)}  \ \text{ for } \
\begin{cases}
 \frac{2+q_{m,d}}{1+q_{m,d}} < \tilde{q}_2 \leq \frac{q+m-1}{m-1},   & \text{ if } d> 2, \vspace{1 mm} \\
\frac{2+q_{m,d}}{1+ q_{m,d}} < \tilde{q}_2 < \frac{q+m-1}{m-1},   & \text{ if } d= 2.
 \end{cases}
\end{equation}
Then, there exists a nonnegative $L^q$-weak solution of \eqref{E:Main}-\eqref{E:Main-bc-ic} in Definition~\ref{D:weak-sol} satisfying \eqref{T:weakSol:E_q} with
$C = C(\left\| V\right\|_{\tilde{\mathfrak{S}}_{m,q}^{(\tilde{q}_1, \tilde{q}_2)}}, \|\rho_0\|_{L^{q}(\Omega)} )$.
\end{itemize}

\end{theorem}

Here is a remark about Theorem~\ref{T:weakSol_tilde}.
\begin{remark}
The existence results in Theorem~\ref{T:weakSol_tilde} (ii) is the same as a result in an unbounded domain \cite[Theorem~2.15]{HKK}. However, Theorem~\ref{T:weakSol_tilde} (i) is new because $p$-moment estimate is trivial in the bounded domain.  We refer \cite[Figure~10]{HKK} for all $m>1$ and $q>1$, where the line $\overline{\textbf{DE}}$ is for \eqref{T:weakSol_tilde:V_1} and the region $\mathcal{R}(\textbf{AbDE})$ is for \eqref{T:weakSol_tilde:V_q}. 
\end{remark}

The second result is the existence of a $L^q$-weak solution as an absolutely continuous curve in Wasserstein space that is the same as \cite[Theorem~2.16]{HKK}.

\begin{theorem}\label{T:ACweakSol_tilde}
Let $m>1$ and $q \geq 1$.

  \begin{itemize}
  \item[(i)] Assume that $\rho_0 \in  \mathcal{P}(\Omega)$, and $\int_{\Omega} \rho_0 \log \rho_0 \,dx < \infty$. Furthermore, suppose that $V \in \tilde{\mathfrak{S}}_{m,1}^{(\tilde{q}_1, \tilde{q}_2)}$ satisfies \eqref{T:weakSol_tilde:V_1}.
Then, there exists a nonnegative $L^1$-weak solution of \eqref{E:Main}-\eqref{E:Main-bc-ic} in Definition~\ref{D:weak-sol} such that  $\rho \in AC(0,T; \mathcal{P}(\Omega))$ with $\rho(\cdot, 0)=\rho_0$.
Furthermore, $\rho$ satisfies \eqref{T:weakSol:E_1},
\eqref{T:ACweakSol:E_1} and \eqref{T:ACweakSol:W_1} with $C=C
(\left\| V\right\|_{\tilde{\mathfrak{S}}_{m,1}^{(\tilde{q}_1,
\tilde{q}_2)}}, \int_{\Omega} \rho_0 \log \rho_0 \,dx )$.

 \item[(ii)] Let $q>1$. Assume that $\rho_0 \in  \mathcal{P} (\Omega)\cap L^{q} (\Omega)$. Furthermore suppose that 
 \begin{equation}\label{T:ACweakSol_tilde:V_q}
V\in \tilde{\mathfrak{S}}_{m,q}^{(\tilde{q}_1, \tilde{q}_2)}  \
\text{ for } \
\begin{cases}
 \frac{2+q_{m,d}}{1+q_{m,d}} < \tilde{q}_2 \leq \frac{q+m-1}{m-1} \,\delta_{\{1 < q \leq m\}} + \frac{2m-1}{m-1}\,\delta_{ \{q > m \} },   & \text{ if } d> 2, \vspace{1 mm}\\
 \frac{2+q_{m,2}}{1+q_{m,2}} < \tilde{q}_2 < \frac{q+m-1}{m-1} \,\delta_{\{1 < q \leq m\}} + \frac{2m-1}{m-1}\,\delta_{ \{q > m \} },   & \text{ if } d= 2.
 \end{cases}
\end{equation}
Then,
there exists a nonnegative $L^q$-weak solution of \eqref{E:Main}-\eqref{E:Main-bc-ic} in Definition~\ref{D:weak-sol} such that  $\rho \in AC(0,T; \mathcal{P}(\Omega))$ with $\rho(\cdot, 0)=\rho_0$.
Furthermore, $\rho$ satisfies \eqref{T:weakSol:E_q}, \eqref{T:ACweakSol:E_q}, and \eqref{T:ACweakSol:W_q} with $C=C
(\left\| V\right\|_{\tilde{\mathfrak{S}}_{m,q}^{(\tilde{q}_1,
\tilde{q}_2)}}, \|\rho_0\|_{L^{q}(\Omega)} )$.

\item[(iii)] (Embedding) Let $q>1$. Suppose that 
\begin{equation}\label{T:ACweakSol_tilde:V_q:Embedding}
V\in \tilde{\mathfrak{S}}_{m,q}^{(\tilde{q}_1, \tilde{q}_2)}  \
\text{ for } \
\begin{cases}
\frac{2m-1}{m-1} < \tilde{q}_2 \leq \frac{q+m-1}{m-1},   & \text{ if } d> 2, \vspace{1 mm}\\
\frac{2m-1}{m-1} \leq \tilde{q}_2 < \frac{q+m-1}{m-1},  & \text{ if
} d= 2 .
\end{cases}
\end{equation}
For given $(\tilde{q}_1, \tilde{q}_2)$ in
\eqref{T:ACweakSol_tilde:V_q:Embedding}, there exists a constant
$\tilde{q}_{2}^{\ast} = \tilde{q}_{2}^{\ast}(\tilde{q}_1) \in
(\frac{2+q^{\ast}_{m,d}}{1+q^{\ast}_{m,d}}, \tilde{q}_2]$ such that
$ V \in \tilde{\mathcal{S}}_{m, q^\ast}^{(\tilde{q}_1,
\tilde{q}_2^{\ast})}$ satisfying \eqref{T:ACweakSol_tilde:V_q}.
Furthermore,
the same conclusion holds as in Theorem~\ref{T:ACweakSol_tilde} (ii) except
that $\lambda_q$ is replaced by $ \lambda_{q^{\ast}}$.
\end{itemize}
\end{theorem}

\subsection{An application}\label{SS:application}

As an application, we consider the following coupled system, which is a Keller-Segel  equations of consumption type.
\begin{equation}\label{PME-KS-10}
\rho_t -\Delta \rho^m =-\chi \nabla (\rho \nabla c),\qquad
c_t-\Delta c =-\rho c,\qquad \Omega\times (0,T),
\end{equation}
where $\Omega\subset\R^d$, $d\ge 2$, is a bounded convex domain with smooth boundaries.
Boundary conditions are of no-flux, namely
\begin{equation}\label{PME-KS-bc}
\frac{\partial \rho}{\partial \textbf{n}}=\frac{\partial c}{\partial \textbf{n}}=0, \qquad \partial\Omega\times (0,T),
\end{equation}
which satisfy \eqref{E:Main-bc-ic}.
We are interested in existence of weak solutions of \eqref{PME-KS-10}-\eqref{PME-KS-bc}, in case that non-negative initial data of $\rho$ and $c$ satisfy
\begin{equation}\label{PME-KS-ic}
\norm{\rho_0}_{L^{1}(\Omega)\cap L^{\infty}(\Omega)}+\norm{c_0}_{L^{\infty}(\Omega)\cap H^1(\Omega)}<\infty.
\end{equation}

It is known in \cite{WX15} that if $m>\frac{3d-2}{2d}$, $d\ge 3$, then there exists a bounded weak solution for \eqref{PME-KS-10}-\eqref{PME-KS-ic}. In the case that $d=3$, such a result was improved in \cite{W18} as $m>\frac{9}{8}$ even for the system coupled with fluid equations (see \cite{CHKK17} for H\"{o}lder continuous solutions in $\R^3$). 
In the case of $d=2$, it was proved in \cite{TW12} that a bounded weak solution exists, provided that $m>1$ (see also \cite{CKK14} for the whole space $\R^2$). One consequence of our main results is  existence of $L^q$-weak solutions for the case $1+\frac{d-2}{3d}\le m \le \frac{3d-2}{2d}$, $d\ge 3$.
We first define the notion of $L^q$-weak solutions of \eqref{PME-KS-10}-\eqref{PME-KS-ic}.
\begin{defn}
Let $1<m \le \frac{3d-2}{2d}$, $d\ge 3$ and $1<q<\infty$.
We say that $(\rho, c)$ is a weak solution of the system \eqref{PME-KS-10}-\eqref{PME-KS-ic}, if the followings are satisfied: 
\begin{itemize}
\item[(i)] $\rho$ and $c$ are non-negative functions defined in $\Omega=\Omega\times (0, T)$ such that
\[
\displaystyle\rho \in L^{\infty}(0, T; (L^1\cap L^{q})(\Omega)),\quad \nabla \rho^{\frac{m+l-1}{2}} \in L^2(\Omega_{T}),\quad 1\le l\le q,
\]
\[
\displaystyle c\in L^{\infty} (0,T; (L^{\infty}\cap H^1)(\Omega))\cap L^2(0, T; H^2(\Omega)).
\]
\item[(ii)] $(\rho, c)$ satisfies \eqref{PME-KS-10} in the sense of distributions, i.e.
\[
\displaystyle \int_0^T\int_{\Omega}(-\rho\varphi_t+\nabla \rho^{m}\cdot\nabla \varphi-\chi \rho \nabla c\cdot\nabla\varphi ) dx dt=\int_{\Omega} \rho_0 \varphi(\cdot, 0)dx, 
\]
\[
\displaystyle \int_0^T\int_{\Omega}(-c\varphi_t+\nabla c\cdot\nabla \varphi+ \rho c\varphi ) dx dt=\int_{\Omega} c_0 \varphi(\cdot, 0)dx
\]
for all test function $\varphi\in\calC^{\infty}_0 (\Omega\times [0,T))$.
\end{itemize}
\end{defn}

We call $(\rho, c)$ a {\it bounded weak solution}, if  $(\rho, c)$ is a $L^q$-weak solution for any $q<\infty$ and $\rho\in L^{\infty}(\Omega_T)$. For simplicity, we assume that $\chi = 1$.
We remind that a priori estimates (see e.g. \cite{W18}) show that
\begin{equation}\label{esti-KS100}
\frac{d}{dt} \bke{\int_{\Omega} \rho \ln \rho+\frac{1}{2}\int_{\Omega} \frac{|\nabla c|^2}{c}}+\frac{4}{m^2}\int_{\Omega} \abs{\nabla \rho^{\frac{m}{2}}}^2+N\bke{\int_{\Omega}\frac{|\nabla c|^4}{c^3}+\int_{\Omega}\frac{|D^2 c|^2}{c} }+\frac{1}{2}\int_{\Omega} \frac{\rho|\nabla c|^2}{c}
\le 0.
\end{equation}
Indeed, this is due to following estimates
\begin{equation}\label{esti-KS110}
\frac{d}{dt} \int_{\Omega} \rho \ln \rho+\frac{4}{m^2}\int_{\Omega} \abs{\nabla \rho^{\frac{m}{2}}}^2=-\int_{\Omega} \rho\Delta c
\end{equation}
\begin{equation}\label{esti-KS120}
\frac{1}{2}\frac{d}{dt}\int_{\Omega} \frac{|\nabla c|^2}{c}+\int_{\Omega} c\abs{D^2 \ln c}^2 +\frac{1}{2}\int_{\Omega} \frac{\rho|\nabla c|^2}{c}
-\frac{1}{2}\int_{\partial\Omega} \frac{1}{c}\frac{|\nabla c|^2}{\partial \nu}=\int_{\Omega} \rho\Delta c
\end{equation}
Adding \eqref{esti-KS110} and \eqref{esti-KS120}, it follows that
\begin{equation}\label{esti-KS130}
\frac{d}{dt} \bke{\int_{\Omega} \rho \ln \rho+\frac{1}{2}\int_{\Omega} \frac{|\nabla c|^2}{c}}+\frac{4}{m^2}\int_{\Omega} \abs{\nabla \rho^{\frac{m}{2}}}^2+\int_{\Omega} c\abs{D^2 \ln c}^2 +\frac{1}{2}\int_{\Omega} \frac{\rho|\nabla c|^2}{c}
-\frac{1}{2}\int_{\partial\Omega} \frac{1}{c}\frac{|\nabla c|^2}{\partial \nu}=0.
\end{equation}
The case of convex domain implies that the boundary term in \eqref{esti-KS130} becomes positive.
Recalling the inequality
\begin{equation}\label{esti-KS140}
\int_{\Omega}\frac{|\nabla c|^4}{c^3}+\int_{\Omega}\frac{|D^2 c|^2}{c}\le N\int_{\Omega} c\abs{D^2 \ln c}^2,
\end{equation}
we obtain \eqref{esti-KS100}.

This estimate implies that $\nabla c\in L^4 (\Omega_{T})$. With the aid of main results, we note that $L^q$-weak solution for the system \eqref{PME-KS-10} exists, provided that 
\[
\frac{d+2+q_{m,d}}{4} \le 1+q_{m,d}\quad \Longrightarrow \quad m \ge 1+\frac{q(d-2)}{3d}\quad \Longrightarrow \quad q\le \frac{3d(m-1)}{d-2}.
\]
If $\frac{2(2d-1)}{3d}\le m\le \frac{3d-2}{2d}$, it follows that 
if $1\le q\le \frac{3d(m-1)}{d-2}$, and thus, a $L^q$-weak solution exists. 
In case that $d=3$, if $15/13 \le m\le 7/6$, via the bootstrapping argument, bounded weak solutions exist, provided that initial data is sufficiently regular.
In general, since a $L^q$-weak solution exists for $q=\frac{3d(m-1)}{d-2}$, we note that 
$q_{m,d}=\frac{d-2}{3}$ and 
it follows from parabolic embedding that 
\begin{equation}\label{embdeeing-KS}
\rho \in L^{r_1, r_2}_{x,t}(\Omega_T),\qquad \frac{d}{r_1}+\frac{2+\frac{d-2}{3}}{r_2}=\frac{d-2}{3(m-1)}, \quad 
m-1+\frac{3d(m-1)}{d-2}\le  r_2\le \infty.
\end{equation}
In particular, choosing $(r_1, r_2)=(d, 2)$, we have 
\[
2+\frac{d-2}{6}=\frac{d-2}{3(m-1)}\quad\Longrightarrow \quad m=1+\frac{2(d-2)}{d+10}.
\]
On the other hand, since $m\le \frac{3d-2}{2d}$, only the case $d=3$ is available to satisfy
$1+\frac{2(d-2)}{d+10}\le \frac{3d-2}{2d}$. Making restriction as $d=3$, we consider the case
 $1+\frac{2(d-2)}{d+10}=\frac{15}{13} \le m\le  \frac{3d-2}{2d}=\frac{7}{6}$.
Due to the maximal regularity of the equation $c$ and spatial embedding, we obtain $\nabla c\in L^{q}_xL^2_t (\Omega_T)$ for any $q<\infty$. Applying main results to this case, we can obtain a $L^q$ weak solution and, in addition, by bootstrapping argument, it follows that $c\in (L^q_tW^{2, q}_x \cap W^{1,q}_t L^q_x)(\Omega_T)$, which in turns improves the regularity of $\rho$. Since  it is rather straightforward that $\rho$ becomes bounded, we skip its details.

In summary, we obtain the following:


\begin{theorem}\label{KS-thm}
Let $\Omega$ be a bounded convex domain in $\mathbb{R}^d$ for $d\geq 3$ with smooth boundary.
\begin{itemize}
\item[(i)]
Let $\frac{2(2d-1)}{3d}\le m\le \frac{3d-2}{2d}$.
Suppose that $\rho_0$ and $c_0$ are non-negative and for 
any $1\le q\le \frac{3d(m-1)}{d-2}$
\begin{equation}\label{esti-KS150}
\norm{\rho_0}_{L^{q}(\Omega)}+\norm{c_0}_{L^{\infty}(\Omega)\cap H^1(\Omega)}<\infty.
\end{equation}
Then, there exists a $L^q$-weak solution $(\rho. c)$ of
the system \eqref{PME-KS-10}-\eqref{PME-KS-bc}.

\item[(ii)]
In case that $d=3$, if $\frac{15}{13}\le m \le \frac{7}{6}$ and if $\rho_0$ and $c_0$ are non-negative and
\begin{equation}\label{esti-KS150}
\norm{\rho_0}_{L^{\infty}(\Omega)}+\norm{c_0}_{W^{2,\infty}(\Omega)}<\infty,
\end{equation}
then  there exists a bounded weak solution of
the system \eqref{PME-KS-10}-\eqref{PME-KS-bc}.
\end{itemize}

\end{theorem}

\begin{remark}
When $\frac{2(2d-1)}{3d}\le m\le \frac{3d-2}{2d}$, $d\ge 4$ and $\frac{9}{8}\le m\le \frac{15}{13}$, $d=3$,
it is not clear whether or not $L^q$-weak solution exists for $\frac{3d(m-1)}{d-2}<q<\infty$, and thus we leave it as an open question. 
\end{remark}


\section{Preliminaries}\label{S:Preliminaries}

In this section, we introduce preliminaries that are used throughout the paper.

For a function $f:\Omega \times [0,T] \to \mathbb{R}$, $\Omega \subset \mathbb{R}^d, d \geq 2$ and constants $q_1, q_2 > 1$, we define
\[
\|f\|_{L^{q_1, q_2}_{x, t}} : = \left(\int_{0}^{T}\left[\int_{\Omega} \abs{f(x,t)}^{q_1} \,dx\right]^{q_2 / q_1} \,dt\right)^{\frac{1}{q_2}}.
\]
For simplicity, let $\|f\|_{L^{q}_{x,t}} = \|f\|_{L^{q, q}_{x,t}}$ for some $q>1$. Also we denote for $\Omega \subseteq \bbr^d$ and $\Omega_T \subseteq Q_T$ that
\begin{equation*}
\begin{gathered}
\|f(\cdot,t)\|_{\calC^{\alpha}(\Omega)} := \sup_{x,y \in \Omega, \, x \neq y} \frac{|f(x, t) - f(y,t)|}{|x-y|^\alpha}, \\
\|f\|_{\calC^{\alpha}(\Omega_T)} := \sup_{(x, t), (y,s) \in \Omega_T, \, (x,t) \neq (y,s)} \frac{|f(x, t) - f(y,s)|}{|x-y|^\alpha + |t-s|^{\alpha/2}}.
\end{gathered}
\end{equation*}

\subsection{Technical lemmas}

\begin{lemma}\label{T:pSobolev} \cite[Propositions~I.3.1 \& I.3.2]{DB93}
  Let $v \in L^{\infty}(0, T; L^{q}(\Omega)) \cap L^{p}(0, T; W^{1,p}(\Omega))$ for some $1 \leq p < d$, and $0 < q < \frac{dp}{d-p}$. Then there exists a constant $c=c(d,p,q, \abs{\Omega})$ such that
  \[
  \iint_{\Omega_{T}} |v(x,t)|^{\frac{p(d+q)}{d}}\,dx\,dt \leq c \left(\sup_{0\leq t \leq T}\int_{\Omega} |v(x,t)|^{q}\,dx \right)^{\frac{p}{d}} \iint_{\Omega_{T}} |\nabla v(x,t)|^{p}\,dx\,dt +  \frac{1}{|\Omega|^{\frac{p(d+q)}{d}-1}}\int_{0}^{T} \|v(\cdot, t)\|_{L^{1}(\Omega)}^{\frac{p(d+q)}{p}}\,dt.
  \]
\end{lemma}

Now we derive the following lemma which is useful to obtain a priori estimates.
 \begin{lemma}\label{P:L_r1r2}\cite[Lemma~3.4]{HKK}
 Let $m >1$ and $q \geq 1$. Suppose that
 \begin{equation}\label{rho-space}
 \rho \in L^{\infty}(0, T; L^{q}(\Omega)) \quad \text{and} \quad
 \rho^{\frac{q+m-1}{2}} \in L^{2}(0, T; W^{1,2}(\Omega)).
 \end{equation}
  Then $\rho\in L^{r_1, r_2}_{x, t} (\Omega \times [0, T])$ such that
 \begin{equation}\label{q-r1r2}
 \begin{gathered}
\frac{d}{r_1} + \frac{2+q_{m,d}}{r_2} =
\frac{d}{q}, \quad \text{for} \quad q_{m,d} : = \frac{d(m-1)}{q}, \\
 \text{for} \ \
 \begin{cases}
  q \leq r_1 \leq \frac{d(q+m-1)}{d-2},  \quad q+m-1 \leq r_2 \leq \infty, & \text{ if } \ d > 2, \\
  q \leq r_1 < \infty, \quad q+m-1 < r_2 \leq \infty, &\text{ if }\ d = 2, \\
   q \leq r_1 \leq \infty, \quad q+m-1 \leq r_2 \leq \infty, &\text{ if } \ d= 1.
  \end{cases}
\end{gathered}
\end{equation}
Moreover, there exists a constant $c=c(d)$ such that, for $(d-2)_{+} = \max \{0, d-2\}$,
 \begin{equation}\label{norm-q-r1r2}
\|\rho\|_{L^{r_1, r_2}_{x, t}} \leq c\left(\sup_{0 \leq t \leq T} \int_{\mathbb{R}^d} \rho^{q}(\cdot, t) \,dx\right)^{\frac{1}{r_1} \left[1-\frac{(r_1 - q) (d-2)_{+}}{d(m-1)+2q}\right]} \,
\left\|\nabla \rho^{\frac{q+m-1}{2}} \right\|_{L^{2}_{x,t}}^{\frac{2}{r_2}}
+ \frac{1}{|\Omega|^{1-\frac{1}{r_1}}} \left(\int_{0}^{T} \|\rho(\cdot, t)\|_{L^{1}(\Omega)}^{r_2}\,dt\right)^{\frac{1}{r_2}}.
\end{equation}
 \end{lemma}

Now we introduce the Aubin-Lions lemma.
\begin{lemma}\label{AL}\cite[Proposition~III.1.3]{Show97}, \cite{Sim87}
Let $X_0$, $X$ and $X_1$ be Banach spaces with $X_0 \subset X \subset X_1$. Suppose that $X_0$ and $X_1$ are reflexive, and $X_0 \hookrightarrow X $ is compact, and $X \hookrightarrow X_1 $ is continuous. For $1 \leq p, q \leq \infty$, let us define
$W = \left\{ u \in L^{p}\left(0, T; X_0\right),\ u_{t} \in L^{q}\left(0, T; X_1\right) \right\}$.
If $p < \infty$, then the inclusion $W \hookrightarrow L^{p}(0,T; X)$ is compact. If $p=\infty$ and $q>1$, then the embedding of $W$ into $\calC(0,T;X)$ is compact.
\end{lemma}

\subsection{Wasserstein space}\label{SS:Wasserstein}
In this subsection, we introduce the Wasserstein space and its properties. For more detail, we refer \cite{ags:book, Santambrosio15, V}.
 Let us denote by $\mathcal{P}_p (\Omega)$ the set of all Borel probability measures on $\Omega$ with a finite $p$-th moment. That is,
 $\mathcal{P}_p (\Omega):= \{\mu \in \mathcal{P} (\Omega) : \int_\Omega |x|^p \,d \mu(x) < \infty \}$.
 We note that $\mathcal{P}_p (\Omega)=\mathcal{P} (\Omega)$ if $\Omega$ is a bounded set.
For $\mu,\nu\in\mathcal{P}_p (\Omega)$, we consider
\begin{equation}\label{Wasserstein dist}
W_p(\mu,\nu):=\left(\inf_{\gamma\in\Gamma(\mu,\nu)}\int_{\Omega\times \Omega}|x-y|^p\, d\gamma(x,y)\right)^{\frac{1}{p}},
\end{equation}
where $\Gamma(\mu,\nu)$ denotes the set of all Borel probability measures on $\Omega\times \Omega$ which has $\mu$ and
$\nu$ as marginals;
$$\gamma(A\times \Omega)= \mu(A) \quad \text{and} \quad \gamma(\Omega\times A)= \nu(A) $$
for every Borel set $A\subset \Omega.$
Equation (\ref{Wasserstein dist}) defines a distance on $\mathcal{P}_p (\Omega)$ which is called the {\it Wasserstein distance} and denoted by $W_p$.
 Equipped with the Wasserstein distance,  $\mathcal{P}_p (\Omega)$ is called the {\it Wasserstein space}.
We denote by $\Gamma_o(\mu,\nu)$ the set of all $\gamma$ which minimize the expression.

We say that a sequence of Borel probability measures $\{\mu_n\} \subset \mathcal{P}(\Omega)$ is narrowly convergent to $ \mu \in \mathcal{P}(\Omega)$ as
$n \rightarrow \infty$ if
\begin{equation}\label{D:narrowly convergent}
\lim_{n\rightarrow \infty} \int_{\Omega} \varphi(x) \,d\mu_n(x) =\int_{\Omega} \varphi(x) \,d\mu(x)
\end{equation}
for every function $ \varphi \in \calC_b (\Omega)$, the space of continuous and bounded real functions defined on $\Omega$. Then we recall
 \begin{equation*}
\lim_{n \rightarrow \infty} W_p(\mu_n, \mu)=0 \quad \Longleftrightarrow \quad
\begin{cases}
\mu_n ~\mbox{narrowly converges to } \mu ~\mbox{ in }  \mathcal{P}(\Omega)\\
\lim_{n\rightarrow \infty} \int_{\Omega} |x|^p\, d\mu_n(x) = \int_{\Omega} |x|^p\,d \mu(x).
\end{cases}
\end{equation*}
Hence, if $\Omega$ is compact then $\mathcal{P}_p(\Omega)$ is also a compact metric space.%


Now, we introduce the notion of absolutely continuous curve and its relation with the continuity equation.
\begin{defn}
Let $\sigma:[0,T]\mapsto \mathcal{P}_p (\Omega)$ be a curve.
We say that $\sigma$ is absolutely continuous and denote it by $\sigma \in AC(0, T;\mathcal{P}_p (\Omega))$, if there exists $l\in
L^1([0,T])$ such that
\begin{equation}\label{AC-curve}
W_p(\sigma(s),\sigma(t))\leq \int_s^t l(r)dr,\qquad \forall ~ 0\leq s\leq t\leq T.
\end{equation}
If $\sigma \in AC(0,T;\mathcal{P} (\Omega))$, then the limit
$$|\sigma'|(t):=\lim_{s\rightarrow t}\frac{W_p(\sigma(s),\sigma(t))}{|s-t|} ,$$
exists for $L^1$-a.e $t\in[0,T]$. Moreover, the function $|\sigma'|$ belongs to $L^1(0,T)$ and satisfies
\begin{equation*}
|\sigma'|(t)\leq l(t) \qquad \mbox{for} ~L^1\mbox{-a.e.}~t\in [0,T],
\end{equation*}
for any $l$ satisfying \eqref{AC-curve}.
We call $|\sigma'|$ by the metric derivative of $\sigma$.
\end{defn}

\begin{lemma}\label{representation of AC curves} \cite[Theorem 5.14]{Santambrosio15}
If a narrowly continuous curve $\sigma : [0,T] \mapsto \mathcal{P}_p (\Omega)$ satisfies the continuity equation
$$\partial_t\sigma +\nabla\cdot(V\sigma)=0,  $$
for some Borel vector field $V$ with $\|V(t)\|_{L^p(\sigma(t))}\in L^1(0,T)$, then $\sigma: [0,T]\mapsto \mathcal{P}_p (\Omega)$
is absolutely continuous and
$|\sigma'|(t)\leq \|V(t)\|_{L^p(\sigma(t))}$ for $L^1$-a.e $t\in [0,T]$.
\end{lemma}

\begin{lemma}\label{Lemma : Arzela-Ascoli} \cite[Proposition 3.3.1]{ags:book}
Let $K  \subset \mathcal{P}_p(\Omega)$ be a sequentially compact set w.r.t the narrow topology.
Let $\sigma_n : [0, T] \rightarrow \mathcal{P} (\Omega)$ be curves such that
\begin{equation*}\label{equi-continuity}
\begin{aligned}
\sigma_n(t) \in K, \quad \forall ~ n &\in \mathbb{N}, ~ t \in [0,T],\\
W_p(\sigma_n(s),\sigma_n(t))&\leq \omega(s,t), \qquad \forall ~
0\leq s\leq t\leq T, ~ n \in \mathbb{N},
\end{aligned}
\end{equation*}
for a continuous function $\omega : [0, T] \times [0, T] \rightarrow [0, \infty)$ such that
$$\omega(t,t)=0, ~~ \forall ~ t \in [0,T].$$
Then there exists a subsequnece $\sigma_{n_k}$ and a limit curve $\sigma : [0,T] \rightarrow {P}_p(\Omega)$ such that
\begin{equation*}\label{eq1 : Lemma : Arzela-Ascoli}
\sigma_{n_k}(t) \mbox{ narrowly converges to} ~~\sigma(t), \qquad
\text{for all} ~~t\in[0, T].
\end{equation*}
\end{lemma}

 \subsection{Flows on $\mathcal{P}_p(\Omega)$ generated by vector fields}
For a given $T>0$, let $V\in L^1(0,T; W^{1,\infty}(\Omega;\mathbb{R}^d))$ be such that $V\cdot \mathbf{n}=0$ on $\partial \Omega$ where $\mathbf{n}$ is the outward unit normal to the
boundary of $\Omega$. For any $s, ~t\in [0,T]$, let
$\psi:[0,T]\times [0,T]\times\Omega\mapsto \Omega$ be the flow map
of the vector field $V$.
 More precisely, $\psi$ solves the following ODE
 \begin{equation}\label{ODE}
\begin{cases}
\frac{d}{dt}\psi(t;s,x)= V(\psi(t;s,x),t),  &\text{for } s, t\in[0,T] \vspace{1 mm}\\
\psi(s;s,x)=x,  &\text{for }  x\in\Omega.
\end{cases}
\end{equation}
Using the flow map $\psi$, we define a flow $\Psi: [0,T]\times [0,T]\times \mathcal{P}_p(\Omega)\mapsto \mathcal{P}_p(\Omega)$ through the push forward operation as follows
\begin{equation}\label{Flow on Wasserstein}
\Psi(t;s,\mu):={\psi}(t;s,\cdot)_\# \mu, \qquad \forall ~ \mu \in \mathcal{P}_p(\Omega).
\end{equation}

In this subsection, we remind two basic results on the flow map $\psi$.
\begin{lemma}\label{Lemma : estimation 1: ODE} \cite[Lemma~3.1]{KK-SIMA}
Let $s\in [0,T]$ and $\psi$ be defined as in \eqref{ODE}. Then, for any $t\in [s,T]$, the map $\psi(t;s,\cdot):\Omega\mapsto \Omega$ satisfies
\begin{equation*}\label{estimation 1: ODE}
e^{-\int_s^t \mbox{Lip}(V(\tau)) \,d\tau}|x-y|\leq |\psi(t;s,x)-\psi(t;s,y)|\leq e^{\int_s^t \mbox{Lip}(V(\tau)) \,d\tau}|x-y| ,
\qquad \forall ~ x,y \in\Omega,
\end{equation*}
where
\[\mbox{Lip}(V(\tau)) := \sup_{x,y \in \Omega , \, x \neq y} \frac{|V(x, \tau) - V(y,\tau)|}{|x-y|}.\]
\end{lemma}

\begin{lemma}\label{Lemma : Lipschitz of Jacobian}
Let $s\in [0,T]$ and $\psi$ be defined as in \eqref{ODE}. For any $t\in [s,T]$, let $J_{s,t}$ be the Jacobian corresponding to the map
 $\psi(t;s,\cdot):\Omega\mapsto \Omega$ . That is,
\begin{equation}\label{Jacobian}
\int_{\Omega} \zeta(y) dy:= \int_{\Omega} \zeta(\psi(t;s,x))
J_{s,t}(x)\,dx,\qquad \forall ~\zeta\in C(\Omega) .
\end{equation}
Then, the Jacobian  $ J_{s,t}$ is given as
\begin{equation}\label{Jacobian - formular}
J_{s,t}(x)=e^{\int_s^t \nabla \cdot V (\psi(\tau;s,x),\tau) \,d\tau}, \qquad \forall ~ x\in \Omega.
\end{equation}
\end{lemma}

\begin{proof}
Let $x\in \Omega$ be given. Then, we have
\begin{equation*}\label{eq6 : Lipschitz of Jacobian}
\begin{aligned}
\frac{d}{dt} J_{s,t}(x)\big |_{t=\tau}& =\nabla \cdot V(\psi(\tau;s,x), \tau)J_{s,\tau}(x).
\end{aligned}
\end{equation*}
Since $J_{s,s}\equiv 1$, we have
\begin{equation}\label{eq7 : Lipschitz of Jacobian}
\log J_{s,t}(x) = \int_s^t \nabla \cdot V( \psi(\tau;s,x),\tau)
\,d\tau.
\end{equation}
This completes the proof.
\end{proof}

\begin{remark}
We note that $\psi(t;s,\cdot) \circ \psi(s;t,\cdot) = Id $, that is,
$$ \psi(t;s, \psi(s; t, x)) = x,  \quad \forall ~ x\in \Omega. $$
Exploiting this, we have
\begin{equation*}
J_{s,t}(\psi(s;t,x )) = \frac{1}{J_{t,s}(x)}.
\end{equation*}
\end{remark}

 \begin{lemma}\label{Corollary-4:Lipschitz}\cite[Lemma~3.3]{KK-SIMA}
Let $s\in [0,T]$ and $\Psi$ be  in \eqref{Flow on Wasserstein}. For any $t\in [s,T]$ and $\mu,\nu\in \mathcal{P}_p (\Omega)$,
we have
\begin{equation}\label{eq1:Corollary-4:Lipschitz}
e^{-\int_s^t \mbox{Lip}(V(\tau)) \,d\tau}W_p(\mu,\nu)\leq W_p(\Psi(t;s,\mu),\Psi(t;s,\nu))\leq e^{\int_s^t \mbox{Lip}(V(\tau)) \,d\tau}W_p(\mu,\nu).
\end{equation}
\end{lemma}

\begin{lemma}\label{Lemma : density relation on the flow}\cite[Lemma~3.4]{KK-SIMA}
Let $\psi$ and $\Psi$ be defined as in \eqref{ODE} and \eqref{Flow on Wasserstein}, respectively. If $\mu \in \mathcal{P}_p^{ac}(\Omega)$ then
$\Psi(t;s,\mu)\in \mathcal{P}_p^{ac}(\Omega)$. Here, $ \mathcal{P}_p^{ac}(\Omega)$ is the set of all probability measures in $\mathcal{P}_p(\Omega)$ which are absolutely
continuous with respect to the {\it Lebesgue measure} in $\Omega$. Furthermore, suppose $\mu=\varrho \,dx$ and $\Psi(t;s,\mu)=\rho \,dx$ then
\begin{equation}\label{Density relation}
\rho(\psi(t;s,x))J_{s,t}(x)=\varrho(x), \qquad a.e \quad x\in \Omega,
\end{equation}
where $J_{s,t}$ is the Jacobian of the map $\psi(t;s,\cdot)$ as in \eqref{Jacobian}.
We also have
\begin{equation}\label{Entropy relation} \int_{\Omega} \rho \log
\rho \,dx = \int_{\Omega}\varrho \log \varrho \,dx -
\int_{\Omega}\varrho \log J_{s,t}  \,dx.
\end{equation}
Moreover, if $\varrho \in L^{q}(\Omega)$ for $q\in[1,\infty]$, then
${\rho}\in L^{q}(\Omega)$ and we have
\begin{equation}\label{L^p relation}
\|{\rho}\|_{L^{q}(\Omega)} \leq
\|\varrho\|_{L^{q}(\Omega)}e^{\frac{q-1}{q}\int_s^t\|\nabla\cdot
V\|_{L^\infty_x} \,d\tau},
\end{equation}
where $\frac{q-1}{q}=1$ if $q=\infty$.
\end{lemma}

\begin{lemma}\label{Lemma : Holder regularity on the flow}\cite[Lemma~3.18]{HKK}
Let $\psi$ and $\Psi$ be defined as in \eqref{ODE} and \eqref{Flow on
Wasserstein}, respectively. Suppose that $V \in  L^1(0,T;W^{1,\infty}(\Omega))$.
\begin{itemize}
\item[(i)] If $V \in L^1(0,T;  \calC^{1, \alpha}(\Omega)) $ and $\varrho\in
\mathcal{P}_p(\Omega) \cap \calC^\alpha (\Omega)$ for some $\alpha \in (0,1)$,
then $\rho:=\Psi(t;s,\varrho)$ is also H\"{o}lder continuous. More
precisely, we have $\| \rho\|_{\calC^\alpha (\Omega)} < C $ where $C=C(\| \varrho\|_{\calC^\alpha (\Omega)},\, \int_s^t \|\nabla V \|_{\calC^\alpha (\Omega)} d\tau )$.
\item[(ii)] If $V \in L^1(0,T;  W^{2, \infty}(\Omega)) $, then, for any  $a>0,~ q\geq 1$, we have
\begin{equation}\label{eq5 : Sobolev}
\begin{aligned}
 \|\nabla \rho^a\|_{L^q (\Omega)}&\leq e^{(a+2)\int_s^t \|\nabla V\|_{L^\infty_x}\, d\tau}\left \{ \|\nabla \varrho^a\|_{L^q (\Omega)} + \| \varrho^{a}\|_{L^p_x}
 \times \left ( a\int_s^t \| \nabla^2 V\|_{L^\infty_x} \,d\tau\right ) \right \}.
\end{aligned}
\end{equation}
\end{itemize}
\end{lemma}

\section{A priori estimates}\label{S:a priori}

Here we provide a priori estimates of a regular solution of \eqref{E:Main}-\eqref{E:Main-bc-ic} given as following definition.
\begin{defn}\label{D:regular-sol}
Let $q\in [1, \infty)$ and $V$ be a measurable vector field.
We say that a nonnegative measurable function $\rho$ is
a \textbf{regular solution} of \eqref{E:Main}-\eqref{E:Main-bc-ic} with $\rho_0 \in \calC^{\alpha}(\overline{\Omega})$ for some $\alpha \in (0,1)$ if the followings are satisfied:
\begin{itemize}
\item[(i)] It holds that
\[
\rho \in L^{\infty}(0,T;\mathcal{C}^{\alpha}(\overline{\Omega})), \quad \nabla \rho^m \in L^1(\Omega_T), \quad \nabla \rho^{\frac{m+q-1}{2}} \in L^{2} (\Omega_T), \quad  \text{and} \quad  V \in L^{1}_{\loc}(\Omega_{T}).
\]
\item[(ii)] For any $\varphi \in {\mathcal{C}}^\infty (\overline{\Omega} \times [0,T))$, it holds that
\[
\iint_{\Omega_{T}} \left\{  \rho \varphi_t - \nabla \rho^m \cdot \nabla \varphi + \rho V\cdot \nabla \varphi \right\} \,dx\,dt = \int_{\Omega} \rho_{0} \varphi(\cdot, 0) \,dx.
\]
\end{itemize}
\end{defn}


\subsection{Energy estimates}

In the following proposition, we deliver a priori estimates considering four cases separately.



\begin{proposition}\label{P:Lq-energy}
Let $m> 1$ and $q\geq 1$. Suppose that $\rho$ is a regular solution of \eqref{E:Main}-\eqref{E:Main-bc-ic} with $\rho_{0}\in  L^{q}(\Omega) \cap \calC^{\alpha}(\overline{\Omega})$.
\begin{itemize}

\item[(i)] Let $1 < m \leq 2$. Assume $\int_{\Omega} \rho_0 \log \rho_0 \,dx < \infty$ and
\begin{equation}\label{V-L1-energy}
V \in \mathcal{S}_{m,1}^{(q_1,q_2)}\cap \calC^{\infty} (\overline{\Omega}_{T}) \quad  \text{for} \quad
\begin{cases}
2 \leq q_2 \leq \frac{m}{m-1},  & \text{ if } d> 2, \vspace{1 mm}\\
2\leq q_2 < \frac{m}{m-1},  & \text{ if } d= 2.
\end{cases}
\end{equation}
Then the following estimate holds:
\begin{equation}\label{L1-energy}
 \sup_{0 \leq t \leq T}  \int_{\Omega \times \{t\}} \rho \abs{\log \rho} (\cdot, t)\,dx
 +  \frac{2}{m}\iint_{\Omega_T} \left|\nabla \rho^{\frac{m}{2}}\right|^2 \,dx\,dt \le C ,
\end{equation}
where  $C = C (\|V\|_{\mathcal{S}_{m,1}^{(q_1, q_2)}}, \, \int_{\Omega} \rho_0 \log \rho_0 \,dx)$.

\item[(ii)] Let $q>1$ and $q \geq m-1$. Assume that
\begin{equation}\label{V-Lq-energy}
V \in \mathcal{S}_{m,q}^{(q_1,q_2)}\cap \calC^{\infty} (\overline{\Omega}_{T}) \quad  \text{for} \quad
\begin{cases}
2 \leq q_2 \leq \frac{q+m-1}{m-1},  & \text{ if } d> 2, \vspace{1 mm}\\
2\leq q_2 < \frac{q+m-1}{m-1},  & \text{ if } d= 2.
\end{cases}
\end{equation}
Then the following estimate holds:
\begin{equation}\label{Lq-energy}
 \sup_{0 \leq t \leq T}  \int_{\Omega \times \{t\}} \rho^{q} (\cdot, t)\,dx
 + \frac{2 q m (q-1)}{(q+m-1)^2} \iint_{\Omega_T} \left|\nabla \rho^{\frac{q+m-1}{2}}\right|^2 \,dx\,dt \le C ,
\end{equation}
where  $C = C (\|V\|_{\mathcal{S}_{m,q}^{(q_1, q_2)}}, \, \|\rho_{0}\|_{L^{q} (\Omega)})$.

\item[(iii)]  Assume that
 \begin{equation}\label{V-tilde-Lq}
V\in \tilde{\mathcal{S}}_{m,q}^{(\tilde{q}_1, \tilde{q}_2)} \cap \calC^{\infty} (\overline{\Omega}_{T}) \quad  \text{for} \quad
\begin{cases}
1 \leq \tilde{q}_2 \leq \frac{q+m-1}{m-1},   & \text{ if } d> 2, \vspace{1 mm}\\
1\leq \tilde{q}_2 < \frac{q+m-1}{m-1},  & \text{ if } d= 2 .
\end{cases}
\end{equation}
$\bullet$ In case $q=1$, assume $\int_{\Omega} \rho_0 \log \rho_0 \,dx < \infty$. Then \eqref{L1-energy} holds with $C = C (\|V\|_{\tilde{\mathcal{S}}_{m,1}^{(\tilde{q}_1, \tilde{q}_2)}}, \, \int_{\Omega} \rho_0 \log \rho_0 \,dx )$.\\
$\bullet$ In case $q>1$, the estimate \eqref{Lq-energy} holds with
$C = C (\|V\|_{\tilde{\mathcal{S}}_{m,q}^{(\tilde{q}_1, \tilde{q}_2)}}, \, \|\rho_{0}\|_{L^{q} (\Omega)})$.

\item[(iv)]  Assume that $V\in\calC^{\infty} (\overline{\Omega}_{T})$ and $\nabla\cdot V \geq 0$.\\
$\bullet$ In case $q=1$, assume $\int_{\Omega} \rho_0 \log \rho_0 \,dx < \infty$. Then \eqref{L1-energy} holds with $C = C (\int_{\Omega} \rho_0 \log \rho_0 \,dx )$.\\
$\bullet$ In case $q>1$, the estimate \eqref{Lq-energy} holds with $C = C ( \|\rho_{0}\|_{L^{q} (\Omega)})$.
\end{itemize}

\end{proposition}

\begin{proof}
When $q=1$, we obtain the following estimate by testing $\log \rho$ to \eqref{E:Main}:
\begin{equation}\label{Elog-0}
\frac{d}{dt} \int_{\Omega} \rho \log \rho \,dx +  \frac{4}{m}\int_{\Omega} \abs{ \nabla \rho^{\frac{m}{2}}}^2\,dx =  \underbrace{\int_{\Omega} V \cdot \nabla \rho \,dx}_{ =: \mathcal{J}_1},
\end{equation}
because of the followings:
\[\begin{gathered}
\int_{\Omega} \rho_t \log \rho \,dx =  \frac{d}{dt} \int_{\Omega} (\rho \log\rho - \rho) \,dx = \frac{d}{dt} \int_{\Omega} \rho \log\rho \,dx, \\
\int_{\Omega} \nabla \cdot (\nabla \rho^m - V\rho) \log \rho \,dx 
=  - \int_{\Omega} \left(\nabla \rho^m \nabla \log \rho - V \rho \nabla \log \rho \right) \,dx,
\end{gathered}\]
because of the mass conservation property and the boundary condition \eqref{E:Main-bc-ic}.

Now let $q > 1$. By testing $\varphi = q \rho^{q-1}$ to \eqref{E:Main}, we have
\begin{equation}\label{apriori2-eqn}
  \frac{d}{dt} \int_{\Omega} \rho^{q}\,dx
+ \frac{4m q (q-1)}{(q+m-1)^2} \int_{\Omega} \abs{\nabla\rho^{\frac{q+m-1}{2}}}^2\,dx
 =  \underbrace{(q-1) \int_{\Omega} V \cdot \nabla \rho^{q}\,dx}_{=: \mathcal{J}_q},
\end{equation}
because of the followings: 
\[\begin{gathered}
\int_{\Omega} q \rho^{q-1} \rho_t\,dx =  \frac{d}{dt} \int_{\Omega} \rho^q \,dx, \\	
q\int_{\Omega} \nabla \cdot (\nabla \rho^m - V\rho) \rho^{q-1} \,dx
= - q\int_{\Omega} \left(\nabla \rho^m \nabla \rho^{q-1} - V \rho \nabla \rho^{q-1} \right) \,dx,
\end{gathered}\]
because of the boundary conditions \eqref{E:Main-bc-ic}.

\smallskip

(i) With the estimate \eqref{Elog-0}, we observe that, for  $1 - \frac m2 \geq 0$ (equivalent to $1 < m \leq 2$),
\begin{equation}\label{energy-V-0}
\mathcal{J}_{1} = \frac{2}{m} \int_{\Omega} V \rho^{1-\frac{m}{2}} \nabla \rho^{\frac m 2} \,dx \leq c  \|V\|_{L^{q_1}_{x}}
  \left\|\nabla \rho^{\frac{m}{2}}\right\|_{L^{2}_{x}}
 \|\rho\|_{L_{x}^{\frac{q_1 m }{q_1 -2}}}^{\frac{m}{2}},
\end{equation}
by taking H\"{o}lder inequality with $\frac{1}{q_1} + \frac{1}{2} + \frac{q_1 -2 }{2q_1} = 1$ for $q_1 > 2$. When $m=2$, we follow \eqref{energy-V2}. Moreover, we denote same settings in \eqref{energy00.4} and \eqref{energy00.5} for $q=1$. Then the same computation in \eqref{energy01} with $\varepsilon = \frac{2}{m}$ provides the following estimate:
\begin{equation}
\frac{d}{dt} \int_{\Omega} \rho \log \rho \,dx +  \frac{2}{m}\int_{\Omega} \abs{ \nabla \rho^{\frac{m}{2}}}^2\,dx
\leq c \|V\|_{L^{q_1}_{x}}^{q_2} \left(\|\rho\|_{L_{x}^{1}} + 1\right) + c
\end{equation}
for $q_2$ in \eqref{energy02} for $q=1$. The estimate \eqref{Lq-energy} is obtained by taking the integration in terms of $t\in [0, T]$ and applying that the negative part of $\int_{\Omega} \rho \log \rho \,dx < \infty$ because $\Omega$ is bounded and
$\|\rho(\cdot, t)\|_{L^{1}_{x}} = 1$ because of the mass conservation property.

(ii) In case $q>1$ and $q \geq m-1$, we observe that
\begin{equation}\label{energy-V-1}
\mathcal{J}_{q} = \frac{2(q-1)}{q+m-1}\int_{\Omega} V \rho^{\frac{q-m+1}{2}} \nabla \rho^{\frac{q+m-1}{2}}\,dx.
\end{equation}

When $q=m-1$ (in particular, $m=2$ if $q=1$), we have
\begin{equation}\label{energy-V2}
\mathcal{J}_q \leq c  \int_{\mathbb{R}^d} |V| \abs{\nabla \rho^{\frac{q+m-1}{2}}}\,dx
\leq c \|V\|_{L^{2}_{x}}
  \left\|\nabla \rho^{\frac{m+q-1}{2}}\right\|_{L^{2}_{x}}
  \leq \varepsilon  \left\|\nabla \rho^{\frac{m+q-1}{2}}\right\|_{L^{2}_{x}}^{2}
  + c(\varepsilon ) \|V\|_{L^{2}_{x}}^{2}.
\end{equation}
Now, let $q > m-1$ for all $q\geq 1$. Then by H\"{o}lder inequality with $\frac{1}{q_1} + \frac{1}{2} + \frac{q_1 -2 }{2q_1} = 1$ for $q_1 > 2$, we have
\begin{equation*}\label{energy00}
\mathcal{J}_q \leq c \int_{\Omega} |V| \rho^{\frac{q-m+1}{2}} \abs{\nabla \rho^{\frac{q+m-1}{2}}}\,dx
\leq  c  \|V\|_{L^{q_1}_{x}}
  \left\|\nabla \rho^{\frac{m+q-1}{2}}\right\|_{L^{2}_{x}}
 \|\rho\|_{L_{x}^{\frac{q_1 (q-m+1)}{q_1 -2}}}^{\frac{q-m+1}{2}}.
\end{equation*}
For simplicity, let
\begin{equation}\label{energy00.4}
    r_1 := \frac{q_1 (q-m+1)}{q_1 -2}, \quad \text{and} \quad \theta := \frac{q}{r_1} \left( 1 - \frac{(r_1 - q)(d-2)}{2q + d(m-1)}\right).
\end{equation}
Then by \eqref{q-r1r2} in Lemma~\ref{P:L_r1r2}, we see that
\begin{equation}\label{energy00.5}
\begin{cases}
     2<\frac{d(q+m-1)}{q+(d-1)(m-1)}\leq q_1 \leq \frac{2q}{m-1}, & \text{ if } d>2, \vspace{1 mm}\\
     2< q_1 \leq \frac{2q}{m-1}, & \text{ if } d = 2.
\end{cases}
\end{equation}
Then by modifying interpolation inequalities in Lemma~\ref{P:L_r1r2} for $\Omega \Subset \bbr^d$, we deduce
\begin{equation}\label{energy01}
\begin{aligned}
\mathcal{J}_q
&\leq c     \|V\|_{L^{q_1}_{x}}
  \left\|\nabla \rho^{\frac{m+q-1}{2}}\right\|_{L^{2}_{x}}
 \left(\left\|\nabla \rho^{\frac{m+q-1}{2}}\right\|_{L^{2}_{x}}^{\frac{(1-\theta)(q-m+1)}{q+m-1}}
 \|\rho\|_{L_{x}^{q}}^{\frac{\theta(q-m+1)}{2}} + \|\rho\|_{L_{x}^{1}}^{\frac{q+m-1}{2}}\right) \\
&\leq \varepsilon \left\|\nabla \rho^{\frac{m+q-1}{2}}\right\|_{L^{2}_{x}}^{2}
 +c(\varepsilon) \left( \|V\|_{L^{q_1}_{x}}^{q_2}\|\rho\|_{L_{x}^{q}}^{q_2 \frac{\theta(q-m+1)}{2}} + c\|V\|_{L^{q_1}_{x}}^{2}\right) \\
&\leq \varepsilon \left\|\nabla \rho^{\frac{m+q-1}{2}}\right\|_{L^{2}_{x}}^{2}
 +c(\varepsilon) \|V\|_{L^{q_1}_{x}}^{q_2} \left(\|\rho\|_{L_{x}^{q}}^{q_2 \frac{\theta(q-m+1)}{2}} + 1\right) + c(\varepsilon) \\
 &\leq \varepsilon \left\|\nabla \rho^{\frac{m+q-1}{2}}\right\|_{L^{2}_{x}}^{2}
 +c(\varepsilon) \|V\|_{L^{q_1}_{x}}^{q_2} \left(\|\rho\|_{L_{x}^{q}}^{q} + 1\right) + c(\varepsilon)
\end{aligned}
\end{equation}
by taking Young's inequality (note that $q_2 \geq 2$). Due to $V\in \mathcal{S}_{m,q}^{(q_1,q_2)}$, it gives
\begin{equation}\label{energy02}
 \frac{2(q+m-1)}{(q+m-1)-(1-\theta)(q-m+1)} = \frac{q_1 (2+q_{m,d})}{q_1 (1+q_{m,d}) - d} =  q_2.
\end{equation}
The conditions \eqref{energy00.5} corrresponds to \eqref{V-Lq-energy}.

For $q>1$, we combine \eqref{apriori2-eqn} and \eqref{energy01} with $\varepsilon = \frac{4mq(q-1)}{(q+m-1)^2}$ that yields
\begin{equation}\label{energy03}
    \frac{d}{dt} \int_{\Omega} \rho^{q}\,dx \leq c \|V\|_{L^{q_1}_{x}}^{q_2} (\|\rho\|_{L_{x}^{q}}^{q} + 1) + c \quad \Longrightarrow \quad
    \sup_{0\leq t \leq T} \int_{\Omega \times \{t\}} \rho^{q}\,dx
    \leq C \int_{\Omega} \rho_{0}^{q}\,dx + C,
\end{equation}
where $C= C(\|V\|_{\mathcal{S}_{m,q}^{(q_1, q_2)}})$ by Gr\"{o}nwall's inequality.
When $q=1$, it is obvious because of the mass conservation property.
Now we take the integration in terms of $t$ after combining estimates \eqref{apriori2-eqn} and \eqref{energy01}  which yields
\begin{equation}\label{energy05}
\begin{aligned}
\sup_{0\leq t\leq T} \int_{\Omega \times \{t\}} \rho^{q}\,dx
&+ \frac{2m q (q-1)}{(q+m-1)^2} \iint_{\Omega_T} \abs{\nabla\rho^{\frac{q+m-1}{2}}}^2\,dx\,dt \\
&\leq \int_{\Omega} \rho_{0}^{q}\,dx
+ c \|V\|_{\mathcal{S}_{m,q}^{(q_1, q_2)}}\sup_{0\leq t \leq T}\|\rho (\cdot, t)\|_{L_{x}^{q}(\Omega)}^{q_2 \frac{\theta(q-m+1)}{2}}.
\end{aligned}
\end{equation}
Therefore, we obtain \eqref{Lq-energy} by substituting \eqref{energy03} into \eqref{energy05}.
\smallskip

(iii) In this case, we take the integration by parts that gives, due to $V\cdot \textbf{n} = 0$ on $\partial \Omega$,
\begin{equation}\label{energy-V-2}
\mathcal{J}_{1} =  - \int_{\Omega} \rho \, \nabla \cdot V \,dx,
\quad \text{and} \quad
\mathcal{J}_{q} = -(q-1) \int_{\Omega} \rho^q \, \nabla \cdot V \,dx.
\end{equation}
By taking H\"{o}lder inequality with $\frac{1}{\tilde{q}_1} + \frac{\tilde{q}_1 -1}{\tilde{q}_1} = 1$ for $\tilde{q}_1 > 1$, we have
 \begin{equation*}
 \mathcal{J}_1 , \, \mathcal{J}_{q} \leq c \|\nabla  V\|_{L_{x}^{\tilde{q}_1}} \|\rho\|_{L^{\frac{q \tilde{q}_1}{\tilde{q}_1 -1}}_{x}}^{q}.
 \end{equation*}
 For simplicity, let
 \begin{equation*}
 \tilde{r}_1 := \frac{q\tilde{q}_1}{\tilde{q}_1 - 1} \quad \text{ and } \quad \theta := \frac{q}{\tilde{r}_1} \left( 1 - \frac{(\tilde{r}_1 - q)(d-2)}{2q+d(m-1)} \right).
 \end{equation*}
Then by \eqref{q-r1r2} in Lemma~\ref{P:L_r1r2} for $\tilde{r}_1$ instead of $r_1$, it gives
 \begin{equation}\label{energy06}
 \begin{cases}
  1<\frac{d+q_{m,d}}{2+q_{m,d}}\leq \tilde{q}_1 \leq \infty, & \text{ if } d>2, \vspace{1 mm}\\
  1< \tilde{q}_1 \leq \infty, & \text{ if } d= 2.
  \end{cases}
 \end{equation}

 Then by Lemma~\ref{P:L_r1r2}, we deduce
 \begin{equation}\label{energy07}
 \begin{aligned}
 \mathcal{J}_1 , \, \mathcal{J}_{q}
 &\leq c \|\nabla V\|_{L_{x}^{\tilde{q}_1}} \left(\|\rho\|_{L^{q}_{x}}^{q\theta}
 \left\|\nabla \rho^{\frac{q+m-1}{2}}\right\|_{L^{2}_{x}}^{\frac{2q(1-\theta)}{q+m-1}} + \|\rho\|^{q}_{L^1_x} \right)\\
 &\leq \varepsilon \left\|\nabla \rho^{\frac{q+m-1}{2}}\right\|_{L^{2}_{x}}^{2}
 + c(\varepsilon) \|\nabla V\|_{L_{x}^{\tilde{q}_1}}^{\tilde{q}_2} \left( \|\rho\|_{L^{q}_{x}}^{q\theta \tilde{q}_2} + 1 \right) \\
  &\leq \varepsilon \left\|\nabla \rho^{\frac{q+m-1}{2}}\right\|_{L^{2}_{x}}^{2}
 + c(\varepsilon) \|\nabla V\|_{L_{x}^{\tilde{q}_1}}^{\tilde{q}_2} \left( \|\rho\|_{L^{q}_{x}}^{q} + 1 \right),
 \end{aligned}
 \end{equation}
 by taking Young's inequality. Because of $V \in \tilde{\mathcal{S}}_{m,q}^{(\tilde{q}_1, \tilde{q}_2)}$, it gives
 \begin{equation*}
  \frac{q+m-1}{(q+m-1)-(1-\theta)q}  = \frac{\tilde{q}_1 (2+q_{m,d})}{\tilde{q}_1 (2+q_{m,d}) - d} =  \tilde{q}_2.
 \end{equation*}
 Also, we observe that \eqref{energy06} corresponds to \eqref{V-tilde-Lq}.

Then we combine \eqref{apriori2-eqn} and \eqref{energy07} with $\varepsilon = \frac{4mq(q-1)}{(q-m+1)^2}$ which yields
\begin{equation}\label{energy08.5}
    \sup_{0 \leq t \leq T} \int_{\Omega} \rho^q (\cdot, t)\,dx
    \leq c \|\nabla V\|_{L_{x}^{\tilde{q}_1}}^{\tilde{q}_2} \left( \|\rho\|_{L^{q}_{x}}^{q} + 1 \right)
    \quad \Longrightarrow \quad
    \sup_{0\leq t \leq T} \int_{\Omega \times \{t\}} \rho^{q}\,dx
    \leq C \int_{\Omega} \rho_{0}^{q}\,dx + C
\end{equation}
where $C=C(\|\nabla \cdot V\|_{\tilde{\mathcal{S}}_{m,q}^{(\tilde{q}_1, \tilde{q}_2)}})$ by applying Gr\"{o}nwall's inequality. Now, the estimate \eqref{Lq-energy} follows by applying \eqref{energy08.5} to the estimate obtained by taking integration in terms of $t$ to the combination of \eqref{apriori2-eqn} and \eqref{energy07}.

(iv) In \eqref{energy-V-2}, we observe that $\mathcal{J} \leq 0$ for all $q\geq 1$ when $\nabla \cdot V \geq 0$. Hence we obtain \eqref{Lq-energy} with $C$ depending only the domain and initial data.
\end{proof}

Under the same hypothesis in Proposition~\ref{P:Lq-energy}, we prove the following proposition that gives a priori estimate for constructing of $L^r$-weak solution for $1 \leq r \leq q$ when  $\rho_{0} \in \left(L^{1}\cap L^{q}\right)(\Omega)$ for $q>1$. 
\begin{proposition}\label{P:Lq-energy-r}
Let $m > 1$ and $q > 1$. Suppose that $\rho$ is a regular solution of \eqref{E:Main}-\eqref{E:Main-bc-ic} with $\rho_{0}\in L^{q}(\Omega) \cap \calC^{\alpha}(\Omega)$.
\begin{itemize}

\item[(i)] Let $q>1$ and $q \geq m-1$. Further, assume that $V$ satisfies \eqref{V-Lq-energy}. Then, for any $ \max\{1, m-1\} \leq r \leq q$, it holds
\begin{equation}\label{Lq-energy-r}
 \sup_{0 \leq t \leq T}  \int_{\Omega \times \{t\}} \rho^{r} (\cdot, t)\,dx
 + \iint_{\Omega_T} \left|\nabla \rho^{\frac{r+m-1}{2}}\right|^2 \,dx\,dt \le C ,
\end{equation}
where  $C = C (\|V\|_{\mathcal{S}_{m,q}^{(q_1, q_2)}}, \, \|\rho_{0}\|_{L^{q} (\Omega)})$.

\item[(ii)]  Let $q>1$ and $q \geq m-1$. If $V$ satisfies \eqref{V-tilde-Lq}, then the estimate \eqref{Lq-energy-r} holds with
$C = C (\|V\|_{\tilde{\mathcal{S}}_{m,q}^{(\tilde{q}_1, \tilde{q}_2)}},\,  \|\rho_{0}\|_{L^{q} (\Omega)})$ for any $1 \leq r \leq q$.

\item[(iii)] Let $q>1$ and $q \geq m-1$. If $V\in \calC^{\infty}(\overline{\Omega}_{T})$ and  $\nabla\cdot V \geq 0$, then the estimate \eqref{Lq-energy-r} holds with
$C = C (\|\rho_{0}\|_{L^{q} (\Omega)})$ for any $1\leq r\leq q$.
\end{itemize}
\end{proposition}

\subsection{Estimates of speed}
\label{SS:Speed}

Note that \eqref{E:Main} can be written as
\begin{equation}\label{E:Main:w}
\partial_t \rho + \nabla \cdot \left( w \rho\right) = 0, \quad \text{where} \quad w:= - \frac{\nabla \rho^m}{\rho} + V .
\end{equation}
Due to Lemma~\ref{representation of AC curves}, \eqref{E:Main:w} can be seen as a curve in $\calP_{p}(\bbr^d)$ whose speed at time $t$ is limited by $\|w(\cdot, t)\|_{L^p (\rho(t))}$. In the following lemma, we deliver the estimate of  $\|w\|_{L^p(0,T;L^p(\rho(t)))}$.

Recalling $\lambda_q$ in \eqref{lambda_q}, we observe that the constant $1 + \frac{d(q-1)+q}{d(m-1)+q}$ coincides with $2$ when $q=m$ but exceed $2$ for $q>m$. Because of difficulties of obtaining a speed estimate for $p>2$, we restrict $\lambda_q = 2$ for $q > m$. Therefore, when $\rho_0 \in L^{q}_{\Omega}$ for $q>m$, of course $\rho_0 \in L^{m}_{\Omega}$, we combine an energy estimate with the speed estimate at $q=m$.

The following lemma is an auxiliary estimate for the speed $w$ in \eqref{E:Main:w}.
\begin{lemma}\label{P:W-p}\cite[Lemma~4.2]{HKK}
Let $m>1$ and $1 \leq q \leq m$. Suppose that $\rho: \Omega_T \mapsto \bbr$ is a nonnegative measurable function satisfying
 \begin{equation*}\label{rho-space-speed}
 \rho \in L^{\infty}(0, T; L^{q}(\Omega)) \quad \text{and} \quad
 \rho^{\frac{q+m-1}{2}} \in L^{2}(0, T; W^{1,2}(\Omega)).
 \end{equation*}
 Furthermore, assume that
\begin{equation}\label{V:speed}
V\in \mathcal{S}_{m,q}^{(q_1, q_2)} \ \text{ for } \
\begin{cases}
 \lambda_q \leq q_2 \leq \frac{\lambda_q (q+m-1)}{q+m-2}, & \text{ if } \ d>2,  \vspace{1 mm}\\
 \lambda_q \leq q_2 < \frac{\lambda_q (q+m-1)}{q+m-2}, & \text{ if } \ d=2.
 \end{cases}
\end{equation}
 Then, for any $\varepsilon >0$, there exists constant $c=c(\varepsilon)$ such that
\begin{equation}\label{W-p}
\iint_{\Omega_T} \abs{\frac{\nabla \rho^m}{\rho}}^{\lambda_q} \rho \,dx\,dt
\leq \varepsilon \iint_{\Omega_T} \abs{ \nabla \rho^{\frac{q+m-1}{2}}}^2 \,dx\,dt
  +cT\left( \sup_{0\leq t \leq T}\int_{\Omega} \rho^q \,dx \right)^{\frac{2\theta}{d(1-\theta)}},
\end{equation}
\begin{equation}\label{V-p}
\iint_{\Omega_T} |V|^{\lambda_q}\rho \,dx\,dt
\leq \varepsilon \iint_{\Omega_T} \abs{ \nabla \rho^{\frac{q+m-1}{2}}}^2 \,dx\,dt
 + c \|\rho\|_{L^{q,\infty}_{x,t}}^{\beta\frac{ q q_2 (q_1 - \lambda_q)}{q_1 \lambda_q}} \|V\|_{\mathcal{S}_{m,q}^{(q_1,q_2)}}^{q_2},
\end{equation}
where $\theta=\frac{d \lambda_q (m-q)}{(2-\lambda_q)[(d+2)q+(m-2)d]}$ and $\beta = 1 - \frac{(d-2)(q_1(1-q)+q \lambda_q )}{(q_1 - \lambda_q)(d(m-1)+2q)}$.
\end{lemma}

Here we obtain the following estimates by combining energy and auxiliary speed estimates.

\begin{proposition}\label{P:Energy-speed}
Let $m> 1$ and $q\geq 1$.  Suppose that $\rho$ is a weak solution of \eqref{E:Main}-\eqref{E:Main-bc-ic} with $\rho_{0}\in  L^{q}(\Omega) \cap \calC^{\alpha}(\overline{\Omega})$.

\begin{itemize}
\item[(i)] Let $1< m \leq 2$. Assume further that $\int_{\Omega} \rho_0 \log \rho_0 \,dx < \infty$ and
\begin{equation}\label{V-L1-energy-speed}
V \in \mathcal{S}_{m,q}^{(q_1,q_2)}\cap \calC^\infty(\overline{\Omega}_T)\quad \text{for}\quad
    \begin{cases}
        q_1 \leq \frac{2m}{m-1}, \  2 \leq q_2 \leq  \frac{m}{m-1},  & \text{if } d > 2, \vspace{1 mm}\\
        q_1 \leq \frac{2m}{m-1}, \ 2 \leq q_2 <  \frac{m}{m-1},  & \text{if } d = 2.
    \end{cases}
\end{equation}
Then the following estimate holds
\begin{equation}\label{L1-energy-speed}
 \sup_{0 \leq t \leq T} \int_{\Omega\times \{t\}} \rho \abs{\log \rho } (\cdot, t)  \,dx
 + \iint_{\Omega_T} \left(\abs{\frac{\nabla \rho^m}{\rho}}^{p_1}+|V|^{p_1}\right ) \rho  \,dx\,dt \le C,
\end{equation}
where $C=C (\abs{\Omega}, \, \|V\|_{\mathcal{S}_{m,1}^{(q_1, q_2)}}, \, \int_{\Omega} \rho_0 \log \rho_0 \,dx)$.

\item[(ii)] Let $q>1$ and $q\geq m-1$.  Assume that
\begin{equation}\label{V-Lq-energy-speed}
V \in \mathcal{S}_{m,q}^{(q_1,q_2)}\cap \calC^\infty(\overline{\Omega}_T)\quad \text{for}\quad
    \begin{cases}
        q_1 \leq \frac{2m}{m-1}, \  2 \leq q_2 \leq  \frac{q+m-1}{m-1} \,\delta_{\{1 < q \leq m\}} + \frac{2m-1}{m-1}\,\delta_{ \{q > m \} },  & \text{if } d > 2, \vspace{1 mm}\\
        q_1 \leq \frac{2m}{m-1}, \ 2 \leq q_2 <  \frac{q+m-1}{m-1}\,\delta_{\{1 < q \leq m\}} + \frac{2m-1}{m-1} \,\delta_{ \{q > m \} },  & \text{if } d = 2.
    \end{cases}
\end{equation}
Then the following estimate holds
\begin{equation}\label{Lq-energy-speed}
 \sup_{0 \leq t \leq T} \int_{\Omega\times \{t\}} \rho^{q} (\cdot , t) \,dx
 + \iint_{\Omega_T} \left(\abs{\frac{\nabla \rho^m}{\rho}}^{\lambda_q}+|V|^{\lambda_q}\right ) \rho  \,dx\,dt \le C,
\end{equation}
where $C=C (\abs{\Omega}, \, \|V\|_{\mathcal{S}_{m,q}^{(q_1, q_2)}}, \, \|\rho_0\|_{L^q(\Omega)} )$.

\item[(iii)] Assume that 
 \begin{equation}\label{V-tilde-Lq-energy-speed}
V\in \tilde{\mathcal{S}}_{m,q}^{(\tilde{q}_1, \tilde{q}_2)}\cap \calC^\infty(\overline{\Omega}_T) \quad \text{for}\quad
\begin{cases}
\frac{2+q_{m,d}}{1+q_{m,d}} < \tilde{q}_2 \leq \frac{q+m-1}{m-1} \,\delta_{\{1 \leq q \leq m\}} + \frac{2m-1}{m-1}\,\delta_{ \{q >  m \} },  & \text{ if } d> 2, \vspace{1 mm}\\
\frac{2+q_{m,2}}{1+q_{m,2}} < \tilde{q}_2 < \frac{q+m-1}{m-1} \,\delta_{\{1 \leq q \leq m\}} + \frac{2m-1}{m-1}\,\delta_{ \{q > m \} },  & \text{ if } d= 2.
\end{cases}
\end{equation}
$\bullet$ In case $q=1$, assume $\int_{\Omega} \rho_0 \log \rho_0 \,dx < \infty$. Then, \eqref{L1-energy-speed} holds with $ C =C  (\abs{\Omega}, \, \|V\|_{\tilde{\mathcal{S}}_{m,1}^{(\tilde{q}_1, \tilde{q}_2)}}, \, \int_{\Omega} \rho_0 \log \rho_0 \,dx ).$\\
$\bullet$ In case $q>1$, the estimate \eqref{Lq-energy-speed} holds with $ C =C  (\abs{\Omega}, \, \|V\|_{\tilde{\mathcal{S}}_{m,q}^{(\tilde{q}_1, \tilde{q}_2)}}, \, \|\rho_0\|_{L^q(\Omega)} ).$

\item[(iv)]  Let $\nabla\cdot V \geq 0$. Assume that
\begin{equation}\label{V-Lq-divfree}
V\in \mathcal{S}_{m,q}^{(q_1, q_2)}\cap \calC^\infty(\overline{\Omega}_T) \quad  \text{for} \quad
\begin{cases}
\lambda_q \leq q_2 \leq \frac{\lambda_q (q+m-1)}{q+m-2}, & \text{ if }  d>2, \ 1 \leq q\leq m, \vspace{1 mm}\\
q_1 \leq \frac{2m}{m-1},\  2 \leq q_2 \leq \frac{2m-1}{m-1}, & \text{ if }  d>2, \ q> m, \vspace{1 mm}\\
\lambda_q \leq q_2 < \frac{\lambda_q (q+m-1)}{q+m-2}, & \text{ if }  d=2, \ 1 \leq q\leq m, \vspace{1 mm}\\
q_1 \leq \frac{2m}{m-1},\  2 \leq q_2 < \frac{2m-1}{m-1}, & \text{ if }  d=2, \ q> m.
\end{cases}
\end{equation}
$\bullet$ In case $q=1$, assume $\int_{\Omega} \rho_0 \log \rho_0 \,dx < \infty$. Then, \eqref{L1-energy-speed} holds with $ C =C  (\abs{\Omega}, \, \|V\|_{\mathcal{S}_{m,1}^{(q_1, q_2)}}, \, \int_{\Omega} \rho_0 \log \rho_0 \,dx ).$\\
$\bullet$ In case $q>1$, the estimate \eqref{Lq-energy-speed} holds with $ C =C  (\abs{\Omega}, \, \|V\|_{\mathcal{S}_{m,q}^{(q_1, q_2)}}, \, \|\rho_0\|_{L^q(\Omega)} ).$

\end{itemize}
\end{proposition}


The following estimates is based on embedding arguments because both the spatial and temporal domains are finite.

\begin{proposition}\label{P:Energy-embedding} (Embedding)
Let $m> 1$ and $q > 1$.  Suppose that $\rho$ is a weak solution of \eqref{E:Main}-\eqref{E:Main-bc-ic} with $\rho_{0}\in  L^{q}(\Omega) \cap \calC^{\alpha}(\overline{\Omega})$.

\begin{itemize}
\item[(i)]
For a given pair $(q_1, q_2)$ satisfying \eqref{V-Lq-energy}, there exists a pair of constants $(q_1^\ast, q_2^\ast)$ such that $q_1^\ast \in [\frac{d(q+m-1)}{q+(d-1)(m-1)},q_1] $ and $q_{2}^{\ast} \in [2, q_2]$ such that $V \in \mathcal{S}_{m, q^\ast}^{(q_1^\ast, q_2^{\ast})}$ satisfies \eqref{V-Lq-energy-speed} and the following estimate
\begin{equation}\label{Lq-energy-speed-embedding}
 \sup_{0 \leq t \leq T} \int_{\Omega\times \{t\}} \rho^{q^\ast}  \,dx
 + \iint_{\Omega_T} \left(\abs{\frac{\nabla \rho^m}{\rho}}^{\lambda_{q^\ast}}+|V|^{\lambda_{q^\ast}}\right ) \rho  \,dx\,dt \le C,
\end{equation}
holds where  $C = C (\abs{\Omega}, \, \|V\|_{\mathcal{S}_{m,q^{\ast}}^{(q_1^\ast, q^{\ast}_2)}}, \|\rho_0\|_{L^{q^\ast}(\Omega)})$.

\item[(ii)]
For a given pair $(\tilde{q}_1, \tilde{q}_2)$ satisfying \eqref{V-tilde-Lq}, there exists a pair $(\tilde{q}_1^\ast, \tilde{q}_2^\ast)$ such that $\tilde{q}_1^\ast \in [\frac{d+q_{m,d}}{2+q_{m,d}},\tilde{q}_1]$ and $\tilde{q}_2^\ast \in [\frac{2+q_{m,d}}{1+q_{m,d}},\tilde{q}_2]$ such that $ V \in \tilde{\mathcal{S}}_{m, q^\ast}^{(\tilde{q}_1^\ast, \tilde{q}_2^{\ast})}$ satisfies \eqref{V-tilde-Lq-energy-speed} and
the estimate \eqref{Lq-energy-speed-embedding} holds with $C = C (\|V\|_{\tilde{\mathcal{S}}_{m,\tilde{q}^{\ast}}^{(\tilde{q}_1^\ast, \tilde{q}^{\ast}_2)}}, \|\rho_0\|_{L^{q^\ast}(\Omega)})$.

\item[(iii)] Let $\nabla\cdot V \geq 0$. Assume that
\begin{equation}\label{V-Lq-divfree-embedding}
V\in \mathcal{S}_{m,q}^{(q_1, q_2)}\cap \calC^\infty(\overline{\Omega}_T) \quad \text{for}\quad
\begin{cases}
\frac{2+q_{m,d}}{1+q_{m,d}} < q_2 \leq \frac{p_1 m}{m-1}, & \text{ if }  d>2, \vspace{1 mm}\\
\frac{2+q_{m,2}}{1+q_{m,2}} < q_2 < \frac{p_1 m}{m-1}, & \text{ if }   d=2.
\end{cases}
\end{equation}
For a given pair $(q_1, q_2)$ satisfying \eqref{V-Lq-divfree-embedding}, there exists a pair of constants $(q_1^\ast, q_2^\ast)$ such that such that $q_1^\ast \in [\frac{2+q_{m,d}}{1+q_{m,d}},q_1] $ and $q_{2}^{\ast} \in [2, q_2]$ such that $V \in \mathcal{S}_{m, q^\ast}^{(q_1^\ast, q_2^{\ast})}$ satisfying \eqref{V-Lq-divfree} and the estimate \eqref{Lq-energy-speed-embedding} holds with
$C = C (\|V\|_{\mathcal{S}_{m,q^{\ast}}^{(q_1^\ast, q^{\ast}_2)}}, \, \|\rho_0\|_{L^{q^\ast}(\Omega)})$.
\end{itemize}
\end{proposition}

\subsection{Estimates of temporal and spatial derivatives.}\label{SS:Temporal}

Now with the initial data in $L^{q}_{x}$ for $q \geq 1$, we search further properties of the temporal and spatial derivatives of $\rho^q$ and $\rho^q$ itself. First, we recall a priori estimates in Proposition~\ref{P:Lq-energy} and remark the integrability of $\rho$ and $V\rho^q$.

Based on a priori estimates, Proposition~\ref{P:Lq-energy}, we remark the following first.
\begin{remark}\label{Remark : AL-2}
From the estimates in Proposition~\ref{P:Lq-energy}, it holds that $\rho \in L^{(q+m-1) + \frac{2q}{d}}_{x,t}$ by applying parabolic Sobolev embedding, that is,
 \begin{equation}\label{E:rho:q-1}
      \iint_{\Omega_T} \rho^{(q+m-1)+\frac{2q}{d}} \,dxdt
        \leq  \sup_{0 \leq t \leq T}  \int_{\Omega} \rho^{q} (\cdot, t)\,dx
 + \iint_{\Omega_T} \left|\nabla \rho^{\frac{q+m-1}{2}}\right|^2 \,dx\,dt \leq C ,
    \end{equation}
where the constant $C$ depends on $\Omega$, $V$, and $\rho_0$.
\end{remark}

We introduce two compactness arguments Proposition~\ref{P:AL1} for $1\leq q \leq m+1$ and Proposition~\ref{P:AL2} for $q \geq \max\{1, m-1\}$ that are used to show convergence of a sequence of regular solutions to a weak solution.

Before Proposition~\ref{P:AL1}, the following lemma informs further restrictions on $V$.
\begin{lemma}\label{L:V_compact_m+1}
 Let $m>1$ and $1\leq q \leq m+1$. Suppose that $ V\in \mathcal{S}_{m,q}^{(q_1, q_2)}$ satisfies
 \begin{equation}\label{V_compact_m+1}
 \begin{cases}
 \frac{1}{\gamma_1} - \frac{1}{q} \leq \frac{1}{q_1} \leq \frac{1}{\gamma_1} - \frac{d-2}{d(q+m-1)}, \quad
\frac{1}{\gamma_1} - \frac{1}{q+m-1} \leq \frac{1}{q_2} \leq \frac{1}{\gamma_1}, & \text{ if } d> 2, \vspace{1 mm} \\
\frac{1}{\gamma_1} - \frac{1}{q} \leq \frac{1}{q_1} < \frac{1}{\gamma_1}, \quad
\frac{1}{\gamma_1} - \frac{1}{q+m-1} < \frac{1}{q_2} \leq \frac{1}{\gamma_1}, & \text{ if } d = 2.
 \end{cases}
 \end{equation}
where
\begin{equation}\label{gamma_1}
\gamma_1 := \frac{d(q+m-1)+2q}{md+q}
\end{equation}
Then, it holds
\[
V\rho \in L_{x,t}^{\gamma_1}.
\]
Futhermore, the followin estimate holds
\begin{equation}\label{E:rho:q-2}
\left\|V\rho\right\|_{L_{x,t}^{\gamma_1}}
\leq \|V\|_{\mathcal{S}_{m,q}^{(q_1,q_2)}} \|\rho\|_{L^{r_1, r_2}_{x,t}}
\end{equation}
where the pair $(r_1, r_2)= (\frac{\gamma_1 q_1}{q_1 - \gamma_1}, \frac{\gamma_1 q_2}{q_2 - \gamma_1})$ satisfies \eqref{q-r1r2} in Lemma~\ref{P:L_r1r2}.
\end{lemma}

We deliver two compactness results: Proposition~\ref{P:AL1} in case $1 \leq q \leq m+1$ and Proposition~\ref{P:AL2} in case $q \geq \max\{1, m+1\}$. 
The first proposition is following which is new and applicable to existence results in \cite{HKK} (see Appendix~\ref{Appendix:compact}). The restriction of $V$ is the intersected region for $q_1$ and $q_2$ that gives a priori estimates (one of the conditions of $V$ in Proposition~\ref{P:Lq-energy}) and \eqref{V_compact_m+1}.

\begin{proposition}\label{P:AL1}
Let $m >1$ and $1\leq q \leq m+1$. Suppose that $\rho$ is a regular solution of \eqref{E:Main}-\eqref{E:Main-bc-ic} with $\rho_{0}\in   L^{q}(\Omega) \cap \calC^{\alpha}(\overline{\Omega})$. Moreover, assume that $V$ satisfies \eqref{V_compact_m+1} and one of hypothesis in Proposition~\ref{P:Lq-energy}.
Then, it holds that
\begin{equation}\label{P:AL1:ass}
\left\| \partial_t \rho \right\|_{W_{x}^{-1, \frac{d(q+m-1)+2q}{md+q}} L_{t}^{\frac{d(q+m-1)+2q}{md+q}}} + \left\|\nabla \rho^m\right\|_{L_{x,t}^{\frac{d(q+m-1)+2q}{md+q}}} < C
\end{equation}
for the constant $C$ depending on $\Omega$, $V$, and $\rho_0$. Furthermore, the following holds
\begin{equation}\label{P:AL1:result}
\rho^m \in L_{x,t}^{p, \frac{d(q+m-1)+2q}{md+q}} \quad \text{ for } \quad p=\frac{d^2(q+m-1) + 2dq}{md(d-1)+d-2q}.
\end{equation}
\end{proposition}

\begin{proof}
The energy estimates in \eqref{L1-energy} and \eqref{Lq-energy} in Proposition~\ref{P:Lq-energy} provide that $\rho^q \in  L^{1, \infty}_{x, t}$ and $\nabla \rho^{\frac{m+q-1}{2}} \in L^{2}_{x,t}$. Let us denote 
\[
X_0 := \left\{ \rho^q \in  L^{1}_{x}\ \text{ and } \  \nabla \rho^{\frac{m+q-1}{2}} \in L^{2}_{x} \right\}.
\]
Note that the restriction $q < m+1$ is equivalent to $m-q+1 >0$. The inequality \eqref{E:rho:q-1} yields that $\rho^{\frac{m-q+1}{2}} \in L_{x,t}^{\frac{2d(q+m-1)+4q}{d(m-q+1)}}$ which implies the following
\begin{equation}\label{AL1:01}
\nabla \rho^m = \frac{2m}{m+q-1} \rho^{\frac{m-q+1}{2}} \nabla \rho^{\frac{m+q-1}{2}}
\in L_{x,t}^{\frac{d(q+m-1)+2q}{md+q}}.
\end{equation}
By directly using \eqref{E:Main} that $\rho_t = \nabla \cdot (\nabla \rho^m - V\rho)$, we observe that
\[
\rho_t \in W_{x}^{-1, \frac{d(q+m-1)+2q}{md+q} }L_{t}^{ \frac{d(q+m-1)+2q}{md+q}}
\]
because of Lemma~\ref{L:V_compact_m+1} and \eqref{AL1:01} which holds \eqref{P:AL1:ass}. Now let
\[
X_1 := W_{x}^{-1, \frac{d(q+m-1)+2q}{md+q}}. 
\]
By applying the Aubin-Lions Lemma with the spaces $X_0$ and $X_1$, we obtain that $\rho \in L_{x}^{mp}$ where $p$ is given in \eqref{P:AL1:result}. In fact, this result directly imply \eqref{P:AL1:result} because $\frac{d(q+m-1)+2q}{md+q} \leq 2$. 
\end{proof}

Now we recall the second compactness result in case $q\geq \max\{1, m-1\}$ that are Lemma~\ref{L-compact} and Proposition~\ref{P:AL2} in \cite{HKK}.

\begin{lemma}\label{L-compact}\cite[Lemma~4.17]{HKK}
 Let $m>1$ and $q\geq \max\{1, m-1\}$. Suppose that $V \in \mathcal{S}_{m,q}^{(q_1, q_2)}$ satisfies
  \begin{equation}\label{V_compact_m-1}
 \begin{cases}
\frac{q}{d(q+m-1)+2q} \leq \frac{1}{q_1} \leq \frac{1}{\gamma_2} - \frac{q(d-2)}{d(q+m-1)}, \quad
\frac{1}{\gamma_2} - \frac{q}{q+m-1} \leq \frac{1}{q_2} \leq \frac{1}{\gamma_2}, & \text{ if } d> 2, \vspace{1 mm} \\
\frac{q}{d(q+m-1)+2q} \leq \frac{1}{q_1} < \frac{1}{\gamma_2}, \quad
\frac{1}{\gamma_2} - \frac{q}{q+m-1} < \frac{1}{q_2} \leq \frac{1}{\gamma_2}, & \text{ if } d = 2,
 \end{cases}
 \end{equation}
where
\begin{equation}\label{gamma_2}
\gamma_2 := \frac{d(q+m-1) + 2q}{d(q+m-1) + q}.
\end{equation}
Then, it holds
\[
V\rho^q \in L_{x,t}^{\gamma_2}.
\]
Futhermore, the followin estimate holds
\begin{equation}\label{E:rho:q-2}
\left\|V\rho^q\right\|_{L_{x,t}^{\gamma_2}}
\leq \|V\|_{\mathcal{S}_{m,q}^{(q_1,q_2)}} \|\rho\|_{L^{r_1, r_2}_{x,t}}^{q}
\end{equation}
where the pair $(r_1, r_2)= (\frac{\gamma_2 q q_1}{q_1 - \gamma_2}, \frac{ \gamma_2 q q_2}{q_2 - \gamma_2})$ satisfies \eqref{q-r1r2} in Lemma~\ref{P:L_r1r2}.
\end{lemma}

\begin{proposition}\label{P:AL2}\cite[Proposition~4.19]{HKK}
Let $m>1$ and $q\geq \max\{ 1, m-1 \}$.
Suppose that $\rho$ is a regular solution of \eqref{E:Main}-\eqref{E:Main-bc-ic} with $\rho_{0}\in  L^{q}(\Omega) \cap \calC^{\alpha}(\overline{\Omega})$. Moreover, assume that $V$ satisfies \eqref{V_compact_m-1} and one of hypothesis in Proposition~\ref{P:Lq-energy}.
Then, it holds that
\begin{equation}\label{rho_t_2}
\norm{\partial_t\rho^q}_{W_{x}^{-1, \frac{d(q+m-1)+2q}{d(q+m-1)+q} }L_{t}^{1}}
+\norm{\nabla  \rho^q}_{L_{x,t}^{\frac{d(q+m-1)+2q}{q(d+1)}}}<C,
\end{equation}
for the constant $C$ depending on $\Omega$, $V$, and $\rho_0$.
Furthermore, the followings holds
\begin{equation}\label{rho_2_2}
\rho^q \in L^{p, \frac{d(q+m-1)+2q}{q(d+1)}}_{x,t} \quad \text{for} \quad p= \frac{d^2(q+m+1)+2dq}{q(d^2-2) - d(m-1)}.
\end{equation}
\end{proposition}

\section{Existence of regular solutions }\label{splitting method}

In this section, we construct a regular solution of \eqref{E:Main}-\eqref{E:Main-bc-ic} when the initial data $\rho_0$ and the vector field $V$ are smooth enough.
For this, we exploit a splitting method in the Wasserstein space $\mathcal{P}_2(\Omega)$, and it turns out that our solution is in the class of absolutely continuous curves in $\mathcal{P}_2(\Omega)$.
For carrying the splitting method, it requires a priori estimates, propagation of compact support, and H\"{o}lder continuity of the following form of homogeneous PME:
\begin{equation}\label{H-PME1}
  \left\{
  \begin{array}{ll}
  \partial_t \varrho =   \Delta \varrho^m,  \quad \text{ in } \ \Omega_{T} := \Omega \times (0, T),  \vspace{2mm}\\
  \frac{\partial \varrho}{\partial \textbf{n}} = 0,  \text{ on } \partial \Omega \times (0, T), 
  \quad \text{and} \quad 
  \varrho(\cdot,0)=\varrho_0,  \text{ on } \ \Omega,
  \end{array}
  \right.
\end{equation}
for $\Omega \Subset \bbr^d$, $d\geq 2$, and $0 <T < \infty$, where $\varrho:\Omega_T\mapsto \bbr$ and the normal $\textbf{n}$ to $\partial \Omega$.

First, we deliver a priori estimates of \eqref{H-PME1} in the following lemma.
\begin{lemma}\label{P:H-PME:e}\cite[Lemma~5.1]{HKK}
Suppose that $\varrho$ is a nonnegative $L^q$-weak solution of \eqref{H-PME1} in Definition~\ref{D:weak-sol}.
\begin{itemize}
  \item[(i)] Let $\int_{\Omega} \varrho_0 (1+\log \varrho_0 ) \,dx < \infty$. Then there exists a positive constant $c=c(m,d,|\Omega|)$ such that
      \begin{equation}\label{P:H-PME:e1}
      \int_{\Omega}  \varrho (\cdot, T) \log \varrho(\cdot, T)  \,dx
      + \iint_{\Omega_T} \abs{\frac{\nabla \varrho^m}{\varrho}}^{\lambda_1} \varrho \,dx\,dt
      \leq
      \int_{\Omega}\varrho_0 \log \varrho_0  \,dx +
      c T  \int_{\Omega}\varrho_0   \,dx.
      \end{equation}
      Moreover, there exists $T^\ast = T^\ast (m,d,p)$ such that, for any positive integer $l=1,2,3,\cdots$,
        \begin{equation}\label{C:H-PME:e1}
      \int_{\Omega}  \varrho (\cdot, lT^\ast) \log \varrho(\cdot, l T^\ast)  \,dx
      + \int_{(l-1)T^\ast}^{lT^\ast}\int_{\Omega} \abs{\frac{\nabla \varrho^m}{\varrho}}^ {\lambda_1} \varrho \,dx\,dt
      \leq
      \int_{\Omega} \left[ \varrho_0 \log \varrho_0  + \varrho_0 \right] \,dx .
      \end{equation}
  \item[(ii)]Let $\varrho_0 \in L^{q}(\Omega)$ for $q>1$. Then, we have
  \begin{equation}\label{P:H-PME:e2}
      \int_{\Omega} \varrho^q(\cdot, T) \,dx + \frac{4mq(q-1)}{(m+q-1)^2} \iint_{\Omega_T} \left| \nabla \varrho^{\frac{m+q-1}{2}} \right|^2  \,dx\,dt = \int_{\Omega} \varrho_0^q \,dx.
          \end{equation}
\end{itemize}
\end{lemma}

Second, it is well known that bounded weak solutions of \eqref{H-PME1} are H\"{o}lder continuous up to the initial boundary with quantitatively determined H\"{o}lder exponents depending on given data (for example, we refer  \cites{DB93, DGV12} and references therein). In \cite[Theorem~5.2, Appendix~A]{HKK} for \eqref{H-PME1} in $\mathbb{R}^d \times (0, T)$ without the Neumann lateral boundary condition, the H\"{o}lder exponent is more clearly specified as the minimum of initial H\"{o}lder exponent and the interior H\"{o}lder exponent. On the other hand, the lateral boundary given in \eqref{H-PME1} does not affect the uniform H\"{o}lder continuity, see \cite[Theorem~III.1.3, Section~III.13]{DB93}.

\begin{theorem}\label{T:Boundary_Holder}\cite[Theorem~5.2]{HKK}
Let $\varrho$ be a bounded nonnegative weak solution of \eqref{H-PME1} where the nonnegative initial data $\varrho_0$ is bounded and H\"{o}lder continuous with the exponent $\alpha_0 \in (0,1)$. Then there exists $\alpha^\ast = \alpha^\ast (m,d)\in (0,1)$ such that $\varrho$ is locally H\"{o}lder continuous in $\Omega_{T}$. Moreover, $\varrho$ is uniformly H\"{o}lder continuous on $\Omega_T$ with the exponent $\alpha = \min\{\alpha_0, \alpha^\ast\}$; that is, $\|\varrho\|_{\mathcal{C}^{\alpha}(\Omega_T)} \leq \gamma$ where $\gamma = \gamma ( m, |\Omega|, \|\varrho_0\|_{L^{\infty}(\Omega)})$.
\end{theorem}

Our main result in this section reads as follow.
\begin{proposition}\label{proposition : regular existence}
 Let  $\tilde{\alpha} \in (0, 1)$ be the constant in Theorem \ref{T:Boundary_Holder}.
 Assume $\alpha \in (0, \tilde{\alpha}]$ and
$$
 {\rho}_0\in\mathcal{P}_2(\Omega)\cap \mathcal{C}^\alpha(\Omega) \quad \mbox{and} \quad
V \in L^1(0,T; W^{2,\infty}(\Omega))\cap
\calC^\infty (\Omega_{T}),
$$
be such that $V\cdot\mathbf{n}=0$ on $\partial \Omega$. 
Then there exists an absolutely continuous curve ${\rho}\in AC(0,T;\mathcal{P}_2(\Omega))$ which is a solution of \eqref{E:Main}-\eqref{E:Main-bc-ic}
in the sense of distributions and satisfies the followings:
\begin{itemize}
\item[(i)] For all $0\leq s<t\leq T$, we have
\begin{equation}\label{eq2 : Theorem : Toy Fokker-Plank}
W_2(\rho(s),\rho(t))\leq C\sqrt{t-s } +\int_s^t \|V(\tau)\|_{L^\infty_x}
\,d\tau 
\end{equation}
where the constant $C=C\left (\|\rho_0\|_{L^m (\Omega)}, \,
\|V\|_{L^1(0,T; W^{1,\infty}(\Omega))} \right )$.

\item[(ii)] For all $x,y \in \Omega$ and $t\in [0, T]$, we have
\begin{equation*}
\left | \rho(x,t)-\rho(y,t)\right | \leq C |x-y|^\alpha, 
\end{equation*}
where the constant $C=C\left (\|\rho_0\|_{L^\infty (\Omega)}, \, \|V\|_{L^1(0,T; C^{1,\alpha}(\Omega))} \right )$.

\item[(iii)]  For $q \geq 1$
\begin{equation}\label{KK-Feb19-100}
\int_0^T \|\nabla \rho^{\frac{m+q-1}{2}}\|_{L^2_x}^2 dt \leq  C,
\end{equation}
where the constant $C=C\left (\|\rho_0\|_{L^{m+q-1}(\Omega)},~~
\|V\|_{L^1(0,T;W^{2,\infty}(\Omega))} \right )$.

\item[(iv)] For all $t\in [0, T]$ and $q\in[1, \infty]$, we have
\begin{equation}\label{KK-Nov24-100}
\|\rho(t)\|_{L^q_x} \leq \|\rho_0\|_{L^q(\Omega)}e^{\frac{q-1}{q}\int_0^T
\|\nabla \cdot V \|_{L^\infty_x} \,d\tau} , 
\end{equation}
where $\frac{q-1}{q}=1$ if $q=\infty$.
\end{itemize}
\end{proposition}

\subsection{Splitting method}
 In this subsection, we introduce the splitting method and construct two sequences of curves in
$\mathcal{P}_2(\Omega)$ which are approximate solutions of \eqref{E:Main}-\eqref{E:Main-bc-ic} where $\rho_0$ and $V$ satisfy  the assumptions in Proposition \ref{proposition : regular existence}. 

For each $n\in \mathbb{N},$ we define approximated solutions
$\varrho_n,~ \rho_n:[0,T]\mapsto \mathcal{P}_2(\Omega)$ as follows;
\begin{itemize}

\item
 For $t\in [0,\frac{T}{n}]$, we define $\varrho_n:[0,\frac{T}{n}]\mapsto \mathcal{P}_2(\Omega)$ as the solution of
\begin{equation*}\label{eq 1 : splitting}
  \left\{
  \begin{array}{ll}
  \partial_t \varrho_n =   \Delta (\varrho_n)^m, & \text{ in } \ \Omega \times (0, \frac{T}{n}],  \vspace{1mm}\\
  \frac{\partial\varrho_n}{\partial\mathbf{n}} = 0, & \text{ on } \partial \Omega \times (0, \frac{T}{n}],  \vspace{1mm}\\
  \varrho_n(\cdot,0)=\varrho_0, & \text{ on } \ \Omega.
  \end{array}
  \right.
\end{equation*}
We also define  $\rho_n:[0,\frac{T}{n}]\mapsto \mathcal{P}_2(\Omega)$ as follows
\begin{equation*}
\rho_n(t)= \Psi(t;0, \varrho_n(t)), \quad \forall ~ t \in [0, \frac T n].
\end{equation*}
Here, we recall \eqref{Flow on Wasserstein} for the definition of $\Psi$. 

\item For $t\in (\frac{T}{n} ,\frac{2T}{n}]$, we define $\varrho_n:(\frac{T}{n} ,\frac{2T}{n}]\mapsto \mathcal{P}_2(\Omega)$ as the solution of
\begin{equation*}\label{eq 3 : splitting}
  \left\{
  \begin{array}{ll}
  \partial_t \varrho_n =   \Delta (\varrho_n)^m, & \text{ in } \ \Omega \times (\frac{T}{n}, \frac{2T}{n}],  \vspace{1mm}\\
  \frac{\partial \varrho_n}{\partial \mathbf{n}} = 0, & \text{ on } \partial \Omega \times (\frac{T}{n}, \frac{2T}{n}],  \vspace{1mm}\\
 \varrho_n(\cdot, \frac Tn)=\rho_n( \frac T n), & \text{ on } \ \Omega.
  \end{array}
  \right.
\end{equation*}
We also define  $\rho_n:(\frac{T}{n} ,\frac{2T}{n}]\mapsto \mathcal{P}_2(\Omega)$ as follows
\begin{equation*}
\rho_n(t)= \Psi(t; T/n, \varrho_n(t)), \quad \forall ~ t \in (\frac{T}{n} ,\frac{2T}{n}].
\end{equation*}

\item In general, for each $i=1, \ldots, n-1$ and $t\in (\frac{iT}{n} ,\frac{(i+1)T}{n}]$, we define $\varrho_n:(\frac{iT}{n} ,\frac{(i+1)T}{n}]\mapsto \mathcal{P}_2(\Omega)$ as the solution of
\begin{equation*}\label{eq 3 : splitting}
  \left\{
  \begin{array}{ll}
  \partial_t \varrho_n =   \Delta (\varrho_n)^m, & \text{ in } \ \Omega \times (\frac{iT}{n}, \frac{(i+1)T}{n}],  \vspace{1mm}\\
  \frac{\partial \varrho_n}{\partial \mathbf{n}} = 0, & \text{ on } \partial \Omega \times (\frac{iT}{n}, \frac{(i+1)T}{n}],  \vspace{1mm}\\
 \varrho_n(\cdot, \frac{iT}{n})=\rho_n( \frac{iT}{n}), & \text{ on } \ \Omega.
  \end{array}
  \right.
\end{equation*}
We also define  $\rho_n:(\frac{iT}{n} ,\frac{(i+1)T}{n}]\mapsto \mathcal{P}_2(\Omega)$ as follows
\begin{equation*}
\rho_n(t)= \Psi(t; iT/n, \varrho_n(t)), \quad \forall ~ t \in (\frac{iT}{n} ,\frac{(i+1)T}{n}].
\end{equation*}
\end{itemize}

 Now, we investigate some properties useful for the proof of Proposition \ref{proposition : regular existence}.
\begin{lemma}\label{Lemma:AC-curve}\cite[Lemma 5.4]{HKK}
Let $V \in L^1(0,T : W^{1,\infty}(\Omega))$ with  $V\cdot\mathbf{n}=0$ on $\partial \Omega$      and $\rho_0 \in \mathcal{P}_2(\Omega) \cap  \mathcal{C}^\alpha(\Omega)$ for some $\alpha \in (0,1)$.
Suppose that $\rho_n,\varrho_n:[0,T]\mapsto \mathcal{P}_2(\Omega)$ are curves defined as above with the initial data $\varrho_n(0),\rho_n(0):=\rho_0 $.
Then, these curves satisfy the following properties;
\begin{itemize}
\item[(i)] For all $t\in [0,T],$ we have
\begin{equation}\label{eq29:Lemma:AC-curve}
\|\varrho_{n}(t)\|_{L^q_x}, \  \|\rho_{n}(t)\|_{L^q_x} \leq \|\rho_0\|_{L^q(\Omega)}
e^{\frac{q-1}{q}\int_0^T \|\nabla \cdot V\|_{L^\infty_x} \,d\tau}, \qquad \forall ~  q \geq 1.
\end{equation}

\item[(ii)] For all $s,t \in [0,T],$ we have
\begin{equation}\label{eq20:Lemma:AC-curve}
\begin{aligned}
W_2(\rho_{n}(s), \rho_{n}(t)) &\leq C \sqrt{t-s}+ \int_s^t \|V(\tau)\|_{L^\infty_x} \,d\tau  ,
\end{aligned}
\end{equation}
where $C= C(\int_0^T \|\nabla V\|_{L^\infty_x}  \,d\tau , \|\rho_0\|_{L^m (\Omega)})$

\item[(iii)] For all $t \in [0,T],$ we have
\begin{equation}\label{eq21:Lemma:AC-curve}
\begin{aligned}
W_2(\rho_{n}(t), {\varrho}_{n}(t)) \leq  \max_{\{i=0,1,\dots,n\}}\int_{\frac{iT}{n}}^{\frac{(i+1)T}{n}} \|V(\tau)\|_{L^\infty_x} \,d\tau.
\end{aligned}
\end{equation}

\end{itemize}
\end{lemma}

\begin{lemma}\label{Lemma:H-curve}\cite[Lemma 5.5 and Lemma 5.6]{HKK}
Let $V \in L^1(0,T : W^{2,\infty}(\Omega)) \cap C^\infty(\overline{\Omega}_T)$ with $V\cdot\mathbf{n}=0$ on $\partial \Omega$ and $\rho_0 \in \mathcal{P}_2(\Omega)\cap \mathcal{C}^\alpha(\Omega)$ for some $\alpha \in (0,1)$.
Suppose that $\rho_n,\varrho_n:[0,T]\mapsto
\mathcal{P}_2(\Omega)$ are curves defined as in Lemma
\ref{Lemma:AC-curve} with the initial data
$\varrho_n(0),\rho_n(0):=\rho_0$. Then, we have
\begin{equation}\label{main : Lemma:H-curve}
\int_0^T \|\nabla \varrho_{n}^{\frac{m+q-1}{2}}\|_{L^2_x}^2 dt + \int_0^T \|\nabla (\rho_{n})^{\frac{m+q-1}{2}}\|_{L^2_x}^2 dt\leq C
\left(\|\rho_0\|_{L^{m+q-1}(\Omega)}, ~\| V\|_{ L^1(0,T : W^{2,\infty}(\Omega))}\right), \quad \forall \ q \geq 1.
\end{equation}
\end{lemma}

\begin{lemma}\label{Lemma:solving ODE}\cite[Lemma 5.7]{HKK}
Let $\varphi \in \calC_c^\infty(\overline{\Omega} \times [0,T))$ and
$\rho_n,\varrho_n:[0,T]\mapsto \mathcal{P}_2(\Omega)$ be the
curves defined as in Lemma \ref{Lemma:AC-curve}. For any $s,~t \in
[0,T]$, we define
\begin{equation*}
\begin{aligned}
E_n &:= \int_{\Omega}\varphi(x,t) ~\rho_n(x,t)\,dx - \int_{\Omega}\varphi(x,s) ~\rho_n(x,s)\,dx\\
  &\quad  - \int_{s}^{t} \int_{\Omega}  ( \nabla   \varphi \cdot V) ~\rho_n \,dx \,d\tau
  - \int_{s}^{t}
\int_{\Omega}  \{ \partial_\tau \varphi ~ \varrho_n - \nabla \varphi \cdot \nabla (\varrho_n)^m \}\,dx \,d\tau.
\end{aligned}
\end{equation*}
Then,
\begin{equation*}
\begin{aligned}
|E_n| \leq  C\frac{T}{n}\|\varphi\|_{\calC^2} e^{C\int_{s}^{t}\|\nabla V\|_{L^\infty_x} \,d\tau}
 \int^{t}_{0} \|V (\tau)\|_{W^{1,\infty}_x} \,d\tau  \longrightarrow 0, \quad as \quad n\rightarrow \infty,
\end{aligned}
\end{equation*}
where $C$ is the same constant in Lemma \ref{Lemma:AC-curve}.
\end{lemma}

\begin{lemma}\label{Lemma : equi-continuity}
Let $\tilde{\alpha}\in (0,1)$ be the constant in Theorem \ref{T:Boundary_Holder}. Suppose that $ \alpha \in (0, \tilde{\alpha}]$ and
\begin{equation*}
\rho_0 \in \mathcal{P}_2(\Omega) \cap \calC^\alpha (\Omega)  \quad \mbox{and} \quad V \in  L^1(0,T ; \calC^{1, \alpha}(\Omega))
\end{equation*}
with  $V\cdot\mathbf{n}=0$ on $\partial \Omega$.
Let $\rho_n,\varrho_n:[0,T]\mapsto \mathcal{P}_2(\Omega)$ be curves defined as in Lemma \ref{Lemma:AC-curve}.
Then,we have
 \begin{equation}\label{eq 2 : Holder of splitting}
|\varrho_n(x,t)-\varrho_n(y,t)| + |\rho_n(x,t)-\rho_n(y,t)|\leq C|x-y|^\alpha,  \qquad \forall ~ x, y \in \Omega , ~ t \in [0, T],
 \end{equation}
  where $C(\|\rho_0\|_{L^\infty(\Omega)},  \int_0^T \|\nabla V\|_{\calC^\alpha (\Omega)}  \,d\tau)$.
\end{lemma}

\begin{proof} 
 First, we suppose $t \in \left[0, \frac{T}{n}\right]$. 
Due to Theorem \ref{T:Boundary_Holder}, it holds
\begin{equation}\label{eq 1 : Lemma equi-cont}
|\varrho_n(x,t)-\varrho_n(y,t) |\leq C |x-y|^\alpha,  \quad \forall ~ x, ~y \in \Omega,
\end{equation}
where $C= C(\|\rho_0 \|_{L^\infty(\Omega)})$.
Next, due to Lemma \ref{Lemma : Holder regularity on the flow}, we have
\begin{equation}\label{eq 2 : Lemma equi-cont}
|\rho_n(x,t)-\rho_n(y,t) | \leq C |x-y|^{\alpha}, \quad   \forall ~x, ~ y \in \Omega,
\end{equation}
where $C= C \left(\text{sup}_{t \in \left[0, \frac{T}{n}\right ]} \|\varrho_n(t)\|_{\calC^\alpha (\Omega)}, \,\int_0^{T/n} \|\nabla V \|_{\calC^\alpha (\Omega)} \,d\tau \right)$.

We combine \eqref{eq 1 : Lemma equi-cont} and \eqref{eq 2 : Lemma equi-cont}, and exploit \eqref{eq29:Lemma:AC-curve} for $q=\infty$ to conclude
 \begin{equation}\label{eq 3 : Lemma equi-cont}
|\varrho_n(x,t)-\varrho_n(y,t)|+ |\rho_n(x,t)-\rho_n(y,t)|\leq C|x-y|^\alpha,  \qquad \forall ~x, ~ y \in \Omega ,
 \end{equation}
where $C(\|\rho_0\|_{L^\infty(\Omega)}, ~ \int_0^{T/n} \|\nabla V \|_{\calC^\alpha (\Omega)} \,dt)$.

For the general case $t \in \left [\frac{iT}{n}, \frac{(i+1)T}{n}\right ], ~ i=1, 2, \dots, n-1$, we repeat what we did in the first step and get
 \begin{equation}\label{eq 4 : Lemma equi-cont}
|\varrho_n(x,t)-\varrho_n(y,t)| +|\rho_n(x,t)-\rho_n(y,t)|\leq C|x-y|^\alpha, \qquad \forall ~ x, ~ y \in \Omega ,
 \end{equation}
where $C(\|\rho(iT/n)\|_{L^\infty_x},  \int_{iT/n}^{(i+1)T/n} \|\nabla V \|_{\calC^\alpha (\Omega)} \,d\tau)$. We combine
\eqref{eq 3 : Lemma equi-cont} and \eqref{eq 4 : Lemma equi-cont} with \eqref{eq29:Lemma:AC-curve} for $q=\infty$, and conclude \eqref{eq 2 : Holder of splitting}
which completes the proof.
\end{proof}

\subsection{Proof of Proposition \ref{proposition : regular existence}}
Since the proof is similar to the proof of Proposition 5.3 in \cite{HKK}, we just give a sketch of the proof. %
Let $\rho_n,\varrho_n:[0,T]\mapsto \mathcal{P}_2(\Omega)$ be as defined in the splitting method. Then, from Lemma \ref{Lemma:solving ODE}, we have
for any $\varphi \in \calC_c^\infty(\overline{\Omega} \times [0,T))$ and $s,~t \in [0,T]$,
\begin{equation}\label{eq3 : Theorem : Toy Fokker-Plank}
\begin{aligned}
& \int_{\Omega}\varphi(x,t) \rho_n(x,t)\,dx - \int_{\Omega}\varphi(x,s) \rho_n(x,s)\,dx\\
&= \int_{s}^{t} \int_{\Omega}   (\nabla   \varphi \cdot V) \rho_n \,dx  \,d\tau + \int_{s}^{t} \int_{\Omega}
\left\{ \partial_\tau\varphi  \varrho_n -  \nabla \varphi \cdot \nabla(\varrho_n)^m\right\}\,dx \,d\tau
 + E_n,
\end{aligned}
\end{equation}
with $|E_n| \longrightarrow 0 $ as $n \rightarrow \infty$.

 Recalling Lemma \ref{Lemma:AC-curve}, we know%
\begin{equation}\label{eq5 : Theorem : Toy Fokker-Plank}
W_2(\rho_n(s),\rho_n(t))\leq C\sqrt{t-s} + \int_s^t \|V(\tau)\|_{L^\infty_x}  \,d\tau
\end{equation}
and
\begin{equation}\label{eq6 : Theorem : Toy Fokker-Plank}
\begin{aligned}
W_2(\rho_n(t), {\varrho}_n(t)) \leq  \max_{\{i=0,1,\dots,n-1\}}\int_{\frac{iT}{n}}^{\frac{(i+1)T}{n}} \|V(\tau)\|_{L^\infty_x} \,d\tau,
\end{aligned}
\end{equation}
where $C=C\left(  \| \rho_0\|_{L^m(\Omega)},  \,  \,\|V\|_{L^1(0,T; W^{1,\infty}(\Omega))} \right) $.
Combining the estimates \eqref{eq5 : Theorem : Toy Fokker-Plank} and Lemma \ref{Lemma : Arzela-Ascoli}, there exist a subsequence (by abusing notation)
$\rho_n$ and a limit curve  $\rho :[0,T]\mapsto \mathcal{P}_2(\Omega)$ such that
\begin{equation}\label{eq11 : Theorem : Toy Fokker-Plank}
\rho_{n}(t) ~~ \mbox{ converges to} ~~\rho(t) ~~ \mbox{ in} ~ \mathcal{P}_2({\Omega}), \qquad \text{for all} ~~ t\in[0,T].
\end{equation}
 Due to \eqref{eq6 : Theorem : Toy Fokker-Plank}, we note that $\varrho_n(t)$ also converges to $\rho(t), \, \forall ~ t \in [0,T]$.
 This implies
\begin{equation}\label{eq 1 : proposition : regular existence}
\begin{aligned}
\int_{\Omega}\varphi(x,t) \rho_n(x,t)\,dx - \int_{\Omega}\varphi(x,s) \rho_n(x,s)\,dx \longrightarrow \int_{\Omega}\varphi(x,t) \rho(x,t)\,dx - \int_{\Omega}\varphi(x,s) \rho(x,s)\,dx,
\end{aligned}
\end{equation}
and
\begin{equation}\label{eq 2 : proposition : regular existence}
\begin{aligned}
\int_{s}^{t} \int_{\Omega}  \left[ (\nabla   \varphi \cdot V) \rho_n + \partial_\tau\varphi  \varrho_n \right] \,dx  \,d\tau
\longrightarrow \int_{s}^{t} \int_{\Omega}  \left[ (\nabla   \varphi\cdot V) + \partial_\tau\varphi  \right]\rho \,dx  \,d\tau,
\end{aligned}
\end{equation}
as $n \rightarrow \infty$. Furthermore,  from Lemma \ref{Lemma : equi-continuity}, $\rho_n$ and $\varrho_n$ are equi-continuous with respect to the space variable.
We combine this equi-continuity with  the uniform bound \eqref{eq29:Lemma:AC-curve} to claim that $\rho_n$ and  $\varrho_n$ actually pointwise converge to $\rho$.
Furthermore, \eqref{main : Lemma:H-curve} with $q=m+1$ implies that $\nabla\varrho_n$ is bounded in $L^2([0,T]\times\Omega)$ and hence it (up to a subsequence) has a weak limit  $L^2([0,T]\times\Omega)$.
Combining this with the pointwise convergence of $\varrho_n^m$ to $\rho^m$, we have
\begin{equation}\label{eq 3 : proposition : regular existence}
\int_{s}^{t} \int_{\Omega}  \nabla \varphi \cdot \nabla(\varrho_n )^m    \,dx \,d\tau \longrightarrow \int_{s}^{t} \int_{\Omega}  \nabla \varphi \cdot \nabla \rho^m  \,dx \,d\tau.
\end{equation}
Finally, we put \eqref{eq 1 : proposition : regular existence}, \eqref{eq 2 : proposition : regular existence} and \eqref{eq 3 : proposition : regular existence} into
\eqref{eq3 : Theorem : Toy Fokker-Plank} to get
\begin{equation}\label{eq 4 : proposition : regular existence}
 \int_{\Omega}\varphi(x,t) \rho(x,t) \,dx - \int_{\Omega}\varphi(x,s) \rho(x,s)\,dx
=\int_{s}^{t} \int_{\Omega}  \left [ \left(\partial_\tau\varphi  + V \cdot \nabla \varphi \right){\rho}
-\nabla  \varphi \cdot \nabla \rho^m   \right ] \,dx\,d\tau.
\end{equation}

Next, from Lemma \ref{Lemma : equi-continuity} and the pointwise convergence of $\rho_n$ and  $\varrho_n$, we have
\begin{equation*}
\left |\rho (x,t)-\rho(y,t)\right | \leq C |x-y|^\alpha, \quad \forall ~x,~y \in \Omega, \, t \in [0,T].
\end{equation*}
We note,  from \eqref{eq29:Lemma:AC-curve}
\begin{equation}\label{eq7 : Theorem : Toy Fokker-Plank}
\|\rho(t)\|_{L^q_x} \leq \|\rho_0\|_{L^q(\Omega)}e^{\frac{q-1}{q}\int_0^T \|\nabla \cdot V \|_{L^\infty_x} \,d\tau}, \quad \forall ~t\in[0,T],
\end{equation}
which gives \eqref{KK-Nov24-100}. 
From \eqref{eq11 : Theorem : Toy Fokker-Plank}  and \eqref{eq5 : Theorem : Toy Fokker-Plank}, we have \eqref{eq2 : Theorem : Toy Fokker-Plank}.
Note that \eqref{main : Lemma:H-curve} 
implies that $\nabla \rho_n^{\frac{m+q-1}{2}}$ (up to a subsequence)
has a weak limit $\eta \in L^2(\Omega\times[0,T])$ and
\begin{equation}\label{eq9 : Theorem : Toy Fokker-Plank}
\iint_{\Omega_T}|\eta(x,t) |^2\,dx \,dt \leq C,
\end{equation}
where $C=C\left(\|\rho_0 \|_{L^{m+q-1}(\Omega)},\,\|V\|_{L^1(0,T;W^{2,\infty}(\Omega)} \right)$.
Since  $\rho_n$ pointwise converges to $\rho$, we conclude $\eta = \nabla \rho^{\frac{m+q-1}{2}}$ and \eqref{KK-Feb19-100} follows from \eqref{eq9 : Theorem : Toy Fokker-Plank}.
This completes the proof.
\qed


\section{Existence of weak solutions}\label{Exist-weak}
In this section, we show the existence of $L^q$-weak solutions in Definition~\ref{D:weak-sol}.
 As mentioned before, since the critical case matters most and the strict-subcritical case is simpler, we only give proof for the critical cases.
\subsection{Existence for case: $V$ in sub-scaling classes}
\subsubsection{Proof of Theorem \ref{T:weakSol}}
{\it Proof of (i).} Let $\rho_0\in  \mathcal{P}(\Omega) $ with $\int_{\Omega}\rho_0 \log \rho_0 \,dx < \infty$ and
$V \in  \mathcal{S}_{m,1}^{(q_1,q_2)} $ be satisfying \eqref{T:weakSol:V_1}.
Then there exists a sequence of vector fields $V_n\in C^\infty_c(\overline{\Omega}\times[0,T)) $ with $V(t)\cdot \mathbf{n}=0$ on $\partial\Omega$ for all $t\in[0,T)$
such that
\begin{equation}\label{eq7 : Theorem-1}
\lim_{n \rightarrow \infty}\|V_n - V \|_{\mathcal{S}_{m,1}^{(q_1,q_2)} } =0.
\end{equation}
Using truncation, mollification and normalization, we have a sequence of functions
$\rho_{0,n}\in \calC_c^\infty(\Omega)\cap \mathcal{P}_2(\Omega)$ satisfying
\begin{equation}\label{eq9 : Theorem-1}
\lim_{n \rightarrow \infty} W_2(\rho_{0,n} , \rho_0) =0,
\quad \text{and} \quad  \lim_{n \rightarrow \infty} \int_{\Omega} \rho_{0,n} \log \rho_{0,n} \,dx = \int_{\Omega} \rho_0\log \rho_{0} \,dx.
\end{equation}
Exploiting Proposition \ref{proposition : regular existence}, we have $\rho_n\in AC(0,T; \mathcal{P}_2(\Omega))$ which is a regular solution of
\begin{equation}\label{eq5 : Theorem-1}
\begin{cases}
   \partial_t \rho_n =   \nabla \cdot( \nabla (\rho_n)^m  - V_n\rho_n),\\
   \rho_n(0)=\rho_{0,n}.
 \end{cases}
\end{equation}
Thus, $\rho_n$ satisfies
\begin{equation}\label{eq21 : Theorem-1}
\begin{aligned}
\iint_{\Omega_T}  \left [\partial_t\varphi + (V_n \cdot \nabla \varphi )\right ]\rho_n
 - \nabla(\rho_n)^m \cdot \nabla \varphi  \,dx\,dt
 =  - \int_{\Omega}\varphi(0) \rho_{0,n}\,dx,
\end{aligned}
\end{equation}
for each $\varphi\in C_c^\infty( \overline{\Omega} \times [0,T))$, and $\rho_n$ enjoys all properties in Proposition \ref{proposition :
regular existence}. Especially, we note
\begin{equation}\label{eq - regular}
\sup_{0 \leq t \leq T}\int_{\Omega \times \{ t\}}\rho_n \,dx + \iint_{\Omega_T} \left|\nabla \rho_{n}^{\frac{m}{2}} \right|^2 \,dx\,dt < \infty.
\end{equation}
 Since \eqref{T:weakSol:V_1} implies \eqref{V-L1-energy}, for $1<m \leq 2$, we exploit Proposition \ref{P:Lq-energy} (i) to have
\begin{equation}\label{eq17 : Theorem-1}
\sup_{0 \leq t \leq T}\int_{\Omega \times \{ t\}}\rho_n |\log \rho_n|  \,dx
+ \iint_{\Omega_T} \left|\nabla \rho_{n}^{\frac{m}{2}} \right|^2 \,dx\,dt <C,
\end{equation}
 where the constant $C=C \left(\|V_n\|_{\mathcal{S}_{m, 1}^{(q_1,q_2)}}, \int_{\Omega} \rho_{0,n} \log \rho_{0,n} \,dx\right)$.
 Actually, due to \eqref{eq7 : Theorem-1} and \eqref{eq9 : Theorem-1}, we may choose the constant
$C=C\left(\|V\|_{\mathcal{S}_{m, 1}^{(q_1,q_2)}}, \int_{\Omega} \rho_{0} \log \rho_{0} \,dx\right)$.

Now, we investigate the convergence of $\rho_n$. First of all, we note that \eqref{T:weakSol:V_1} implies \eqref{V_compact_m+1} for $q=1$.
Hence, we exploit Proposition \ref{P:AL1} with $q=1$ to have
\begin{equation}\label{eq26 : Theorem-1}
\| \partial_t\rho_n\|_{W_x^{-1,\frac{md+2}{md+1}}L_t^{\frac{md+2}{md+1}}} + \|\nabla \rho_n^m \|_{L_{x,t}^{\frac{md+2}{md+1}}}
 \leq C ,
\end{equation}
and 
$\rho_n$ (up to a subsequence)  converges to $\rho$ in $L_{x, t}^{mp, \frac{m(md+2)}{md+1}}$ for $p< \frac{md^2 +2d}{md(d-1)+d-2}$.
 From \eqref{eq26 : Theorem-1}, we note $\nabla \rho_n^m$ has a weak limit $\eta$ in $L_{x,t}^{\frac{md+2}{md+1}}$. Due to the strong convergence of $\rho_n$ to $\rho$, we have
$\eta=\nabla \rho^m$. We put these together and get
\begin{equation}\label{eq3 : Theorem-1}
\begin{aligned}
\iint_{\Omega_T}  \left[\partial_t\varphi \rho_n  - \nabla(\rho_n)^m \cdot \nabla \varphi\right ] \,dx\,dt \longrightarrow
\iint_{\Omega_T}  \left[\partial_t\varphi \rho  - \nabla \rho^m \cdot \nabla \varphi\right ] \,dx\,dt, ~\mbox{as} \ n\rightarrow \infty.
\end{aligned}
\end{equation}

Next, we exploit Lemma \ref{L:V_compact_m+1} to have
\begin{equation}\label{eq32 : Theorem-1}
\begin{aligned}
\iint_{\Omega_T}  (V_n \cdot \nabla \varphi  ) \rho_n \,dx\,dt
\longrightarrow  \iint_{\Omega_T}  (V \cdot \nabla \varphi ) \rho \,dx\,dt, \,~\mbox{as} \ n\rightarrow \infty,
\end{aligned}
\end{equation}
for any $\varphi \in \calC_c^\infty(\overline{\Omega} \times [0,T)) $. Indeed, we have
\begin{equation*}
\begin{aligned}
&\iint_{\Omega_T}  \big[(V_n \cdot \nabla \varphi ) \rho_n
- ( V \cdot \nabla \varphi ) \rho\big ] \,dx\,dt \\
&\quad = \iint_{\Omega_T}  \nabla \varphi \cdot \big (V_n- V\big )\rho_n  \,dx\,dt
+ \iint_{\Omega_T}  \nabla \varphi \cdot  V \bke{\rho_n- \rho}  \,dx\,dt
 = I + II.
\end{aligned}
\end{equation*}
 Let us first estimate $I$.  Due to \eqref{eq17 : Theorem-1}, we have Lemma \ref{P:L_r1r2}.
 Once $V$ satisfies \eqref{V_compact_m+1} with $q=1$, we exploit Lemma \ref{L:V_compact_m+1} to obtain
\begin{equation}\label{eq28 : Theorem-1}
\|(V_n-V)\rho_n \|_{L_{x,t}^{\frac{md+2}{md+1}}}
\leq \|V_n-V \|_{{\mathcal{S}}_{m,1}^{({q}_1, {q}_2)}} \| \rho_n\|_{L_{x,t}^{r_1, r_2}}
 \leq C 
\|V_n-V \|_{\mathcal{S}_{m,1}^{(q_1,q_2)}} ,
\end{equation}
where we used the fact that $\| \rho_n\|_{L_{x,t}^{r_1, r_2}} \leq C \left(\|\rho_n\|_{L_{x,t}^{1,\infty}},~ \|\nabla \rho_n^{\frac{m}{2}}\|_{L_{x,t}^{2,2}} \right)
=C\left(\|V\|_{\mathcal{S}_{m, 1}^{(q_1,q_2)}}, \int_{\Omega} \rho_{0} \log \rho_{0} \,dx\right)$
from Lemma \ref{P:L_r1r2} and \eqref{eq17 : Theorem-1}.
Hence, we have
\begin{equation}\label{eq13 : Theorem-1}
| I |\leq \|\nabla \varphi \|_{L_{x,t}^{md+2}} \|(V_n - V)\rho_n \|_{L_{x,t}^{\frac{md+2}{md+1}}}
\leq C \|V_n - V \|_{\mathcal{S}_{m,1}^{(q_1,q_2)}} \longrightarrow 0, \quad \mbox{as} \quad n\rightarrow \infty.
\end{equation}
 Also, the weak convergence of $\rho_n$ implies $II$ converges to $0$ as $n \rightarrow \infty$.  Hence, this with \eqref{eq13 : Theorem-1} implies \eqref{eq32 : Theorem-1}.
 We put \eqref{eq3 : Theorem-1}, \eqref{eq32 : Theorem-1} and $W_2(\rho_{0,n}, \rho_0) \rightarrow 0$ into \eqref{eq21 : Theorem-1}, and get
\begin{equation}\label{eq10 : Theorem-1}
\begin{aligned}
\iint_{\Omega_T}  \left [\partial_t\varphi + (V \cdot \nabla \varphi )\right ]\rho
 - \nabla\rho^m \cdot \nabla \varphi  \,dx\,dt
 =  - \int_{\Omega}\varphi(0) \rho_{0}\,dx,
\end{aligned}
\end{equation}
for each $\varphi\in C_c^\infty( \overline{\Omega} \times [0,T))$.

 Next, we prove \eqref{T:weakSol:E_1}.  From \eqref{eq17 : Theorem-1}, we have
\begin{equation}\label{eq24 : Theorem-1}
\| \nabla \rho_n^{\frac{m}{2}}\|_{L_{x,t}^2} \leq C.
\end{equation}
Using lower semi-continuity of $L^2-$norm with respect to the weak convergence and the fact that $\nabla \rho_n^{\frac{m}{2}}$ weakly converges to $\nabla \rho^{\frac{m}{2}}$
 in $L^2(\overline{\Omega}_{T})$,  we have
\begin{equation}\label{eq18 : Theorem-1}
\begin{aligned}
\|\nabla \rho^{\frac{m}{2}} \|_{L_{x,t}^{2}} 
& \leq C.
\end{aligned}
\end{equation}
 We remind that the entropy is lower semicontinuous with respect to the narrow convergence (refer \cite{ags:book}). We also note that the strong convergence of $\rho_n$ to $\rho$ implies
 $\rho_n(t)$ narrowly converges to $\rho(t)$ for a.e $t \in [0,T)$.
  Hence, from \eqref{eq17 : Theorem-1}, we have
\begin{equation}\label{eq15 : Theorem-1}
\begin{aligned}
\esssup_{0\leq t\leq T}\int_{\Omega\times \{t\}} \rho  \log \rho \,dx
  &\leq \liminf_{n \rightarrow \infty} \int_{\Omega\times \{t\}} \rho_n  \log \rho_n  \,dx
  \leq C.
\end{aligned}
\end{equation}
We also remind that there exists a constant $c>0$ (independent of $\rho$) such that
\begin{equation}\label{eq16 : Theorem-1}
\begin{aligned}
\int_{\Omega} | \min \{\rho \log \rho,~0 \} | \,dx \leq c \int_{\Omega}  \rho  \langle x \rangle^2  \, dx \leq c (1 + |\Omega|^2).
\end{aligned}
\end{equation}
Combining \eqref{eq15 : Theorem-1} and \eqref{eq16 : Theorem-1}, we have
\begin{equation}\label{eq20 : Theorem-1}
\begin{aligned}
\esssup_{0\leq t \leq T} &\int_{\Omega\times \{t\}} \rho  \abs{\log \rho}  \,dx \leq  C.
\end{aligned}
\end{equation}
We combine \eqref{eq18 : Theorem-1} and \eqref{eq20 : Theorem-1} to get \eqref{T:weakSol:E_1} which completes the proof for the case {\it (i)}.

 Before we end the proof of {\it (i)}, we remark that the proof works whenever approximate solutions satisfy the energy estimate and the compactness results in
 Proposition \ref{P:Lq-energy} and Proposition \ref{P:AL1}, respectively. These results follow if $V \in \mathcal{S}_{m,1}^{(q_1,q_2)}$  satisfies \eqref{V-L1-energy}
and \eqref{V_compact_m+1} with $q=1$. We note that the intersection of these two conditions is \eqref{T:weakSol:V_1}.

{\it Proof of (ii)} :
Suppose $\rho_0\in  \mathcal{P}(\Omega) \cap L^{q}(\Omega)$ and $V \in  \mathcal{S}_{m,q}^{(q_1,q_2)} $ satisfies \eqref{T:weakSol:V_q}.
Then there exists a sequence of vector fields $V_n\in C^\infty_c(\overline{\Omega}\times[0,T)) $ with $V(t)\cdot \mathbf{n}=0$ on $\partial\Omega$ for all $t\in[0,T)$
such that
\begin{equation}\label{eq7 : Theorem-2-a}
\lim_{n \rightarrow \infty}\|V_n - V \|_{\mathcal{S}_{m,q}^{(q_1,q_2)} } =0.
\end{equation}
Using truncation, mollification and normalization, we have a sequence of functions
$\rho_{0,n}\in \calC_c^\infty(\Omega)\cap \mathcal{P}(\Omega)$ satisfying
\begin{equation}\label{eq9 : Theorem-2-a}
\lim_{n \rightarrow \infty}\|\rho_{0,n} - \rho_0 \|_{L^q(\Omega)} =0.
\end{equation}
As in the proof of part {\it (i)}, exploiting Proposition \ref{proposition : regular existence},  we have $\rho_n\in AC(0,T; \mathcal{P}_2(\Omega))$ a regular  solution of \eqref{eq5 : Theorem-1} satisfying \eqref{eq21 : Theorem-1} 
and
\begin{equation}\label{eq - regular - q}
\sup_{0 \leq t \leq T}\int_{\Omega \times \{ t\}} (\rho_n)^q \,dx + \iint_{\Omega_T} \left|\nabla \rho_{n}^{\frac{q+m-1}{2}} \right|^2 \,dx\,dt < \infty.
\end{equation}

Since \eqref{T:weakSol:V_q} implies \eqref{V-Lq-energy}, from Proposition \ref{P:Lq-energy} (ii), we have
\begin{equation}\label{eq12 : Theorem-2-a}
\begin{aligned}
\sup_{0\leq t \leq T} &\int_{\Omega\times \{t\}}  (\rho_n)^q   \,dx
 + \frac{2qm(q-1)}{(q+m-1)^2} \iint_{\Omega_T} \left |\nabla \rho_n^{\frac{q+m-1}{2}}\right |^2  \,dx\,dt \leq C,
\end{aligned}\end{equation}
where $C= C( \|V_n\|_{\mathcal{S}_{m, q}^{(q_1,q_2)}}, \|\rho_{0,n}\|_{L^q(\Omega)} )$.
Due to \eqref{eq7 : Theorem-2-a} and \eqref{eq9 : Theorem-2-a}, we may choose
$C= C( \|V\|_{\mathcal{S}_{m, q}^{(q_1,q_2)}}, \|\rho_{0}\|_{L^q(\Omega)} )$.

For the convergence of $\rho_n$, we consider two cases depending the range of $m$ and $q$.
\begin{itemize}
\item {\it For the case } $1<m\leq 2, ~~ q>1$ : Note that if $V$ satisfies \eqref{T:weakSol:V_q} for $q>1$ then it satisfies \eqref{T:weakSol:V_1}. Hence, from the proof of (i),
we recall $\rho_n$ (up to a subsequence) converges to $\rho$ in $L_{x,t}^{mp, \frac{m(md+2)}{md+1}}$ for $p< \frac{md^2+2d}{md(d-1)+d-2}$ and this $\rho$ is a weak solution. That is, $\rho$
satisfies \eqref{eq10 : Theorem-1}.

\item {\it For the case } $m > 2, ~~ q \geq m-1$ : Note that if $\rho_0 \in L^q(\Omega)$ for $q\geq m-1$ with $m>2$ then $\rho_0 \in L^{m-1}(\Omega)$.
Hence, we focus on the case $q=m-1$. For the convergence of $\rho_n$, we follow similar steps as in the proof of (i). First, we exploit Proposition \ref{P:AL2} with $q=m-1$ and get
\begin{equation}\label{eq30 : Theorem-1}
\| \partial_t\rho_n^{m-1}\|_{W_x^{-1,\frac{2d+2}{2d+1}}L_t^{1}} + \|\nabla \rho_n^{m-1} \|_{L_{x,t}^{2}}
 \leq C ,
\end{equation}
and 
$\rho_n$ (up to a subsequence)  converges to $\rho$ in $L_{x, t}^{(m-1)p, (m-1)2}$.
Next, we note
\begin{equation}
\nabla \rho_n^m \sim \rho_n \nabla \rho_n^{m-1} \in L_{x,t}^{r,s}
\end{equation}
where $\frac{1}{r}=\frac{1}{(m-1)p}+\frac{1}{2}$ and $\frac{1}{s}=\big (\frac{1}{(m-1)2} +\frac{1}{2} \big)\frac{1}{r}$. We can check $r,~s >1$. Hence,
 $\nabla \rho_n^m$ has a weak limit $\eta$ in $L_{x,t}^{r,s}$. Due to the strong convergence of $\rho_n$ to $\rho$, we have
$\eta=\nabla \rho^m$. We put these together and have
\begin{equation}\label{eq31 : Theorem-1}
\begin{aligned}
\iint_{\Omega_T}  \left[\partial_t\varphi \rho_n  - \nabla(\rho_n)^m \cdot \nabla \varphi\right ] \,dx\,dt \longrightarrow
\iint_{\Omega_T}  \left[\partial_t\varphi \rho  - \nabla \rho^m \cdot \nabla \varphi\right ] \,dx\,dt, ~\mbox{as} \ n\rightarrow \infty.
\end{aligned}
\end{equation}
Also, from \eqref{eq30 : Theorem-1}, we note that $\rho_n$ is bounded in $L^{r_1, r_2}_{x,t}(\overline{\Omega}_T)$ where $(r_1, r_2)$ are satisfying \eqref{q-r1r2} with $q=m-1$.
%
Exploiting this, similar to \eqref{eq32 : Theorem-1}, we have
\begin{equation}\label{eq33 : Theorem-1}
\begin{aligned}
\iint_{\Omega_T}  (V_n \cdot \nabla \varphi  ) \rho_n
\,dx\,dt \longrightarrow  \iint_{\Omega_T}  (V \cdot \nabla
\varphi) \rho \,dx\,dt, ~~~\mbox{as}~~ n\rightarrow \infty,
\end{aligned}
\end{equation}
for any $\varphi \in \calC_c^\infty(\Omega \times [0,T)) $. The only difference is that, instead of \eqref{eq28 : Theorem-1} and \eqref{eq13 : Theorem-1}, we have
\begin{equation}\label{eq34 : Theorem-1}
\begin{aligned}
I &\leq \left(\iint_{\Omega_T} \left|(V_n-V) \rho_n^{m-1} \nabla \varphi \right| \,dxdt \right)^{\frac{1}{m-1}} \left(\iint_{\Omega_T} \abs{ (V_n-V)  \nabla \varphi} \,dxdt \right)^{1-\frac{1}{m-1}} \\
&\leq \left(\iint_{\Omega_T} \abs{ (V_n-V) \rho_n^{m-1}  \nabla \varphi }^{r} \,dxdt \right)^{\frac{1}{(m-1)r}} \left(\iint_{\Omega_T} \abs{\nabla \varphi} \,dxdt \right)^{\frac{1}{m-1}(1-\frac{1}{r})} \left(\iint_{\Omega_T} \abs{ (V_n-V) \nabla \varphi} \,dxdt \right)^{1-\frac{1}{m-1}} \\
& \leq C(\varphi) \|V_n-V\|_{\mathcal{S}_{m,m-1}^{(q_1,q_2)}}^{\frac{1}{m-1}} \|\rho_n\|_{L^{r_1, r_2}_{x,t}} \|V_n-V\|_{\mathcal{S}_{m,m-1}^{(q_1,q_2)}}^{1-\frac{1}{m-1}}\\
& \leq C(\varphi) \|V_n-V\|_{\mathcal{S}_{m,m-1}^{(q_1,q_2)}} \|\rho_n\|_{L^{r_1, r_2}_{x,t}} \longrightarrow 0, \quad \mbox{as}\quad n\rightarrow \infty,
\end{aligned}
\end{equation}
where we use Lemma \ref{P:L_r1r2} and Lemma \ref{L-compact}. 
\end{itemize}

Now, we investigate \eqref{T:weakSol:E_q}. In either case $1<m\leq 2, ~ g>1$ or $m\geq 2, ~ g \geq m-1$ above, we have a strong convergence of $\rho_n$ to $\rho$.
Using \eqref{eq12 : Theorem-2-a} and  the strong convergence of $\rho_n$, we follow similar arguments in the
proof of part {\it (i)} and have
\begin{equation}\label{eq24 : Theorem-2-a}
\begin{aligned}
\esssup_{0\leq t\leq T}\int_{\Omega \times \{t\}} \rho^q \,dx + \iint_{\Omega_T} \abs{\nabla \rho^{\frac{q+m-1}{2}}}^2 \,dxdt  \leq C.
\end{aligned}
\end{equation}
Finally, we combine \eqref{eq10 : Theorem-1} and \eqref{eq24 : Theorem-2-a} to complete the proof of (ii). We note that the assumption \eqref{T:weakSol:V_q} on $V$ satisfies
the conditions \eqref{V-Lq-energy} for energy estimate and the conditions for compactness results in Proposition \ref{P:AL1} and Proposition \ref{P:AL2}.
\qed

\subsubsection{Proof of Theorem \ref{T:ACweakSol}}
{\it Proof of (i).} 
Let $V_n$, $\rho_{0,n}$ and $\rho_n$ be the same as in the proof of (i) of Theorem \ref{T:weakSol}. If $V$ satisfies the assumptions for the energy estimate and the speed estimates,
that are Proposition \ref{P:Lq-energy} (i) and Proposition \ref{P:Energy-speed} (i), respectively. Then, similar to the proof of (i) in Theorem \ref{T:weakSol}, we have
\begin{equation}\label{eq11 : Theorem-2-a}
\sup_{0 \leq t \leq T}\int_{\Omega \times \{ t\}}\rho_n |\log \rho_n|  \,dx
+ \iint_{\Omega_T} \{ \left|\nabla \rho_{n}^{\frac{m}{2}} \right|^2 + \left (\left | \frac{\nabla \rho_n^m}{\rho_n}\right |^{\lambda_1} +|V_n|^{\lambda_1} \right )\rho_n \}\,dx\,dt <C,
\end{equation}
where the constant $C=C \big(\|V\|_{\mathcal{S}_{m, 1}^{(q_1,q_2)}}, \int_{\Omega} \rho_0 \log \rho_0 \,dx\big)$.
Now, we prove $\rho_n \in AC(0,T;\mathcal{P}_p(\Omega))$. For this, 
we rewrite \eqref{eq5 : Theorem-1} as follows
$$\partial_t \rho_n +  \nabla \cdot (w_n \rho_n )=0, \qquad \mbox{where}\quad
w_n:= -\frac{\nabla (\rho_n)^m}{\rho_n}+ V_n.$$
Then, \eqref{eq11 : Theorem-2-a} says
\begin{equation*}\label{eq25 : Theorem-1}
\sup_{0 \leq t \leq T}\int_{\Omega \times \{ t\}}\rho_n |\log \rho_n|  \,dx +\int_0^T  |w_{n}(t)|_{L^p(\rho_n(t))}^p \,dt <C.
\end{equation*}
We exploit Lemma \ref{representation of AC curves} and \eqref{eq11 : Theorem-2-a} to have $\rho_n \in AC(0,T;\mathcal{P}_p(\Omega))$ and
\begin{equation}\label{eq35 : Theorem-a}
\begin{aligned}
W_p(\rho_n(s),\rho_n(t)) &\leq \int_s^t  \|w_n(\tau)\|_{L^p(\rho_n(\tau))} \,d\tau \leq  C (t-s)^{\frac{p-1}{p}}, \qquad \forall ~ 0\leq s\leq t\leq T.
\end{aligned}
\end{equation}
Next, from  Lemma \ref{Lemma : Arzela-Ascoli} with  \eqref{eq35 : Theorem-a}, we note that there
exists a curve $\rho:[0,T] \mapsto \mathcal{P}_p(\Omega)$ such that
\begin{equation}\label{eq36 : Theorem-1}
\rho_n(t) \,\mbox{(up to a subsequence) \,narrowly converges to}\, \rho(t) \,\mbox{as} \, n \rightarrow \infty, \quad \forall \ 0\leq t \leq T.
\end{equation}
Due to the lower semicontinuity of $p$-Wasserstein distance $W_p(\cdot, \cdot)$ with respect to the narrow convergence, we have
\begin{equation}\label{eq37 : Theorem-1}
 W_p(\rho(s),\rho(t)) \leq   C (t-s)^{\frac{p-1}{p}},
\qquad \forall ~ 0\leq s\leq t\leq T,
\end{equation}
which concludes \eqref{T:ACweakSol:W_1}.

Furthermore, as we have seen in the proof of (i) in Theorem \ref{T:weakSol} if $V$ satisfies \eqref{V_compact_m+1} with $q=1$ for the compactness, then
the curve $\rho :[0,T] \mapsto \mathcal{P}_p(\Omega)$ is a solution of \eqref{E:Main}-\eqref{E:Main-bc-ic} and satisfies \eqref{eq15 : Theorem-1}.

Next, we investigate \eqref{T:ACweakSol:E_1}. From Lemma \ref{P:W-p}, once we have \eqref{eq15 : Theorem-1}, if $V$ satisfies \eqref{V:speed}, then we have
\begin{equation}\label{eq1 : Theorem-2-a}
 \iint_{\Omega_{T}} \left (\left | \frac{\nabla \rho^m}{\rho}\right |^{\lambda_1} +|V|^{\lambda_1} \right )\rho \,dx\,dt <C,
\end{equation}
where $C=C(\|\rho\|_{L_{x,t}^{1,\infty}},~ \|V\|_{\mathcal{S}_{m,1}^{(q_1,q_2)}})$. Combining \eqref{eq1 : Theorem-2-a} and \eqref{eq15 : Theorem-1}, we have \eqref{T:ACweakSol:E_1}.
This completes the proof.

In summary, if $V$ satisfies \eqref{V:speed} with $q=1$ for the speed estimate in addition to \eqref{T:weakSol:E_1} for the weak solution then the proof works.

 \noindent {\it For the case(ii)} : As in the proof of (i), if $V$ satisfies \eqref{T:weakSol:E_q} and \eqref{V-Lq-energy-speed} then we can complete the proof.
 We can check \eqref{T:ACweakSol:V_q} is the intersection of \eqref{T:weakSol:E_q} and \eqref{V:speed}.
\qed


\subsection{Existence for case: $\nabla\cdot V\geq 0$}
Here we prove theorems in Section~\ref{SS:divergence-free}: Theorem~\ref{T:weakSol_DivFree_1} and \ref{T:weakSol_DivFree_q}.

\subsubsection{Proof of Theorem \ref{T:weakSol_DivFree_1}}
Let $\rho_{0,n}$ and $V_n$ be as in the proof of Theorem \ref{T:weakSol} (i) except that $V_n$ further satisfies $\nabla \cdot V_n=0$.

 Compared to the proof of Theorem \ref{T:weakSol} (i). In this case, Proposition \ref{P:Lq-energy} (iv) says the energy estimates hold for all $(q_1, q_2)$.
Hence, once $V$ satisfies assumption \eqref{V_compact_m+1} with $q=1$ for the compactness argument, we complete the proof.%

Now, we claim that the condition  \eqref{T:weakSol_DivFree:V_1} is
equivalent to the assumption \eqref{V_compact_m+1} for $q=1$. Indeed, for the case $d>2$, \eqref{V_compact_m+1} for $q=1$ reads to
\begin{equation}\label{eq34 : Theorem-2-c}
\quad  \frac{md+1}{md+2}-\frac{1}{m}\leq \frac{1}{q_2}\leq \frac{md+1}{md+2} \quad {\mbox and}\quad \frac{d}{q_1} +\frac{2+q_{m,d}}{q_2}=1+q_{m,d},
\end{equation}
where $q_{m,d}=d(m-1)$. Let $m^*$ be as in \eqref{m_ast}. We note
\begin{equation}\label{eq35 : Theorem-2-c}
 0\leq  \frac{md+1}{md+2}-\frac{1}{m} \Longleftrightarrow m \geq m^*.
\end{equation}
Due to the fact $q_1, ~ q_2 >0$, for $1<m \leq m^*$, we have from \eqref{eq34 : Theorem-2-c}
\begin{equation}\label{eq36 : Theorem-2-c}
 0\leq \frac{1}{q_2}\leq \frac{md+1}{md+2},
\end{equation}
and, for $m \geq  m^*$, we have
\begin{equation}\label{eq37 : Theorem-2-c}
\quad  \frac{md+1}{md+2}-\frac{1}{m}\leq \frac{1}{q_2}\leq \frac{md+1}{md+2} .
\end{equation}
Similarly, for the case $d=2$, \eqref{V_compact_m+1} for $q=1$ reads to
\begin{equation}\label{eq38 : Theorem-2-c}
\quad  \frac{2m+1}{2m+2}-\frac{1}{m} < \frac{1}{q_2}\leq \frac{2m+1}{2m+2} \quad {\mbox and}\quad \frac{2}{q_1} +\frac{2+q_{m,2}}{q_2}=1+q_{m,2}.
\end{equation}
Hence, for $m=m^*$, we have
\begin{equation}\label{eq39 : Theorem-2-c}
0 < \frac{1}{q_2}\leq \frac{2m+1}{2m+2} ,
\end{equation}
for $m >m^*$
\begin{equation}\label{eq40 : Theorem-2-c}
 \frac{2m+1}{2m+2}-\frac{1}{m} < \frac{1}{q_2}\leq \frac{2m+1}{2m+2}
\end{equation}
and for $m <m^*$
\begin{equation}\label{eq41 : Theorem-2-c}
0 \leq \frac{1}{q_2}\leq \frac{2m+1}{2m+2}.
\end{equation}
Combining \eqref{eq36 : Theorem-2-c}, \eqref{eq37 : Theorem-2-c}, \eqref{eq39 : Theorem-2-c}, \eqref{eq40 : Theorem-2-c} and \eqref{eq41 : Theorem-2-c},
we have \eqref{T:weakSol_DivFree:V_1}. In summary, the range of $q_2$ in \eqref{T:weakSol_DivFree:V_1} is equivalent to
$\frac{d}{q_1}+\frac{2+d(m-1)}{q_2}=1+d(m-1)$ with $0 \leq \frac{1}{q_1} \leq l$ .
\qed

\subsubsection{Proof of Theorem \ref{T:weakSol_DivFree_q}}
Let $\rho_{0,n}$, $V_n$  and $\rho_n$ be as in the proof of Theorem \ref{T:weakSol} (ii) where $V_n$ further satisfies $\nabla \cdot V_n=0$.
Due to Proposition \ref{P:Lq-energy} (iv), for all  $(q_1,q_2)$ and $m>1$, we have
\begin{equation}\label{eq21 : Theorem-2-d}
\begin{aligned}
\iint_{\Omega_T}  \left [\partial_t\varphi + (V_n \cdot \nabla \varphi )\right ]\rho_n
 - \nabla(\rho_n)^m \cdot \nabla \varphi  \,dx\,dt
 =  - \int_{\Omega}\varphi(0) \rho_{0,n}\,dx,
\end{aligned}
\end{equation}
for each $\varphi\in C_c^\infty( \overline{\Omega} \times [0,T))$, and 
\begin{equation}\label{eq17 : Theorem-2-d}
\sup_{0\leq t \leq T} \int_{\Omega\times \{t\}}  (\rho_n)^q   \,dx
 +  \iint_{\Omega_T} \left |\nabla \rho_n^{\frac{q+m-1}{2}}\right |^2  \,dx\,dt \leq C,
\end{equation}
 where the constant $C=C( \|V\|_{\mathcal{S}_{m, q}^{(q_1,q_2)}}, \|\rho_{0}\|_{L^q(\Omega)} )$.

Now, we note $\int \rho_{0,n}\log \rho_{0,n}\, dx <\infty$ for all $n$.  Hence, we also have
\begin{equation}\label{eq18 : Theorem-2-d}
\sup_{0 \leq t \leq T}\int_{\Omega \times \{ t\}}\rho_n |\log \rho_n|  \,dx
+ \iint_{\Omega_T} \left|\nabla \rho_{n}^{\frac{m}{2}} \right|^2 \,dx\,dt <C,
\end{equation}
 where the constant $C=C( \|V\|_{\mathcal{S}_{m, q}^{(q_1,q_2)}}, \|\rho_{0}\|_{L^q(\Omega)} )$.

Once we have $V \in \mathcal{S}_{m,1}^{(\tilde{q}_1,\tilde{q}_2)}$ such that $(\tilde{q}_1,\tilde{q}_2)$ satisfies \eqref{V_compact_m+1} for $q=1$,
then we follow exactly the same step as in the proof of Theorem \ref{T:weakSol_DivFree_1}.
Then, we conclude $\rho_n$ (up to a subsequence) converges to $\rho$ in $L_{x,t}^{mp, \frac{m(md+2)}{md+1}}$ and $\rho$ is a $L^1$
weak solution of \eqref{E:Main}-\eqref{E:Main-bc-ic}.
Furthermore, from the convergence of $\rho_n$ to $\rho$ in $L_{x,t}^{mp, \frac{m(md+2)}{md+1}}$ and the estimate \eqref{eq17 : Theorem-2-d} which is independent of $n$,
we have
\begin{equation}\label{eq19 : Theorem-2-d}
\sup_{0\leq t \leq T} \int_{\Omega\times \{t\}}  (\rho_n)^q   \,dx
 +  \iint_{\Omega_T} \left |\nabla \rho_n^{\frac{q+m-1}{2}}\right |^2  \,dx\,dt \leq C.
\end{equation}

To complete the proof, we recall that \eqref{V_compact_m+1} for $q=1$ is equivalent to \eqref{T:weakSol_DivFree:V_1}, and we sketch that the assumptions
\eqref{T:weakSol_DivFree:V_q1} and \eqref{T:weakSol_DivFree:V_q2} imply
$V \in \mathcal{S}_{m,1}^{(\tilde{q}_1,\tilde{q}_2)}$ for some $(\tilde{q}_1,\tilde{q}_2)$ satisfying \eqref{T:weakSol_DivFree:V_1}.

{\it (i) For the case $1<m\leq m^*$ :} We first note that $V \in \mathcal{S}_{m,q}^{(q_1,q_2)}$ that is $\frac{d}{q_1}+\frac{2+q_{m,d}}{q_2}=1+q_{m,d}$ with $q_1, q_2 \geq 0$  implies
\begin{equation}\label{eq20 : Theorem-2-d}
0\leq \frac{1}{q_1} \leq \frac{1+q_{m,d}}{d}, \quad 0\leq \frac{1}{q_2}\leq \frac{1+q_{m,d}}{2+q_{m,d}} \Longrightarrow \frac{2+q_{m,d}}{1+q_{m,d}} \leq q_2\leq \infty .
\end{equation}

Since $\frac{2+d(m-1)}{1+d(m-1)} \leq \frac{2+q_{m,d}}{1+q_{m,d}}$ for all $q \geq 1$, recalling $\mathcal{S}_{m,q}^{(q_1,q_2)} \subset \mathcal{S}_{m,1}^{(\tilde{q}_1,q_2)}$
for some $\tilde{q}_1 \leq q_1$, we conclude that if $V \in \mathcal{S}_{m,q}^{(q_1,q_2)}$ satisfying \eqref{T:weakSol_DivFree:V_q1} then $V \in \mathcal{S}_{m,1}^{(\tilde{q}_1,q_2)}$
satisfying \eqref{T:weakSol_DivFree:V_1} for the case of $1<m\leq m^*$.

{\it (ii) For the case $m > m^*$ :} Let $q^*$ be such that $\frac{1+q^*_{m,d}}{d}=\frac{md+1}{md+2}-\frac{d-2}{md}(:=l)$. If $q\geq q^*$ then we have
$ \frac{1+q_{m,d}}{d} \leq \frac{1+q^*_{m,d}}{d}$. Hence $V \in \mathcal{S}_{m,q}^{(q_1,q_2)} \Rightarrow V\in \mathcal{S}_{m,1}^{(q_1,\tilde{q_2})}$
for some $\tilde{q}_2 \leq q_2$ and $\frac{1}{q_1}\leq l$. That is $q_2^{q,l} \leq q_2 \leq  \infty$ works.

If $1<q< q^*$ then only $0\leq \frac{1}{q_1} \leq l$ works. That means
\begin{equation}
1+q_{m,d}-ld \leq \frac{2+q_{m,d}}{q_2}\leq 1+q_{m,d}
\end{equation}
implies
\begin{equation}
\frac{2+q_{m,d}}{1+q_{m,d}}\leq q_2 \leq \frac{2+q_{m,d}}{1+q_{m,d}-ld}(=q_2^{q,u}).
\end{equation}

\qed


\subsubsection{Proof of Theorem \ref{T:ACweakSol_DivFree}}
{\it For the case (i)} : Similar to the argument in the proof of Theorem \ref{T:ACweakSol} (i), if $V$ satisfies the assumptions in Theorem \ref{T:weakSol_DivFree_1} and
Proposition \ref{P:Energy-speed} (iv) for $q=1$ then we conclude the result. Note that the condition \eqref{T:ACweakSol_DivFree:V_1} is the intersection of
\eqref{T:weakSol_DivFree:V_1} and \eqref{V-Lq-divfree} for $q=1$.

{\it For the case (ii)} : It is enough to say that we can check \eqref{T:ACweakSol_DivFree:V_q} implies \eqref{T:weakSol_DivFree:V_q1} and \eqref{T:weakSol_DivFree:V_q2} (for $1<m\leq m^*$ and $m>m^*$ respectively),
 and \eqref{V-Lq-divfree}. This completes the proof.

\noindent {\it For the case (iii)} : It follows from Proposition \ref{P:Energy-embedding} (iii). 
\qed

\subsection{Existence for case: $\nabla V$ in sub-scaling classes }
We prove theorems in Section~\ref{SS:tilde Serrin}: Theorem~\ref{T:weakSol_tilde} and \ref{T:ACweakSol_tilde}.

\subsubsection{Proof of Theorem \ref{T:weakSol_tilde}}
{\it For the case (i) :} Let 
 $V_n\in C^\infty_c(\overline{\Omega}\times[0,T)) $ with $V(t)\cdot \mathbf{n}=0$ on $\partial\Omega$ for all $t\in[0,T)$ be
such that
\begin{equation}\label{eq7 : T:weakSol_tilde}
\lim_{n \rightarrow \infty}\|V_n - V \|_{\tilde{\mathcal{S}}_{m,1}^{(\tilde{q}_1,\tilde{q}_2)} } =0.
\end{equation}
and  $\rho_{0,n}$ and $\rho_n$ be same as in Theorem \ref{T:weakSol} (i).
If $V$ satisfies \eqref{V-tilde-Lq} for $q=1$ then we exploit Proposition \ref{P:Lq-energy} (iii) for $q=1$ 
to have
\begin{equation}\label{eq17 : T:weakSol_tilde}
\sup_{0\leq t \leq T} \int_{\Omega\times \{t\}}  \rho_n | \log \rho_n|(\cdot,t)   \,dx
 +  \iint_{\Omega_T} \left |\nabla \rho_n^{\frac{m}{2}}\right |^2  \,dx\,dt \leq C,
\end{equation}
 where the constant $C=C(|\Omega|,  \|V\|_{\tilde{\mathcal{S}}_{m, 1}^{(\tilde{q}_1,\tilde{q}_2)}}, \int_{\Omega}\rho_0 \log \rho_0 \,dx )$.

 Next, we note that if $V$ satisfies \eqref{tilde-q2} then we have $V\in \mathcal{S}_{m,1}^{(q_1,q_2)}$ with
$\big(q_1, q_2\big):=\big(\frac{d\tilde{q}_1}{d-\tilde{q}_1},\tilde{q}_2\big)$. Hence, if $V$ satisfies both \eqref{tilde-q2} and \eqref{V-tilde-Lq} for $q=1$,
 that is
\begin{equation}\label{eq1 : T:weakSol_tilde}
V\in \tilde{\mathcal{S}}_{m,1}^{(\tilde{q}_1, \tilde{q}_2)}  \ \text{ for } \
\begin{cases}
 \frac{2+d(m-1)}{1+d(m-1)} < \tilde{q}_2 \leq \frac{m}{m-1},    &\text{ if } d \geq 2, \vspace{2 mm}\\
  \frac{2+d(m-1)}{1+d(m-1)} < \tilde{q}_2 < \frac{m}{m-1},   & \text{ if } d= 2.
 \end{cases}
\end{equation}
then we have that $V$ belongs to $\mathcal{S}_{m,1}^{(q_1,q_2)}$ where $q_2$ satisfies
the same range of $\tilde{q}_2$ as in \eqref{eq1 : T:weakSol_tilde} which is
\begin{equation}\label{eq2 : T:weakSol_tilde}
V\in {\mathcal{S}}_{m,1}^{({q}_1, {q}_2)}  \ \text{ for } \
\begin{cases}
 \frac{2+d(m-1)}{1+d(m-1)} < {q}_2 \leq \frac{m}{m-1},    &\text{ if } d \geq 2, \vspace{2 mm}\\
  \frac{2+d(m-1)}{1+d(m-1)} < {q}_2 < \frac{m}{m-1},   & \text{ if } d= 2.
 \end{cases}
\end{equation}
Note that \eqref{eq2 : T:weakSol_tilde} implies \eqref{V_compact_m+1} for $q=1$. Hence, for the rest of the proof, we follow exactly same step as in Theorem  \ref{T:weakSol} (i).

In summary, suppose $V$ satisfies \eqref{V-tilde-Lq} for $q=1$ in Proposition \ref{P:Lq-energy},
\eqref{tilde-q2} and \eqref{V_compact_m+1} for $q=1$ in Lemma \ref{L:V_compact_m+1}. Equivalently, $V$ satisfies \eqref{T:weakSol_tilde:V_1}. Then we can complete the proof.
%

{\it For the case (ii) :}  Let $\rho_{0,n},~ V_n$ and $\rho_n$ be same as before.

Suppose $V$ satisfies \eqref{V-tilde-Lq} then 
 we have
\begin{equation}\label{eq37 : T:weakSol_tilde}
\sup_{0\leq t \leq T} \int_{\Omega\times \{t\}}  (\rho_n)^q   \,dx
 +  \iint_{\Omega_T} \left |\nabla \rho_n^{\frac{q+m-1}{2}}\right |^2  \,dx\,dt \leq C,
\end{equation}
 where the constant $C=C( \|V\|_{\tilde{\mathcal{S}}_{m, q}^{(q_1,q_2)}}, \|\rho_{0}\|_{L^q(\Omega)} )$.
 As before, if $V$ further satisfies \eqref{tilde-q2} 
that is \begin{equation}\label{eq40 : T:weakSol_tilde}
V\in \tilde{\mathcal{S}}_{m,q}^{(\tilde{q}_1, \tilde{q}_2)}  \ \text{ for } \
\begin{cases}
 \frac{2+q_{m,d}}{1+q_{m,d}} < \tilde{q}_2 \leq \frac{q+m-1}{m-1},    &\text{ if } d \geq 2, \vspace{2 mm}\\
  \frac{2+q_{m,d}}{1+q_{m,d}} < \tilde{q}_2 < \frac{q+m-1}{m-1},   & \text{ if } d= 2.
 \end{cases}
\end{equation}
then we have that $V$ belongs to $\mathcal{S}_{m,q}^{(q_1,q_2)}$ where $q_2$ satisfies
the same range of $\tilde{q}_2$ as in \eqref{eq1 : T:weakSol_tilde} which is
then we have
\begin{equation}\label{eq41 : T:weakSol_tilde}
V\in {\mathcal{S}}_{m,q}^{({q}_1, {q}_2)}  \ \text{ for } \
\begin{cases}
 \frac{2+q_{m,d}}{1+q_{m,d}} < {q}_2 \leq \frac{q+m-1}{m-1},    &\text{ if } d \geq 2, \vspace{2 mm}\\
  \frac{2+q_{m,d}}{1+q_{m,d}} < {q}_2 < \frac{q+m-1}{m-1},   & \text{ if } d= 2.
 \end{cases}
\end{equation}

We note 
that \eqref{eq37 : T:weakSol_tilde} gives us
\begin{equation}\label{eq38 : T:weakSol_tilde}
\sup_{0 \leq t \leq T}\int_{\Omega \times \{ t\}}\rho_n |\log
\rho_n|  \,dx + \iint_{\Omega_T} \left|\nabla \rho_{n}^{\frac{q+m-1}{2}}
\right|^2 \,dx\,dt <C,
\end{equation}
where the constant $C=C( \|V\|_{\tilde{\mathcal{S}}_{m, q}^{(q_1,q_2)}}, \|\rho_{0}\|_{L^q(\Omega)} )$.

For the convergence of $\rho_n$, we can check that if $V\in  {\mathcal{S}}_{m,q}^{({q}_1, {q}_2)}$ satisfies \eqref{eq41 : T:weakSol_tilde} then
$V\in  {\mathcal{S}}_{m,1}^{({q}_1, \tilde{{q}}_2)}$ where $(q_1, \tilde{q}_2)$ satisfies \eqref{V_compact_m+1} for $q=1$. Hence, we have a compactness for $q=1$, and the rest of the
proof follows exactly the same as we did in the proof of Theorem \ref{T:weakSol_DivFree_q}.
\qed

\subsubsection{Proof of Theorem \ref{T:ACweakSol_tilde}}
{\it For the case (i)} : As in the proof of Theorem \ref{T:ACweakSol_DivFree} (i), if $V$ satisfies the assumptions in Theorem \ref{T:weakSol_tilde} (i) and
Proposition \ref{P:Energy-speed} (iii) for $q=1$ then we conclude the result. Note that the condition \eqref{T:weakSol_tilde:V_1} satisfies \eqref{V-tilde-Lq-energy-speed} for $q=1$.

\noindent{\it For the case (ii)} : Now, it is enough to say that we can check \eqref{T:ACweakSol_tilde:V_q} implies \eqref{T:weakSol_tilde:V_q} and
\eqref{V-tilde-Lq-energy-speed} for $q>11$
This completes the proof.

\noindent {\it For the case (iii)} : It follows from Proposition \ref{P:Energy-embedding} (ii). 
\qed




\appendix  \section{Figure supplements}\label{Appendix:fig}

\begin{itemize}
\item Fig.~\ref{F:S:1m2:more} ((Theorem~\ref{T:ACweakSol} for case $1<m \leq 2$, $q\geq 1$)): Refer Remark~\ref{R:T:ACweakSol}, and Fig.~\ref{F:S:1m2}.
As $d$ increases, the point $\textbf{b}$ approaches closer to the origin and $\textbf{A}$ may locate on the right-hand side of $\textbf{B}$ and $\textbf{b}$.

\begin{figure}
\centering

\begin{tikzpicture}[domain=0:16]


\draw (-0.5, -1) node[right] { \scriptsize (i). $\max\{2,\frac{2m}{(2m-1)(m-1)}\} < d \leq \frac{2m}{m-1}$.};

\fill[fill= lgray]
(0, 2)  -- (0.8,0) -- (2.35, 1.1)--(1.45, 2);

\fill[fill= gray]
(0.7, 2)--(0.7, 0.78) -- (1.97,0.78) -- (2.35, 1.1) -- (1.45, 2)--(0.7, 2);

\draw[->] (0,0) node[left] {\scriptsize $0$}
-- (5,0) node[right] {\scriptsize $\frac{1}{q_1}$};
\draw[->] (0,0) -- (0,4.5) node[left] { \scriptsize $\frac{1}{q_2}$};

\draw (0,4) node{\scriptsize $+$} node[left]{\scriptsize $1$} ;
\draw (4,0) node{\scriptsize $+$} ;

\draw[very thin]
(0, 2) -- (1.45, 2) ;
\draw[thick](1.45, 2) node{\scriptsize $\bullet$} node[right] {\scriptsize \textbf{E}};

\draw (0, 2) -- (1.45, 2);

\draw (0.8, 0) -- (2.35, 1.1);

\draw[thick] (2.35, 1.1) node{\scriptsize $\bullet$} node[right] {\scriptsize \textbf{D}};
\draw[very thin]
 (0.48, 0.78)
-- (1.92, 0.78) node{\scriptsize $\bullet$} node[right]{\scriptsize \textbf{C}} ;

\draw[thick](0.48, 0.78) circle(0.05) node[left] {\scriptsize \textbf{B}};

\draw (0.7, 0.78) node{\scriptsize $\bullet$} node[above]{\scriptsize \textbf{G}}  ;

\draw[very thin]
(0.7, 0.3)
 --(0.7, 2) node{\scriptsize $\bullet$} node[above] {\scriptsize \textbf{F}};

\draw[thick](0.7, 0.25) circle(0.05) node[left] {\scriptsize \textbf{A}};

\draw[thick, dotted] (0,2) node{\scriptsize $\times$}  node[left] {\scriptsize $\textbf{a}$}
 -- (0.8, 0)  node[below] {\scriptsize $\textbf{b}$};
\draw[thick] (0.8, 0) circle(0.05);

\draw
(0,3.4) node {\scriptsize $\times$} node[left] {\scriptsize $\frac{1}{p_1}$}
-- (3.5, 0) node {\scriptsize $\times$} node[below] {\scriptsize $\frac{1+d(m-1)}{d}$};
\draw (3.9, 0.3) node{\scriptsize $\mathcal{S}_{m,1}^{(q_1, q_2)}$};

\draw (0,2.7) node {\scriptsize $\times$} node[left] {\scriptsize $\frac{m+d(m-1)}{2m+d(m-1)}$}
    -- (2.7, 0) node {\scriptsize $\times$} node[below] {\scriptsize \textbf{c}};
\draw (2.7, 0.3) node{\scriptsize $\mathcal{S}_{m,m}^{(q_1, q_2)}$};



\draw[thin, dotted] (1.85, 0) node{\scriptsize $*$}  -- (1.85, 0.63);

\draw[thin, dotted] (2.35, 0) node{\scriptsize $*$}  -- (2.35, 1.1);
\draw[thin, dotted] (0, 1.1) node{\scriptsize $*$}  -- (2.35, 1.1);

\draw[thin, dotted] (1.45, 0) node{\scriptsize $*$}  -- (1.45, 2);


\draw(0.5, 4.5) node[right]{\scriptsize $\overline{\textbf{ab}}, \textbf{A}-\textbf{F}, \textbf{c}$: same as in Fig. 4-(i)};
\draw (0.5, 4) node[right] {\scriptsize $\mathcal{R} (\textbf{GCDEF})$: scaling invariant class of \eqref{T:ACweakSol:V_q} };
\draw (0.5, 3.5) node[right] {\scriptsize $\mathcal{R} (\textbf{GAbC})$: scaling invariant class of \eqref{T:weakSol:V_q} };
\draw (1, 3) node[right] {\scriptsize $\textbf{G} = (\frac{m-1}{2m}, \frac{m-1}{2m-1})$ };


\draw (7.5, -1) node[right] { \scriptsize (ii). $ d > \frac{2m}{m-1}$.};

\fill[fill= lgray]
(8, 2) -- (8.5, 0) -- (10.15, 1.3)--(9.45, 2);

\fill[fill= gray]
(8.75, 2)--(8.75, 1) -- (9.75,1) -- (10.15, 1.3) -- (9.45, 2)--(8.75, 2);

\draw[->] (8,0) node[left] {\scriptsize $0$}
-- (13,0) node[right] {\scriptsize $\frac{1}{q_1}$};
\draw[->] (8,0) -- (8,4.5) node[left] { \scriptsize $\frac{1}{q_2}$};

\draw (8,4) node{\scriptsize $+$} node[left]{\scriptsize $1$} ;
\draw (12,0) node{\scriptsize $+$} ;

\draw[very thin]
(8, 2) -- (9.45, 2) ;
\draw[thick](9.45, 2) node{\scriptsize $\bullet$} node[right] {\scriptsize \textbf{E}};

\draw (8, 2) -- (9.45, 2);

\draw (8.5, 0) -- (10.15, 1.3);

\draw[thick] (10.15, 1.3) node{\scriptsize $\bullet$} node[right] {\scriptsize \textbf{D}};
\draw[very thin]
 (8.25, 1)
-- (9.75, 1) node{\scriptsize $\bullet$} node[right]{\scriptsize \textbf{C}} ;

\draw[thick](8.25, 1) circle(0.05) node[below] {\scriptsize \textbf{B}};

\draw (8.75, 1) node{\scriptsize $\bullet$} node[above]{\scriptsize \textbf{G}}  ;
\draw (8.75, 0.2) node{\scriptsize $\bullet$} node[above]{\scriptsize \textbf{H}}  ;

\draw[very thin]
(8.75, 0)
 --(8.75, 2) node{\scriptsize $\bullet$} node[above] {\scriptsize \textbf{F}};

\draw[thick](8.75, 0) circle(0.05) node[below] {\scriptsize \textbf{A}};

\draw[thick, dotted] (8,2) node{\scriptsize $\times$} node[left] {\scriptsize $\textbf{a}$}
 -- (8.5, 0) node{\scriptsize $\times$}  node[below] {\scriptsize $\textbf{b}$};

\draw
(8,3.4) node {\scriptsize $\times$} node[left] {\scriptsize $\frac{1}{p_1}$}
-- (11.5, 0) node {\scriptsize $\times$} node[below] {\scriptsize $\frac{1+d(m-1)}{d}$};
\draw (11.9, 0.3) node{\scriptsize $\mathcal{S}_{m,1}^{(q_1, q_2)}$};

\draw (8,2.8) node {\scriptsize $\times$} node[left] {\scriptsize $\frac{m+d(m-1)}{2m+d(m-1)}$}
    -- (10.7, 0) node {\scriptsize $\times$} node[below] {\scriptsize \textbf{c}};
\draw (10.7, 0.3) node{\scriptsize $\mathcal{S}_{m,m}^{(q_1, q_2)}$};




\draw[thin, dotted] (10.35, 0) node{\scriptsize $*$}  -- (10.35, 1.1);

\draw[thin, dotted] (9.45, 0) node{\scriptsize $*$}  -- (9.45, 2);


\draw(8.5, 4.5) node[right]{\scriptsize $\overline{\textbf{ab}}, \textbf{A}-\textbf{F}, \textbf{c}$: same as in Fig. 4-(i)};
\draw (8.5, 4) node[right] {\scriptsize $\mathcal{R} (\textbf{GCDEF})$: scaling invariant class of \eqref{T:ACweakSol:V_q} };
\draw (8.5, 3.5) node[right] {\scriptsize $\mathcal{R} (\textbf{GHC})$: scaling invariant class of \eqref{T:weakSol:V_q}};

\draw (10, 3) node[right] {\scriptsize $\textbf{G} = (\frac{m-1}{2m}, \frac{m-1}{2m-1})$};
\draw (10, 2.5) node[right] {\scriptsize $\textbf{H} = (\frac{m-1}{2m}, \frac{(m-1)d-2m}{2m(d-2)})$ };

\end{tikzpicture}
\caption{\footnotesize Theorem~\ref{T:ACweakSol} for $1< m \leq 2$, $q>1$}
\label{F:S:1m2:more}
\end{figure}
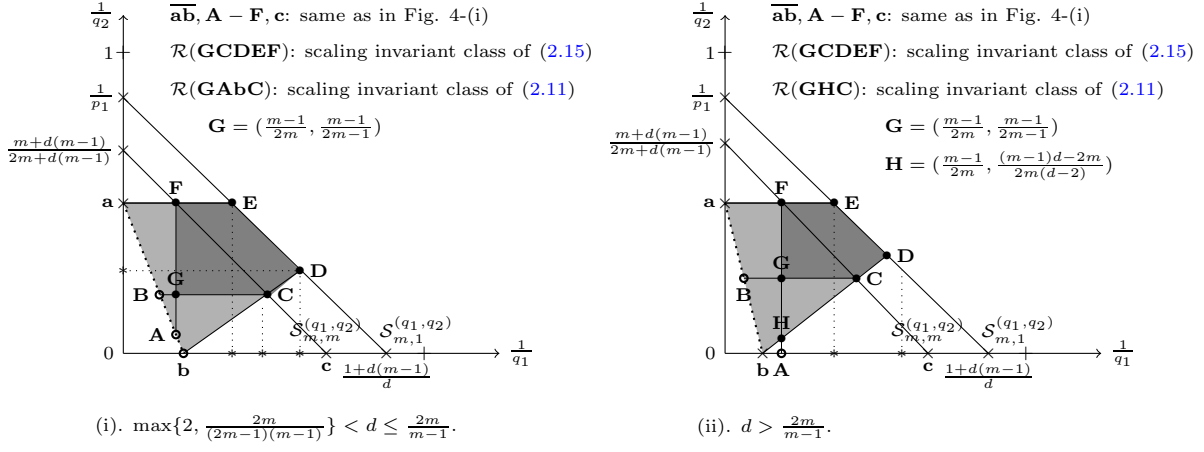


\item Fig.~\ref{F:S:m2:more} (Theorem~\ref{T:ACweakSol} for case $m>2$, $q\geq m-1$): Refer Remark~\ref{R:T:ACweakSol} and Fig.~\ref{F:S:m2}. The region $\mathcal{R}(\textbf{abD})$ is the range of $(q_1,q_2)$ satisfying \eqref{T:weakSol:V_q}, and the dark shaded region $\mathcal{R}(\textbf{GCDE})$ is for \eqref{T:ACweakSol:V_q}. When $d=2$, $\overline{\textbf{GCD}}$ is excluded from $\mathcal{R}(\textbf{GCDE})$.

\begin{figure}
\centering
\begin{tikzpicture}[domain=0:16]

\draw (-0.5, -1) node[right] { \scriptsize (i). $ 2 < d \leq \frac{2m}{m-1}$.};

\fill[fill= lgray]
(0, 2) -- (1,0) -- (2,2);


\fill[fill= gray]
(0.7, 2) -- (0.7,1.15) -- (1.58, 1.15) --(2,2);

\draw[->] (0,0) node[left] {\scriptsize $0$} -- (5,0) node[right] {\scriptsize $\frac{1}{q_1}$};
\draw[->] (0,0) -- (0,4.5) node[left] { \scriptsize $\frac{1}{q_2}$};

\draw (0,4) node{\scriptsize $+$} node[left] {\scriptsize $1$} ;
\draw (4,0) node{\scriptsize $+$};

\draw[thick](0,2)  circle(0.05) node[left] {\scriptsize \textbf{a}};

\draw (0, 2)-- (2, 2) ;

\draw[thick, dotted] (0,2)
 -- (1, 0)  node[below] {\scriptsize $\textbf{b}$};
\draw[thick] (1,0) circle(0.05);

\draw(2, 2) node{\scriptsize $\bullet$} node[above] {\scriptsize \textbf{D}};

\draw(0.7,2) node{\scriptsize $\bullet$} node[above]{\scriptsize \textbf{E}}
--(0.7, 0.6) node[left]{\scriptsize \textbf{A}} ;
\draw[thick] (0.7,0.6) circle(0.05);

\draw (0.43, 1.15) node[left]{\scriptsize \textbf{B}}
--(1.58, 1.15)node{\scriptsize $\bullet$}node[right]{\scriptsize \textbf{C}};
\draw[thick] (0.43, 1.15) circle(0.05);

\draw(0.7, 1.15) node{\scriptsize $\bullet$} node[above]{\scriptsize \textbf{G}};

\draw (1, 0)  -- (2, 2);


\draw (0,3.5) node {\scriptsize $\times$} node[left] {\scriptsize $\frac{1+d}{2+d}$}
-- (4.6, 0) node {\scriptsize $\times$} node[below] {\scriptsize $\frac{1+d}{d}$} ;
\draw (4.6, 0.5) node {\scriptsize $\mathcal{S}_{m,m-1}^{(q_1, q_2)}$} ;

\draw (0,2.7) node {\scriptsize $\times$} node[left] {\scriptsize $\frac{m+d(m-1)}{md}$}
 -- (2.8, 0) node {\scriptsize $\times$} node[below] {\scriptsize $\frac{m+d(m-1)}{2m+d(m-1)}$};
\draw (3.2, 0.3) node{\scriptsize $\mathcal{S}_{m,m}^{(q_1, q_2)}$};

\draw (0.5, 5) node[right] {\scriptsize $\overline{\textbf{ab}} = \mathcal{S}_{m, \infty}^{(q_1, q_2)}$, $\textbf{a} = (0, \frac 12)$, $\textbf{b} = (\frac 1d, 0)$ };
\draw (0.5, 4.5) node[right] {\scriptsize $\mathcal{R}(\textbf{GCDE})$: scaling invariant class of \eqref{T:ACweakSol:V_q} };
\draw (0.5, 4) node[right] {\scriptsize $\mathcal{R}(\textbf{GAbC}))$: scaling invariant class of \eqref{T:weakSol:V_q} };

\draw (0.5, 3.5) node[right] {\scriptsize $\textbf{A} = (\frac{m-1}{2m}, \frac{d+m(2-d)}{4m})$, $\textbf{B} = (\frac{1}{d(2m-1)}, \frac{m-1}{2m-1})$};
\draw (1, 3) node[right] {\scriptsize $\textbf{C} = (\frac{1+d(m-1)}{d(2m-1)}, \frac{m-1}{2m-1})$, $\textbf{D} = (\frac 12, \frac 12)$ };
\draw (2, 2.5) node[right] {\scriptsize  $\textbf{E} = (\frac{m-1}{2m}, \frac 12)$, $\textbf{G}= (\frac{m-1}{2m}, \frac{m-1}{2m-1}) $};


\draw (7.5, -1) node[right] { \scriptsize (ii). $d >\frac{2m}{m-1}$.};

\fill[fill= lgray]
(8, 2)-- (8.6, 0) -- (10,2);

\fill[fill= gray]
(8.85, 2) -- (8.85,1.27) -- (9.5, 1.27) --(10,2);

\draw[->] (8,0) node[left] {\scriptsize $0$} -- (13,0) node[right] {\scriptsize $\frac{1}{q_1}$};
\draw[->] (8,0) -- (8,4.5) node[left] { \scriptsize $\frac{1}{q_2}$};

\draw (8,4) node{\scriptsize $+$} node[left] {\scriptsize $1$} ;
\draw (12,0) node{\scriptsize $+$};

\draw (8, 2) node[left] {\scriptsize \textbf{a}} -- (10, 2) ;

\draw[thick, dotted] (8,2) node{\scriptsize $\times$} node[left] {\scriptsize $\textbf{a}$}
 -- (8.6, 0)node{\scriptsize $\times$} node[below] {\scriptsize $\textbf{b}$};

\draw(10, 2) node{\scriptsize $\bullet$} node[above] {\scriptsize \textbf{D}};

\draw(8.85,2) node{\scriptsize $\bullet$} node[above]{\scriptsize \textbf{E}}
--(8.85, 0) node[below]{\scriptsize \textbf{A}} ;
\draw[thick] (8.85, 0) circle(0.05);

\draw (8.23, 1.27) node[below]{\scriptsize \textbf{B}}
--(9.5, 1.27)node{\scriptsize $\bullet$}node[right]{\scriptsize \textbf{C}};
\draw[thick] (8.23, 1.27) circle(0.05);

\draw(8.85, 1.27) node{\scriptsize $\bullet$} node[above]{\scriptsize \textbf{G}};

\draw(8.85, 0.38) node{\scriptsize $\bullet$}node[right]{\scriptsize \textbf{H}};

\draw (8.6, 0)  -- (10, 2);


\draw (8,3.5) node {\scriptsize $\times$} node[left] {\scriptsize $\frac{1+d}{2+d}$}
-- (12.6, 0) node {\scriptsize $\times$} node[below] {\scriptsize $\frac{1+d}{d}$} ;
\draw (12.5, 1) node {\scriptsize $\mathcal{S}_{m,m-1}^{(q_1, q_2)}$} ;

\draw (8,2.9) node {\scriptsize $\times$} node[left] {\scriptsize $\frac{m+d(m-1)}{md}$}
 -- (10.7, 0) node {\scriptsize $\times$} node[below] {\scriptsize $\frac{m+d(m-1)}{2m+d(m-1)}$};
\draw (11.2, 0.5) node{\scriptsize $\mathcal{S}_{m,m}^{(q_1, q_2)}$};

\draw (8.5, 4.5) node[right] {\scriptsize $\overline{\textbf{ab}}, \textbf{B}-\textbf{E}$: same as in Fig. 2-(i)};
\draw (8.5, 4) node[right] {\scriptsize $\mathcal{R}(\textbf{GCDE})$: scaling invariant class of \eqref{T:ACweakSol:V_q} };
\draw (8.5, 3.5) node[right] {\scriptsize $\mathcal{R}(\textbf{GHC})$: scaling invariant class of \eqref{T:weakSol:V_q} };
\draw (9, 3) node[right] {\scriptsize $\textbf{A}= (\frac{m-1}{2m},0)$  };
\draw (9.5, 2.5) node[right] {\scriptsize $\textbf{G}, \textbf{H}$: same as in Fig. 1-(iii) };

\end{tikzpicture}
\caption{\footnotesize Theorem~\ref{T:ACweakSol} for $m>2$ and $q\geq m-1$}
\label{F:S:m2:more}
\end{figure}
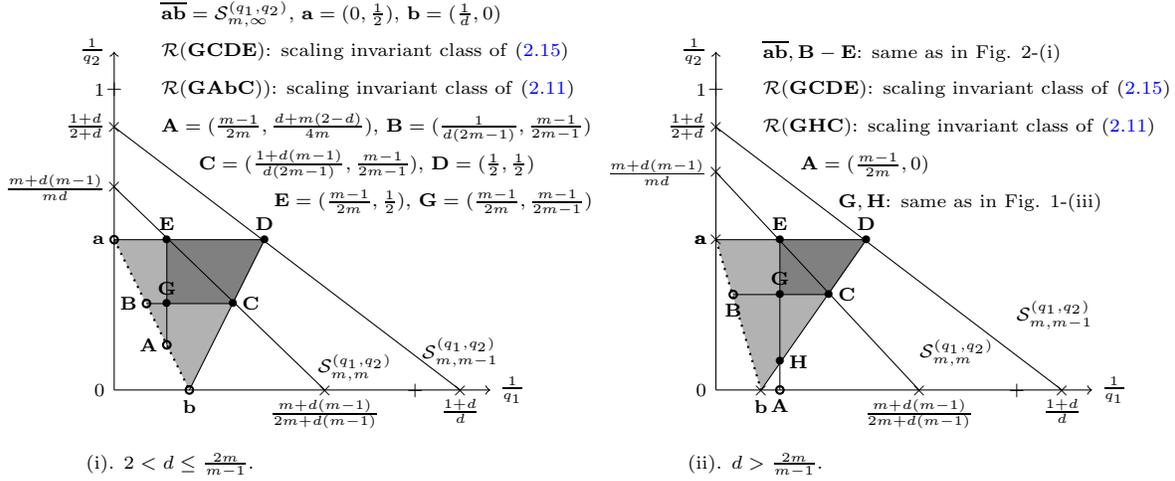

\end{itemize}

\section{Compactness arguments on $\mathbb{R}^d$}\label{Appendix:compact}

The existence of weak solutions of  \eqref{E:Main} on $\mathbb{R}^{d}\times (0, T)$ is studied in \cite{HKK}. Regarding the compactness arguments in \cite[Proposition~4.19]{HKK} (which is same as Proposition~\ref{P:AL2}), we obtain Proposition~\ref{P:AL1} newly that can improve results in \cite[Theorem~2.11 \& Theorem~2.15]{HKK} in case $1 < q \leq m+1$. 
The compactness argument in Proposition~\ref{P:AL1} which accompanied with the range of $(q_1, q_2)$ for $V \in \mathcal{S}_{m,q}^{(q_1, q_2)}$ satisfying \eqref{V_compact_m+1} improve the existence results either in case $\nabla \cdot V \geq 0$ (including the divergence-free case) or $V \in \tilde{\mathcal{S}}_{m,q}^{\tilde{q}_1, \tilde{q}_2}$ for all $q>1$. In particular, if $\nabla \cdot V \geq 0$ and  $1 < q \leq m+1$, we have the following existence theorem (cf. \cite[Theorem~2.11]{HKK}).

\begin{theorem}\label{T:unbounded_divfree}
Let $m>1$ and $1<q \leq m+1$. Suppose that $\rho_0 \in \mathcal{P}(\mathbb{R}^d) \cap L^{q}(\mathbb{R}^d)$. Furthermore, assume that $\nabla \cdot V \geq 0$ and 
\begin{equation}\label{V_unbounded_divfree}
V\in \mathfrak{S}_{m,q}^{(q_1,q_2)} \quad \text{for} \quad
\begin{array}{l l}
\begin{cases}
 \frac{2+q_{m,d}}{1+q_{m,d}} \leq q_2 \leq \left[\frac{1}{\gamma_1} - \frac{1}{q+m-1}\right]^{-1},
 & \text{if }  1<q\leq \overline{q} , \vspace{1 mm} \\
\gamma_1 \leq q_2 \leq \left[\frac{1}{\gamma_1} - \frac{1}{q+m-1}\right]^{-1} ,
& \text{if }  \overline{q} < q \leq m+1, \vspace{1 mm} \\
\end{cases}
\text{ and } d>2,\vspace{2 mm} \\
\begin{cases}
 \frac{2+q_{m,d}}{1+q_{m,d}} \leq q_2 < \left[\frac{1}{\gamma_1} - \frac{1}{q+m-1}\right]^{-1},
 & \text{if }   1<q \leq \overline{q}, \vspace{1 mm} \\
  \gamma_1 \leq q_2 < \left[\frac{1}{\gamma_1} - \frac{1}{q+m-1}\right]^{-1},
 & \text{if }   \overline{q} < q \leq m+1, \vspace{1 mm} 
\end{cases}
 \text{ and } d=2,
\end{array}
\end{equation}
where 
\begin{equation}\label{overline_q}
\gamma_1 = \frac{d(q+m-1) + 2q}{md+q}, \ \text{ and } \ \overline{q} = \frac{1}{2}\left( \sqrt{d^2(m-1)^2 + 4} -d(m-1)+2\right).
\end{equation}
Then, there exists a nonnegative $L^q$-weak solution of  \eqref{E:Main} on $\mathbb{R}^{d}\times (0, T)$ in \cite[Definition~2.1]{HKK} that holds \eqref{T:weakSol:E_q} where $C=C(\|\rho_0\|_{L^{q}(\mathbb{R}^d)})$. 
\end{theorem}

\begin{remark}
\begin{itemize}
\item[(i)] Furthermore, by applying Proposition~\ref{P:AL1}, the existence results are extended for all $q>1$ in \cite[Theorem~2.13, Theorem~2.15, \& Theorem~2.16]{HKK}. 


\item[(ii)] If $1< m < m^\ast$ with $m^\ast$ in \eqref{m_ast} and $1< q \leq \tilde{q}$ with
\begin{equation}\label{tilde_q}
\tilde{q} = \frac{1}{2} \left( \sqrt{d^2(m-1)^2 - 2d(m^2-1)+(3-m)^2} - d(m-1) + (3-m) \right), 
\end{equation}
then the range of $q_2$ allows $\infty$. 

\item[(iii)] Fig.~\ref{F:compact1} illustrates \eqref{V_unbounded_divfree} when $ 1 < m \leq m^\ast$  for  $m^\ast$ in \eqref{m_ast} and $1 < q \leq m+1$. Fig.~\ref{F:compact2} illustrates \eqref{V_unbounded_divfree} when $ m > m^\ast$  for  $m^\ast$ in \eqref{m_ast} and $1 < q \leq m+1$. 

\end{itemize}
\end{remark}

\begin{figure}
\centering
\begin{tikzpicture}[domain=0:16]


\fill[fill= lgray]
(0.57, 2)--(2, 0.5) -- (4.5, 0) -- (5.25, 0) -- (0, 3.6) -- (0,3) -- (0.57,2);

\draw[->] (0,0) node[left] {\scriptsize $0$}
-- (5.5,0) node[right] {\scriptsize $\frac{1}{q_1}$};
\draw[->] (0,0) -- (0,4.5) node[left] { \scriptsize $\frac{1}{q_2}$};

\draw (0,4) node{\scriptsize $+$} node[left]{\scriptsize $1$} ;



\draw[dotted] (0,2) node {\scriptsize $\times$} node[left] {\scriptsize $\frac{1}{2}$}
-- (1.1, 0) node {\scriptsize $\times$} node[below] {\scriptsize $\frac{1}{d}$};

\draw
(-1,4.3) -- (6, -0.5) ;
\draw (5.25, 0) node{\scriptsize $\bullet$} node[above] {\scriptsize $\textbf{E}$};
\draw (0, 3.62) node{\scriptsize $\bullet$} node[left] {\scriptsize $\textbf{F}$} ;
\draw (4.7, 0.8) node{\scriptsize $\mathcal{S}_{m,1}^{(q_1, q_2)}$};


\draw (0,2.6) node {\scriptsize $\times$}  node[left]{\scriptsize $\frac{(m+1)+d(m-1)}{2(m+1)+d(m-1)}$}
    -- (2.5, 0) node {\scriptsize $\times$}  node[below]{\scriptsize $\frac{m+1 + d(m-1)}{d(m+1)}$};
\draw[thick] (2, 0.5)  node {\scriptsize $\bullet$} node[below] {\scriptsize \textbf{C}};
\draw (1.15, 0.8) node{\scriptsize $\mathcal{S}_{m,m+1}^{(q_1, q_2)}$};

\draw[dotted] (0,2) -- (0.57,2);
\draw (0.57,2) node{\scriptsize $\bullet$} node[below] {\scriptsize \textbf{H}};

\draw[dotted] (0, 0.5) node{\scriptsize $\times$} node[left]{\scriptsize $\frac{m-1}{2m}$} -- (2, 0.5) ;

\draw (0.57, 2) -- (-0.7, 4.2);
\draw (2, 0.5) -- (6, -0.3);

\draw[thick] (0,3) node {\scriptsize $\bullet$} node[left]{\scriptsize \textbf{G}} -- (3.6, 0) ;
\draw (2.2, 0.8) node{\scriptsize $\mathcal{S}_{m,\overline{q}}^{(q_1, q_2)}$};

\draw[thick] (0,3.3) -- (4.5, 0) node{\scriptsize $\bullet$} node[below] {\scriptsize $\textbf{D}$};
\draw (3.25, 0.8) node{\scriptsize $\mathcal{S}_{m,\tilde{q}}^{(q_1, q_2)}$};




\draw (7, 4) node[right] {\scriptsize $\mathcal{R} (\textbf{CDEFGH})$: scaling invariant class of \eqref{T:weakSol_DivFree:V_q1}};

\draw (7, 3.5) node[right] {\scriptsize \textbf{C}, \textbf{H} : on $\mathcal{S}_{m, m+1}^{(q_1, q_2)}$ where };
\draw (8, 3) node[right] {\scriptsize $\textbf{C} = (\frac{2+d(m-1)}{2md}, \frac{m-1}{2m})$, $\textbf{H} = (\frac{m-1}{2(m+1)} , \frac{1}{2})$. };
\draw (7, 2.5) node[right] {\scriptsize \textbf{E}, \textbf{F} : on $\mathcal{S}_{m, 1}^{(q_1, q_2)}$ where };
\draw (8, 2) node[right] {\scriptsize $\textbf{E} = (\frac{1+d(m-1)}{d}, 0)$, $\textbf{F} = (0 , \frac{1+d(m-1)}{2+d(m-1)})$. };
\draw (7, 1.5) node[right] {\scriptsize $\textbf{G} = (0, \frac{\overline{q} + d(m-1)}{2 \overline{q} + d(m-1)})$ on $\mathcal{S}_{m,\overline{q}}^{(q_1, q_2)}$ for $\overline{q}$ in \eqref{overline_q} }  ;
\draw (7, 1) node[right] {\scriptsize $\textbf{D} = (\frac{ \tilde{q} + d(m-1)}{d \tilde{q}}, 0)$ on $\mathcal{S}_{m, \tilde{q}}^{(q_1, q_2)}$  for $\tilde{q}$ in \eqref{tilde_q}. } ;
\end{tikzpicture}
\caption{\footnotesize Theorem~\ref{T:unbounded_divfree} for $ 1< m \leq m^\ast$, $ 1 < q \leq m+1$, $d>2$ .}
\label{F:compact1}
\end{figure}

\begin{figure}
\centering
\begin{tikzpicture}[domain=0:16]


\fill[fill= lgray]
(0.57, 2)--(1.63, 0.9) -- (4.7, 0.4) -- (0, 3.6) -- (0,3) -- (0.57,2);

\draw[->] (0,0) node[left] {\scriptsize $0$}
-- (5.5,0) node[right] {\scriptsize $\frac{1}{q_1}$};
\draw[->] (0,0) -- (0,4.5) node[left] { \scriptsize $\frac{1}{q_2}$};

\draw (0,4) node{\scriptsize $+$} node[left]{\scriptsize $1$} ;



\draw[dotted] (0,2) node {\scriptsize $\times$} node[left] {\scriptsize $\frac{1}{2}$}
-- (1.1, 0) node {\scriptsize $\times$} node[below] {\scriptsize $\frac{1}{d}$};

\draw
(-1,4.3) -- (6, -0.5) ;
\draw (4.7, 0.4) node{\scriptsize $\bullet$} node[above] {\scriptsize $\textbf{E}$};
\draw (0, 3.62) node{\scriptsize $\bullet$} node[left] {\scriptsize $\textbf{F}$} ;
\draw (4.2, 1.2) node{\scriptsize $\mathcal{S}_{m,1}^{(q_1, q_2)}$};


\draw (0,2.6) node {\scriptsize $\times$}  node[left]{\scriptsize $\frac{(m+1)+d(m-1)}{2(m+1)+d(m-1)}$}
    -- (2.5, 0) node {\scriptsize $\times$}  node[below]{\scriptsize $\frac{m+1 + d(m-1)}{d(m+1)}$};
\draw[thick] (1.63, 0.9)  node {\scriptsize $\bullet$} node[below] {\scriptsize \textbf{C}};
\draw (0.8, 1.2) node{\scriptsize $\mathcal{S}_{m,m+1}^{(q_1, q_2)}$};

\draw[dotted] (0,2) -- (0.57,2);
\draw (0.57,2) node{\scriptsize $\bullet$} node[below] {\scriptsize \textbf{H}};

\draw[dotted] (0, 0.9) node{\scriptsize $\times$} node[left]{\scriptsize $\frac{m-1}{2m}$} -- (1.63, 0.9) ;

\draw (0.57, 2) -- (-0.7, 4.2);
\draw (1.63, 0.9) -- (4.7, 0.4);

\draw[thick] (0,3) node {\scriptsize $\bullet$} node[left]{\scriptsize \textbf{G}} -- (3.6, 0) ;
\draw (1.9, 1.2) node{\scriptsize $\mathcal{S}_{m,\overline{q}}^{(q_1, q_2)}$};

\draw[thick] (0,3.3) -- (4.5, 0) ;
\draw (3.75, 0.55) node{\scriptsize $\bullet$} node[below] {\scriptsize $\textbf{D}$};
\draw (2.8, 1.2) node{\scriptsize $\mathcal{S}_{m,\tilde{q}}^{(q_1, q_2)}$};




\draw (7, 4) node[right] {\scriptsize $\mathcal{R} (\textbf{CDEFGH})$: scaling invariant class of \eqref{T:weakSol_DivFree:V_q1}};

\draw (7, 3.5) node[right] {\scriptsize \textbf{C}, \textbf{H} : on $\mathcal{S}_{m, m+1}^{(q_1, q_2)}$ where };
\draw (8, 3) node[right] {\scriptsize $\textbf{C} = (\frac{2+d(m-1)}{2md}, \frac{m-1}{2m})$, $\textbf{H} = (\frac{m-1}{2(m+1)} , \frac{1}{2})$. };
\draw (7, 2.5) node[right] {\scriptsize \textbf{E}, \textbf{F} : on $\mathcal{S}_{m, 1}^{(q_1, q_2)}$ where };
\draw (8, 2) node[right] {\scriptsize $\textbf{E} = (\frac{1+d(m-1)}{d}, 0)$, $\textbf{F} = (0 , \frac{1+d(m-1)}{2+d(m-1)})$. };
\draw (7, 1.5) node[right] {\scriptsize $\textbf{G} = (0, \frac{\overline{q} + d(m-1)}{2 \overline{q} + d(m-1)})$ on $\mathcal{S}_{m,\overline{q}}^{(q_1, q_2)}$ for $\overline{q}$ in \eqref{overline_q} }  ;
\draw (7, 1) node[right] {\scriptsize $\textbf{D} = (\frac{1}{\gamma_1} - \frac{d-2}{d(q+m-1)}, \frac{1}{\gamma_1}-\frac{1}{q+m-1})$  } ;
\draw (8, 0.5) node[right] {\scriptsize on $\mathcal{S}_{m, \tilde{q}}^{(q_1, q_2)}$  for $\tilde{q}$ in \eqref{tilde_q} and $\gamma_1$ in \eqref{gamma_1} . } ;

\end{tikzpicture}
\caption{\footnotesize Theorem~\ref{T:unbounded_divfree} for $ m > m^\ast$, $ 1 < q \leq m+1$, $d>2$ .}
\label{F:compact2}
\end{figure}





\bibliographystyle{amsplain}

\begin{bibdiv}
\begin{biblist}

\bib{ags:book}{book}{
   author={Ambrosio, L.},
   author={Gigli, N.},
   author={Savar\'{e}, G.},
   title={Gradient flows in metric spaces and in the space of probability
   measures},
   series={Lectures in Mathematics ETH Z\"{u}rich},
   edition={2},
   publisher={Birkh\"{a}user Verlag, Basel},
   date={2008},
}

\bib{CHKK17}{article}{
   author={Chung, Y.},
   author={Hwang, S.},
   author={Kang, K.},
   author={Kim, J.},
   title={H\"{o}lder continuity of Keller-Segel equations of porous medium type
   coupled to fluid equations},
   journal={J. Differential Equations},
   volume={263},
   date={2017},
   number={4},
   pages={2157--2212},
}

\bib{CKK14}{article}{
   author={Chung, Y.},
   author={Kang, K.},
   author={Kim, J.},
 TITLE = {Global existence of weak solutions for a
 {K}eller-{S}egel-fluid model with nonlinear diffusion},
   JOURNAL = {J. Korean Math. Soc.},
  FJOURNAL = {Journal of the Korean Mathematical Society},
    VOLUME = {51},
 YEAR = {2014},
    NUMBER = {3},
 PAGES = {635--654},
}


\bib{DB93}{book}{
   author={DiBenedetto, E.},
   title={Degenerate parabolic equations},
   series={Universitext},
   publisher={Springer-Verlag, New York},
   date={1993},
   pages={xvi+387},
}



\bib{DGV12}{book}{
   author={DiBenedetto, E.},
   author={Gianazza, U.},
   author={Vespri, V.},
   title={Harnack's inequality for degenerate and singular parabolic
   equations},
   series={Springer Monographs in Mathematics},
   publisher={Springer, New York},
   date={2012},
  }
  
 \bib{HKK}{article}{
    AUTHOR = {Hwang, S.},
    AUTHOR = {Kang, K.},
    Author = {Kim, H. K.},
     TITLE = {Existence of weak solutions for porous medium equation with a
              divergence type of drift term},
   JOURNAL = {Calc. Var. Partial Differential Equations},
  FJOURNAL = {Calculus of Variations and Partial Differential Equations},
    VOLUME = {62},
      YEAR = {2023},
    NUMBER = {4},
}

\bib{HZ21}{article}{
   author={Hwang, S.},
   author={Zhang, Y. P.},
   title={Continuity results for degenerate diffusion equations with $L_t^p
   L_x^q$ drifts},
   journal={Nonlinear Anal.},
   volume={211},
   date={2021},
   pages={112413},
}

\bib{KK-SIMA}{article}{
   author={Kang, K.},
   author={Kim, H. K.},
   title={Existence of weak solutions in Wasserstein space for a chemotaxis
   model coupled to fluid equations},
   journal={SIAM J. Math. Anal.},
   volume={49},
   date={2017},
   number={4},
   pages={2965--3004},
}

\bib{KZ18}{article}{
   author={Kim, I.},
   author={Zhang, Y. P.},
   title={Regularity properties of degenerate diffusion equations with
   drifts},
   journal={SIAM J. Math. Anal.},
   volume={50},
   date={2018},
   number={4},
   pages={4371--4406},
}

\bib{Santambrosio15}{book}{
   author={Santambosio, F.},
   title={Optimal transport for applied mathematicians. Calculus of variations, PDEs, and modeling },
   series={Progress in Nonlinear Differential Equations and Their Applications},
   volume={87},
   publisher={Birkh$\ddot{a}$user},
   date={2015},
}

\bib{Show97}{book}{
   author={Showalter, R. E.},
   title={Monotone operators in Banach space and nonlinear partial
   differential equations},
   series={Mathematical Surveys and Monographs},
   volume={49},
   publisher={American Mathematical Society, Providence, RI},
   date={1997},
}

\bib{Sim87}{article}{
   author={Simon, J.},
   title={Compact sets in the space $L^p(0,T;B)$},
   journal={Ann. Mat. Pura Appl. (4)},
   volume={146},
   date={1987},
   pages={65--96},
}

\bib{TW12}{article}{
    AUTHOR = {Tao, Y.},
    AUTHOR = {Winkler, M.},
 TITLE = {Global existence and boundedness in a
 {K}eller-{S}egel-{S}tokes model with arbitrary porous medium
 diffusion},
   JOURNAL = {Discrete Contin. Dyn. Syst.},
  FJOURNAL = {Discrete and Continuous Dynamical Systems. Series A},
    VOLUME = {32},
 YEAR = {2012},
    NUMBER = {5},
 PAGES = {1901--1914},
}

\bib{Vaz07}{book}{
   author={V\'{a}zquez, J. L.},
   title={The porous medium equation},
   series={Oxford Mathematical Monographs},
   note={Mathematical theory},
   publisher={The Clarendon Press, Oxford University Press, Oxford},
   date={2007},
}

\bib{V}{book}{
   author={Villani, C.},
   title={Optimal transport},
   series={Grundlehren der Mathematischen Wissenschaften [Fundamental
   Principles of Mathematical Sciences]},
   volume={338},
   note={Old and new},
   publisher={Springer-Verlag, Berlin},
   date={2009},
}

\bib{WX15}{article}{
    AUTHOR = {Wang, Y},
    author={Xiang, Z.},
 TITLE = {Global existence and boundedness in a higher-dimensional
 quasilinear chemotaxis system},
   JOURNAL = {Z. Angew. Math. Phys.},
  FJOURNAL = {Zeitschrift f\"{u}r Angewandte Mathematik und Physik. ZAMP.
 Journal of Applied Mathematics and Physics. Journal de
 Math\'{e}matiques et de Physique Appliqu\'{e}es},
    VOLUME = {66},
 YEAR = {2015},
    NUMBER = {6},
 PAGES = {3159--3179},
}

\bib{W18}{article}{
    AUTHOR = {Winkler, M.},
 TITLE = {Global existence and stabilization in a degenerate
 chemotaxis-{S}tokes system with mildly strong diffusion
 enhancement},
   JOURNAL = {J. Differential Equations},
  FJOURNAL = {Journal of Differential Equations},
    VOLUME = {264},
 YEAR = {2018},
    NUMBER = {10},
 PAGES = {6109--6151},
}

   \end{biblist}
\end{bibdiv}

\end{document}